\documentclass[a4paper]{amsart}

\usepackage{amsmath}
\usepackage{amssymb}
\usepackage{latexsym}
\usepackage[mathscr]{eucal}
\usepackage{mathrsfs}
\usepackage{graphics}
\usepackage{newtxmath}
\usepackage{framed}
\usepackage{ulem}
\usepackage{cancel}
\usepackage{arydshln}
\usepackage{bm}
\usepackage{longtable}
\allowdisplaybreaks
\usepackage{multicol}
\usepackage{type1cm}

\theoremstyle{plain}
\newtheorem{theorem}{\bf Theorem}[section]
\newtheorem{lemma}[theorem]{\bf Lemma}
\newtheorem{proposition}[theorem]{\bf Proposition}

\newtheorem*{claim}{\bf Claim}
\newtheorem*{claim1}{\bf Claim 1}
\newtheorem*{claim2}{\bf Claim 2}

\newcommand{\relmiddle}[1]{\mathrel{}\middle#1\mathrel{}}

\makeatletter
\def\thefootnote{\ifnum\c@footnote>\z@\leavevmode\lower.5ex%
      \hbox{$^{\@arabic\c@footnote)}$}\fi}
\makeatother

\theoremstyle{definition}

\theoremstyle{remark}

\def\Aut{\mbox{\rm {Aut}}}
\def\det{\mbox{\rm {det}}}

\def\diag{\mbox{\rm {diag}}}

\def\ad{\mbox{\rm {ad}}}

\def\Hom{\mbox{\rm {Hom}}}

\def\Iso{\mbox{\rm {Iso}}}

\def\Ker{\mbox{\rm {Ker}}}

\def\tr{\mbox{\rm {tr}}}

\def\ov{\overline}

\def\ti{\tilde}

\def\dsum{\displaystyle \sum}

\def\dfrac#1#2{\displaystyle \frac{#1}{#2}}

\def\C{\mbox{\boldmath $C$}}

\def\H{\mbox{\boldmath $H$}}

\def\O{\mbox{\boldmath $O$}}

\def\R{\mbox{\boldmath $R$}}

\def\Z{\mbox{\boldmath $Z$}}

\def\sR{\mbox{\boldmath $\scriptstyle{R}$}}
\def\sC{\mbox{\boldmath $\scriptstyle{C}$}}
\def\sH{\mbox{\boldmath $\scriptstyle{H}$}}

\def\0{\mbox{\boldmath {0}}}
\def\1{\mbox{\boldmath {1}}}
\def\2{\mbox{\boldmath {2}}}
\def\3{\mbox{\boldmath {3}}}
\def\4{\mbox{\boldmath {4}}}
\def\5{\mbox{\boldmath {5}}}
\def\6{\mbox{\boldmath {6}}}
\def\7{\mbox{\boldmath {7}}}
\def\8{\mbox{\boldmath {8}}}
\def\9{\mbox{\boldmath {9}}}

\def\a{\mbox{\boldmath $a$}}

\def\tr{\mbox{\rm {tr}}}
\def\ad{\mbox{\rm {ad}}}

\def\det{\mbox{\rm {det}}}
\def\diag{\mbox{\rm {diag}}}
\def\Iso{\mbox{\rm {Iso}}}
\def\Hom{\mbox{\rm {Hom}}}
\def\Ker{\mbox{\rm {Ker}}}
\def\ov{\overline}

\def\dsum{\displaystyle \sum}

\def\sR{\mbox{\boldmath $\scriptstyle{R}$}}
\def\sC{\mbox{\boldmath $\scriptstyle{C}$}}
\def\sH{\mbox{\boldmath $\scriptstyle{H}$}}
\def\dfrac#1#2{\displaystyle \frac{#1}{#2}}

\def\C{\mbox{\boldmath $C$}}

\def\H{\mbox{\boldmath $H$}}

\def\O{\mbox{\boldmath $O$}}

\def\R{\mbox{\boldmath $R$}}

\def\Z{\mbox{\boldmath $Z$}}
\def\sR{\mbox{\boldmath $\scriptstyle{R}$}}

\def\0{\mbox{\boldmath {0}}}
\def\1{\mbox{\boldmath {1}}}
\def\2{\mbox{\boldmath {2}}}
\def\3{\mbox{\boldmath {3}}}
\def\4{\mbox{\boldmath {4}}}
\def\5{\mbox{\boldmath {5}}}
\def\6{\mbox{\boldmath {6}}}
\def\7{\mbox{\boldmath {7}}}
\def\8{\mbox{\boldmath {8}}}
\def\9{\mbox{\boldmath {9}}}
\def\a{\mbox{\boldmath $a$}}

\def\v{\mbox{\boldmath $v$}}

\usepackage{epic,eepic}

\allowdisplaybreaks

\begin{document}

\title[On realizations of the Lie groups $ G_{2,\sH}, F_{4,\sH},
E_{6,\sH}, E_{7,\sH} ,E_{8,\sH} $  ]
{On realizations of the Lie groups $ G_{2,\sH}, F_{4,\sH},
    E_{6,\sH}, E_{7,\sH} ,E_{8,\sH} $, \\
second edition }

\author[Toshikazu Miyashita]{Toshikazu Miyashita}


\begin{abstract}
In order to define the exceptional compact Lie groups $ G_2,F_4,E_6,E_7,E_8 $, we usually use the Cayley algebra $ \mathfrak{C} $ or its complexification $ \mathfrak{C}^C $. In the present article, we consider replacing the Cayley algebra $ \mathfrak{C} $ with the field of quaternion numbers $ \H $ in the definition of the groups above, and these groups are denoted as in title above. Our aim is to determine the structure of these groups. We call “realization”to determine the structure of the groups.
\end{abstract}



\address{1365-3 Bessho onsen      \endgraf
	     Ueda City                \endgraf
	     Nagano Pref. 386-1431    \endgraf
	     Japan}
\email{anarchybin@gmail.com}

\maketitle

\setcounter{section}{0}

\section{Introduction}
The realizations of all exceptional simple Lie groups had been completed by Ichiro Yokota and the members of his school thirty-five years ago. Until now, their study has yielded the valuable and useful results, and those are summarized in \cite{iy0} and \cite{iy11}.

As mentioned in abstract, in order to define the exceptional compact Lie groups $ G_2,F_4,E_6,\allowbreak E_7, E_8 $, we usually use the Cayley algebra $ \mathfrak{C} $ or its complexification $ \mathfrak{C}^C $. We consider replacing $ \mathfrak{C} $ with the field of quaternion numbers $ \H $ in the definition of the  groups above. These groups are denoted by $ G_{2,\sH}, F_{4,\sH},E_{6,\sH},E_{7,\sH},E_{8,\sH} $. In the present article, our aim is to determine the structure of those groups using the results obtained and the ideas of several proofs in the exceptional Lie groups.

Here, we will sketch out the contents of this article. First, we describe the definitions of the groups $ G_{2,\sH}, F_{4,\sH},E_{6,\sH},E_{7,\sH},E_{8,\sH} $. For details of the definition, see the text of the article.
\begin{align*}
G_{2,\sH}&=\left\lbrace \alpha \in \Iso_{\sR}(\H)\relmiddle{|} (\alpha x)(\alpha y)=\alpha(xy)\right\rbrace,
\\
F_{4,\sH}&=\left\lbrace \alpha \in \Iso_{\sR}(\mathfrak{J}(3,\H)) \relmiddle{|} \alpha X \circ \alpha Y=\alpha(X \circ Y) \right\rbrace ,
\\
E_{6,\sH}&=\left\lbrace \alpha \in \Iso_C(\mathfrak{J}(3,\H^C)) \relmiddle{|} \det\,\alpha X=\det\,X, \langle\alpha X,\alpha Y\rangle=\langle X, Y \rangle \right\rbrace ,
\\
E_{7,\sH}&=\left\lbrace \alpha \in \Iso_C((\mathfrak{P}_{\sH})^C) \relmiddle{|} \alpha(P \times Q)\alpha^{-1}=\alpha P \times \alpha Q, \langle\alpha P,\alpha Q\rangle=\langle P, Q \rangle \right\rbrace ,
\\
E_{8,\sH}&=\left\lbrace \alpha \in \Iso_C((\mathfrak{e}_{8,\sH})^C) \relmiddle{|} \alpha[R_1,R_2]=[\alpha R_1,\alpha R_2], \langle\alpha R_1,\alpha R_2 \rangle=\langle R_1, R_2 \rangle \right\rbrace .
\end{align*}
In Section $ 2 $, we describe the definitions of the Cayley algebras $ \mathfrak{C} , \mathfrak{C}^C$, the exceptional Jordan algebras $ \mathfrak{J}(3,\mathfrak{C}), \mathfrak{J}(3,\mathfrak{C}^C) $, the Freudenthal $ C $-vector space $ \mathfrak{P}^C $, the $ 248 $-dimensional $ C $-vector space $ {\mathfrak{e}_8}^C $, the exceptional Lie groups $ G_2,F_4,E_6,E_7,E_8 $ and these complexification, and in addition, we also describe the properties related to algebras, vector spaces or the results related to the exceptional Lie groups.

\noindent In Section $ 3 $, by using the isomorphism $ (G_{2,\sH})^C \cong Sp(1,\H^C) $ which had already been proved by Ichiro Yokota (\cite[Proposition 1.3.1]{iy7}), we will determine the structure of the group $ G_{2,\sH} $.

\noindent In Sections $ 4,5 $, since the structure of the groups $ F_{4,\sH} $ and $ E_{6,\sH} $ had already been proved (\cite[Propositions 2.11.3,3.11.3]{iy0}), these results will be described as theorems only. Moreover, in order to determine the structure of $ E_{8,\sH} $ in final section,
we will determine the root system and the Dynkin diagram of the Lie algebras $ {(\mathfrak{f}_4}^C)^{\varepsilon_1,\varepsilon_2} $ and $ ({\mathfrak{e}_6}^C)^{\varepsilon_1,\varepsilon_2} $ of the groups $ ({F_4}^C)^{\varepsilon_1,\varepsilon_2} $ and $ ({E_6}^C)^{\varepsilon_1,\varepsilon_2} $, respectively.

\noindent In Section $ 6 $, by using the isomorphism $ ({E_7}^C)^{\varepsilon_1,\varepsilon_2} \cong Spin(12,C)$ (\cite[Proposition 1.1.7]{miya1}), we can show the isomorphism $ (E_7)^{\varepsilon_1,\varepsilon_2} \cong Spin(12) $, so that we will determine the structure of the group $ E_{7,\sH} $. Moreover, as in the previous sections, we will determine the root system and the Dynkin diagram of the Lie algebras $ ({\mathfrak{e}_7}^C)^{\varepsilon_1,\varepsilon_2} $ of the group $ ({E_7}^C)^{\varepsilon_1,\varepsilon_2} $.

\noindent In the final section (Section $ 7 $), first the connectedness of the group $ (E_{8,\sH})^C $ is shown, so that we can prove that the connectedness of the group $ E_{8,\sH} $ follows from the connectedness of the group $ (E_{8,\sH})^C $. Subsequently, the connectedness of the group $ ({E_8}^C)^{\varepsilon_1,\varepsilon_2} $ is shown, moreover
we can draw the Dynkin diagram of the Lie algebra $ ({\mathfrak{e}_8}^C)^{\varepsilon_1,\varepsilon_2} $ by using the results of root systems obtained in the previous sections, so that
the type of the group $ {E_8}^C)^{\varepsilon_1,\varepsilon_2} $ as Lie algebras is determined. Hence, together with the result of the center $ z(({E_8}^C)^{\varepsilon_1,\varepsilon_2}) $, we will determine the structure of the group $ ({E_8}^C)^{\varepsilon_1,\varepsilon_2} $. Finally, we will determine the structure of the compact real form $ (E_8)^{\varepsilon_1,\varepsilon_2} $ of $ ({E_8}^C)^{\varepsilon_1,\varepsilon_2} $, and we will determine the structure of the group $ E_{8,\sH} $ by defining a mapping $ f_{{}_8}:(E_8)^{\varepsilon_1,\varepsilon_2} \to E_{8,\sH} $.

Our results are given as follows.
\begin{align*}
    &G_{2,\sH} \cong Sp(1)/\Z_2,
    \\
    &F_{4,\sH} \cong Sp(3)/\Z_2,
    \\
    &E_{6,\sH} \cong SU(6)/\Z_2,
    \\
    &E_{7,\sH} \cong Ss(12)(:=Spin(12)/\Z_2),
    \\
    &E_{8,\sH} \cong E_7/\Z_2.
\end{align*}
These results are in \cite{iy10} except $ G_{2,\sH} \cong Sp(1)/\Z_2 $. However, there exists no proof in there. Thus, since Ichiro Yokota had determined the structure of the groups $ F_{4,\sH} $ and $ E_{6,\sH} $ (\cite{iy0}),
the author have added the proofs for the remaining groups $ G_{2,\sH},E_{7,\sH} $ and $ E_{8,\sH} $ .

We refer the reader to \cite{miya1} and \cite{iy0} for notations.
The author would like to say that one of the features of this  article is to prove the connectedness of various groups by elementary methods.

\section{Preliminaries}

In this section, we will focus on the definitions of the exceptional complex and compact Lie groups of type $ G_2, F_4, E_6, E_7 $ and $ E_8 $, the structure of those Lie algebras and several results related to them. Since the proofs are omitted, if necessary, refer to \cite{iy7}, \cite{iy2}, \cite{iy1},\cite{iy3} for the complex Lie groups and \cite{iy0} for the compact Lie groups.

\subsection{The Lie groups $ {G_2}^C $ and $ G_2 $ of the type  $ G_2 $ }

Let $\mathfrak{C} $ and $ \mathfrak{C}^C $ be the division Cayley algebra and its complexification, where $ \mathfrak{C}=\{e_0 =1, e_1, e_2, e_3, \allowbreak e_4, e_5, e_6, e_7 \}_{\sR} $. In $\mathfrak{C}$ (resp. $ \mathfrak{C}^C $), since the multiplication and the inner product are well known, these are omitted.
\vspace{1mm}

The simply connected complex Lie group $ {G_2}^C $ and the compact Lie group $ G_2 $ are defined by
\begin{align*}
{G_2}^C&=\left\lbrace \alpha \in \Iso_{C}(\mathfrak{C}^C)\relmiddle{|} \alpha(xy)=(\alpha x) (\alpha y) \right\rbrace ,
\\
G_2 &=\left\lbrace \alpha \in \Iso_{\sR}(\mathfrak{C})\relmiddle{|} \alpha(xy)=(\alpha x) (\alpha y) \right\rbrace .
\end{align*}

We define an $ \R $-linear transformation $ \gamma $ of $ \mathfrak{C} $ by
\begin{align*}
  \gamma(m+ne_4)=m-ne_4, \,\, m+ne_4 \in \H \oplus \H e_4=\mathfrak{C}.
\end{align*}
Then we have $ \gamma \in G_2 $ and $ \gamma^2=1 $. Hence $ \gamma $ induces the involutive inner automorphism $ \tilde{\gamma} $ on $ G_2 $: $ \tilde{\gamma}(\alpha)=\gamma \alpha \gamma, \alpha \in G_2 $. Besides, $ \gamma $ is naturally extended to the $ C $-linear transformation of $ \mathfrak{C}^C $.

We need the following well-known result later.

\begin{theorem}[{\cite[Theorem 1.10.1]{iy0}}]\label{theorem 2.1}
 The group $ (G_2)^\gamma $ is isomorphic to the group $ (Sp(1)\times Sp(1))/\Z_2, \Z_2\allowbreak =\{(1,1),(-1,-1)\} ${\rm :} $  (G_2)^\gamma \cong (Sp(1)\times Sp(1))/\Z_2 $.
\end{theorem}

\subsection{The Lie groups $ {F_4}^C $ and $ F_4 $ of the type $ F_4 $}

Let $\mathfrak{J}(3,\mathfrak{C} )$ and $ \mathfrak{J}(3,\mathfrak{C^C}) $) be the exceptional Jordan algebra and its complexification.
In $\mathfrak{J}(3,\mathfrak{C} )$ (recp. $ \mathfrak{J}(3,\mathfrak{C}^C) $), the Jordan multiplication $X \circ Y$, the
inner product $(X,Y)$ and a cross multiplication $X \times Y$, called the Freudenthal multiplication, are defined by
$$
\begin{array}{c}
X \circ Y = \dfrac{1}{2}(XY + YX), \quad (X,Y) = \tr(X \circ Y),
\vspace{1mm}\\
X \times Y = \dfrac{1}{2}(2X \circ Y-\tr(X)Y - \tr(Y)X + (\tr(X)\tr(Y)
- (X, Y))E),
\end{array}
$$
respectively, where $E$ is the $3 \times 3$ unit matrix. Moreover, we define the trilinear form $(X, Y, Z)$, the determinant $\det \,X$ by
$$
(X, Y, Z)=(X, Y \times Z),\quad \det \,X=\dfrac{1}{3}(X, X, X),
$$
respectively, and briefly denote $\mathfrak{J}(3, \mathfrak{C})$ (resp. $ \mathfrak{J}(3,\mathfrak{C}^C) $)
by $\mathfrak{J}^C$ (resp. $ \mathfrak{J} $).

In $\mathfrak{J}$ (resp. $ \mathfrak{J}^C $), we often use the following nations:
\begin{align*}
E_1 &= \begin{pmatrix}
1 & 0 & 0 \\
0 & 0 & 0 \\
0 & 0 & 0
\end{pmatrix},  \,\,\,\,\,\,\,\,
E_2 = \begin{pmatrix}
0 & 0 & 0 \\
0 & 1 & 0 \\
0 & 0 & 0
\end{pmatrix},  \,\,\,\,\,\,\,\,\,\,
E_3 =\begin{pmatrix}
0 & 0 & 0 \\
0 & 0 & 0 \\
0 & 0 & 1
\end{pmatrix},
\\[2mm]
F_1 (x) &= \begin{pmatrix}
0 & 0 & 0 \\
0 & 0 & x \\
0 & \ov{x} & 0
\end{pmatrix},  \,\,
F_2(x) = \begin{pmatrix}
0 & 0 & \ov{x} \\
0 & 0 & 0 \\
x & 0 & 0
\end{pmatrix},  \,\,
F_3 (x) = \begin{pmatrix}
0 & x & 0 \\
\ov{x} & 0 & 0 \\
0 & 0 & 0
\end{pmatrix}.
\end{align*}

The simply connected complex Lie group ${F_4}^C$ and the compact Lie group $ F_4 $ are defined by
\begin{align*}
{F_4}^C
&= \left\lbrace \alpha \in \Iso_C(\mathfrak{J}^C) \relmiddle{|} \alpha(X \circ Y) = \alpha X \circ \alpha Y \right\rbrace ,
\\[1mm]
F_4
&= \left\lbrace \alpha \in \Iso_{\sR}(\mathfrak{J}) \relmiddle{|} \alpha(X \circ Y) = \alpha X \circ \alpha Y \right\rbrace .
 \end{align*}

The Lie algebra ${\mathfrak{f}_4}^C$
of the group ${F_4}^C$
is given by
\begin{align*}
{\mathfrak{f}_4}^C=\{\delta \in \mathrm{Hom}_{C}({\mathfrak{J}}^C)\,|\, \delta(X \circ Y)=\delta X \circ Y + X \circ \delta Y  \},
\end{align*}
and any element $\delta$ of the Lie algebras ${\mathfrak{f}_4}^C$ can be uniquely expressed by
\begin{align*}
\delta =(D_1,D_2,D_3) +\ti{A}_1(a_1) +\ti{A}_2(a_2) +\ti{A}_3(a_3), D_i \in \mathfrak{so}(8,C), a_k \in \mathfrak{C}^C,
\end{align*}
where $D_2, D_3 \in \mathfrak{so}(8,C)$ are uniquely determined by the Principle of triality on $ \mathfrak{so}(8,C) $: $(D_1 x)y+x(D_2 y)=\overline{D_3 (\overline{xy})},\, x,y \in \mathfrak{C}^C$ for a given
$D_1 \in \mathfrak{so}(8,C)$ and $\ti{A}_k(a_k)$ is the $C$-linear mapping of $\mathfrak{J}^C$.

\subsection{The Lie groups $ {E_6}^C $ and $ E_6 $ of the type $ E_6 $}

The simply connected complex Lie group ${E_6}^C$ and the compact Lie group $ E_6 $ are defined by
\begin{align*}
{E_6}^C &= \left\lbrace \alpha \in \Iso_C(\mathfrak{J}^C) \relmiddle{|} \det \, \alpha X = \det \, X \right\rbrace ,
\\[1mm]
E_6 &= \left\lbrace \alpha \in \Iso_C(\mathfrak{J}^C) \relmiddle{|} \det \, \alpha X = \det \, X, \langle \alpha X, \alpha Y \rangle=\langle X, Y \rangle \right\rbrace ,
\end{align*}
where the Hermite inner product $ \langle X, Y \rangle $ is defined by $ (\tau X,Y) $ using $ \tau $ is the complex conjugation in $ \mathfrak{J}^C $: $ \langle X, Y \rangle=(\tau X,Y) $.

Then we have naturally the inclusion ${F_4}^C \hookrightarrow {E_6}^C$ and $ ({E_6}^C)_E = {F_4}^C$ (resp. $ F_4 \hookrightarrow E_6$ and $ (E_6)_E=F_4 $).

The Lie algebra ${\mathfrak{e}_6}^C$ of the group ${E_6}^C$
is given by
\begin{align*}
{\mathfrak{e}_6}^C=\{\phi \in \mathrm{Hom}_{C}({\mathfrak{J}}^C)\,|\, (\phi X,X,X)=0 \},
\end{align*}
and any element $ \phi $ of Lie algebras $ {\mathfrak{e}_6}^C$ can be uniquely expressed by
\begin{align*}
\phi=\delta+\tilde{T},\,\,\delta \in {\mathfrak{f}_4}^C, \, T \in (\mathfrak{J}^C)_0,
\end{align*}
where $(\mathfrak{J}^C)_0=\{X \in \mathfrak{J}^C \,|\,\tr(X)=0 \}$ and
the $C$-linear mapping $\tilde{T}$ of $\mathfrak{J}^C$ is defined by $\tilde{T}X=T \circ X, X \in \mathfrak{J}^C$.

\subsection{The Lie groups $ {E_7}^C $ and $ E_7 $ of the type $ E_7 $}

Let $\mathfrak{P}^C$ be the $56$-dimensional Freudenthal $C$-vector space
$$
\mathfrak{P}^C = \mathfrak{J}^C \oplus \mathfrak{J}^C \oplus C \oplus C,
$$
with the inner product $ (P,Q) $, the Hermite inner product $ \langle P,Q \rangle $ and the skew-symmetric inner product $ \{ P,Q \} $ which are defined by
\begin{align*}
(P,Q):&=(X,Z)+(Y,W)+\xi\zeta+\eta\omega
\\
\langle P,Q \rangle:&=\langle X,Z \rangle+\langle Y,W \rangle+(\tau \xi)+(\tau \eta)\omega,
\\
\{P, Q\}:&=(X,W)-(Y,Z)+\xi\omega-\eta\zeta,
\end{align*}
respectively, where $ P=(X,Y,\xi,\eta), Q=(Z,W,\zeta,\omega) \in
\mathfrak{P}^C $. Note that $ \langle P,Q \rangle=(\tau P,Q)=\{\tau\lambda
P,Q \} $.

In $ \mathfrak{P}^C $, we often use the following notations:
\begin{align*}
\dot{X}:=(X,0,0,0),\,\,\underset{\dot{}}{Y}:=(0,Y,0,0),\,\dot{1}:=(0,0,1,0),\,\,\underset{\dot{}}{1}:=(0,0,0,1).
\end{align*}

For $\phi \in {\mathfrak{e}_6}^C$, $A, B \in \mathfrak{J}^C$ and $\nu \in C$, we define a $C$-linear mapping  $\varPhi(\phi, A, B, \nu) : \mathfrak{P}^C \to \mathfrak{P}^C$ by
$$
\varPhi(\phi, A, B, \nu) \begin{pmatrix}X \vspace{1.5mm}\\
Y \vspace{1.5mm}\\
\xi \vspace{.5mm}\\
\eta
\end{pmatrix}
=  \begin{pmatrix}\phi X - \dfrac{1}{3}\nu X + 2B
\times Y + \eta A \vspace{-0.5mm}\\
2A \times X - {}^t\phi Y + \dfrac{1}{3}\nu Y +
\xi B \vspace{1mm}\\
(A, Y) + \nu\xi \vspace{1.5mm}\\
(B, X) - \nu\eta
\end{pmatrix},
$$
where ${}^t\phi \in {\mathfrak{e}_6}^C$ is the transpose of $\phi$ with respect to the
inner product $(X, Y)$: $({}^t\phi X, Y) = (X, \phi Y)$.

For $ P = (X, Y, \xi, \eta), Q = (Z, W, \zeta, \omega) \in \mathfrak{P}^C $, we define a $C$-linear
mapping $P \times Q: \mathfrak{P}^C \to \mathfrak{P}^C$, called the Freudenthal cross operation, by
$$
P \times Q := \varPhi(\phi, A, B, \nu), \quad
\left \{ \begin{array}{l}
\vspace{1mm}
\phi = - \dfrac{1}{2}(X \vee W + Z \vee Y) \\
\vspace{1mm}
A = - \dfrac{1}{4}(2Y \times W - \xi Z - \zeta X) \\
\vspace{1mm}
B =  \dfrac{1}{4}(2X \times Z - \eta W - \omega Y) \\
\vspace{1mm}
\nu = \dfrac{1}{8}((X, W) + (Z, Y) - 3(\xi\omega + \zeta\eta)),
\end{array} \right.
$$
\noindent where $X \vee W \in {\mathfrak{e}_6}^C$ is defined by
$$
X \vee W = [\tilde{X}, \tilde{W}] + (X \circ W - \dfrac{1}{3}(X, W)E)^{\sim},
$$
\noindent  here the $C$-linear mappings $\tilde{X}, \tilde{W}$ of $\mathfrak{J}^C$ are the same ones as that in ${E_6}^C$.
\vspace{1mm}

The simply connected complex Lie group ${E_7}^C$ and the compact Lie $ E_7 $ are defined by
\begin{align*}
{E_7}^C &=  \left\lbrace \alpha \in \Iso_C(\mathfrak{P}^C) \relmiddle{|}\alpha(P \times Q)\alpha^{-1}=\alpha P \times \alpha Q \right\rbrace ,
\\[1mm]
E_7 &=  \left\lbrace \alpha \in \Iso_C(\mathfrak{P}^C) \relmiddle{|} \alpha(P \times Q)\alpha^{-1}=\alpha P \times \alpha Q, \langle \alpha P. \alpha Q \rangle=\langle P,Q \rangle \right\rbrace .
\end{align*}

The Lie algebra ${\mathfrak{e}_7}^C$ of the group ${E_7}^C$ and the Lie algebra $ \mathfrak{e}_7 $ of the group $ E_7 $ are given by
\begin{align*}
{\mathfrak{e}_7}^C &= \left\lbrace \varPhi(\phi, A, B, \nu) \in \Hom(\mathfrak{P}^C) \relmiddle{|} \phi \in {\mathfrak{e}_6}^C, A, B \in \mathfrak{J}^C, \nu \in C \right\rbrace ,
\\[1mm]
\mathfrak{e}_7 &= \left\lbrace \varPhi(\phi, A, -\tau A, \nu) \in \Hom(\mathfrak{P}^C) \relmiddle{|} \phi \in \mathfrak{e}_6, A \in \mathfrak{J}^C, \nu \in i\R \right\rbrace .
\end{align*}
For $\alpha \in {E_6}^C$, the mapping $\ti{\alpha}: \mathfrak{P}^C \to \mathfrak{P}^C$ is defined by
$$
\ti{\alpha}(X, Y, \xi, \eta)=(\alpha X, {}^t \alpha^{-1}Y, \xi, \eta),
$$
then we have $\ti{\alpha} \in {E_7}^C$, and so $\alpha$ and $\ti{\alpha}$ will be identified. The group ${E_7}^C$ contains ${E_6}^C$ as a subgroup by
$$
{E_6}^C=({E_7}^C)_{\dot{1}, \underset{\dot{}}{1}}(:=\{ \alpha \in {E_7}^C \,|\, \alpha \dot{1}=\dot{1}, \alpha \underset{\dot{}}{1}=\underset{\dot{}}{1} \}\,).
$$
Hence we have the inclusion $ {F_4}^C \hookrightarrow {E_6}^C \hookrightarrow {E_7}^C$.
For $ \alpha \in E_6 $, suppose $ \alpha \underset{\dot{}}{1}=\underset{\dot{}}{1} $, we have $ \alpha \dot{1}=\dot{1} $, so $ E_7 $ contains $ E_6 $ as a subgroup by
\begin{align*}
  E_6=(E_7)_{\text{\d{$ 1 $}}}.
\end{align*}
Hence we have the inclusion $ F_4 \hookrightarrow E_6 \hookrightarrow E_7 $.
Finally, note that $ \alpha \in {E_7}^C $ leaves the skew-symmetric inner product invariant: $ \{\alpha P, \alpha Q \}=\{P,Q \}, P,Q, \in \mathfrak{P}^C $.

\subsection{The Lie group $ {E_8}^C $ and $ E_8 $ of the type $ E_8 $ }

Let ${\mathfrak{e}_8}^C$ be the $248$-dimensional $C$-vector space
$$
{\mathfrak{e}_8}^C = {\mathfrak{e}_7}^C \oplus \mathfrak{P}^C \oplus \mathfrak{P}^C \oplus C \oplus C \oplus C.
$$
We define a Lie bracket $[R_1, R_2], R_1\!=\!(\varPhi_1, P_1, Q_1, r_1, s_1, t_1), R_2\!=\!(\varPhi_2, P_2, Q_2, r_2, s_2, t_2)$, by
\begin{align*}
[R_1, R_2]:= (\varPhi, P, Q, r, s, t),\,\,\left\{\begin{array}{l}
\varPhi = {[}\varPhi_1, \varPhi_2] + P_1 \times Q_2 - P_2 \times Q_1
\vspace{1mm} \\
P = \varPhi_1P_2 - \varPhi_2P_1 + r_1P_2 - r_2P_1 + s_1Q_2 - s_2Q_1
\vspace{1mm} \\
Q = \varPhi_1Q_2 - \varPhi_2Q_1 - r_1Q_2 + r_2Q_1 + t_1P_2 - t_2P_1
\vspace{1mm} \\
r = - \dfrac{1}{8}\{P_1, Q_2\} + \dfrac{1}{8}\{P_2, Q_1\} + s_1t_2 - s_2t_1\vspace{1mm} \\
s = \,\,\, \dfrac{1}{4}\{P_1, P_2\} + 2r_1s_2 - 2r_2s_1
\vspace{1mm} \\
t = - \dfrac{1}{4}\{Q_1, Q_2\} - 2r_1t_2 + 2r_2t_1.
\end{array} \right.
\end{align*}
Then the $C$-vector space ${\mathfrak{e}_8}^C$ becomes a complex simple Lie algebra of type $E_8$.

In ${\mathfrak{e}_8}^C$, we often use the following notations:
$$
\begin{array}{c}
\varPhi = (\varPhi, 0, 0, 0, 0, 0), \quad P^- = (0, P, 0, 0, 0, 0), \quad Q_- = (0, 0, Q, 0, 0, 0),
\vspace{1mm}\\
\tilde{r} = (0, 0, 0, r, 0, 0), \quad s^- = (0, 0, 0, 0, s, 0), \quad t_- = (0, 0, 0, 0, 0, t).
\end{array}
$$

We define a $C$-linear transformation $\lambda_\omega$ of ${\mathfrak{e}_8}^C$
by
$$
\lambda_\omega(\varPhi, P, Q, r, s, t) = (\lambda\varPhi\lambda^{-1}, \lambda Q, - \lambda P, -r, -t, -s),
$$
where a $C$-linear transformation $\lambda$ of $\mathfrak{P}^C$ on the right-hand side is defined by
$ \lambda(X, Y, \xi, \eta)$ $ = (Y, - X, \eta,\allowbreak -\xi)$.
As in $\mathfrak{J}^C$, the complex conjugation in ${\mathfrak{e}_8}^C$ is denoted by $\tau$:
$$
\tau(\varPhi, P, Q, r, s, t) = (\tau\varPhi\tau, \tau P, \tau Q, \tau r, \tau s, \tau t).
$$

The simply connected complex Lie group ${E_8}^C$ and the compact Lie group $E_8$ are defined by
\begin{align*}
{E_8}^C &=\Aut({\mathfrak{e}_8}^C)= \left\lbrace \alpha \in \Iso_C({\mathfrak{e}_8}^C) \relmiddle{|} \alpha[R, R'] = [\alpha R, \alpha R']\right\rbrace ,
\\[1mm]
E_8 &= \left\lbrace \alpha \in {E_8}^C \relmiddle{|} \langle \alpha R_1, \alpha R_2 \rangle=\langle R_1,R_2 \rangle \right\rbrace  =({E_8}^C)^{\tau\lambda_\omega },
\end{align*}
where the Hermite inner product $ \langle R_1,R_2 \rangle $ is defined by $ (-1/15)B_8(\tau\lambda_\omega R_1,R_2) $ using the Killing form $ B_8 $ of $ {\mathfrak{e}_8}^C $: $ \langle R_1,R_2 \rangle=(-1/15)B_8(\tau\lambda_\omega R_1,R_2) $.

The Lie algebra $ {\mathfrak{e}_8}^C $ of the group $ {E_8}^C $ and he Lie algebra $\mathfrak{e}_8$ of the group $E_8$ are given by
\begin{align*}
{\mathfrak{e}_8}^C&=\left\lbrace (\varPhi, P, Q, r, s, t) \relmiddle{|}\varPhi \in {\mathfrak{e}_7}^C, P,Q \in \mathfrak{J}^C, r,s,t \in C \right\rbrace ,
\\[1mm]
\mathfrak{e}_8&=\left\lbrace (\varPhi, P, -\tau\lambda P, r, s, -\tau s)\relmiddle{|}\varPhi \in \mathfrak{e}_7, P \in \mathfrak{J}^C, r \in i\R, s \in C \right\rbrace .
\end{align*}

For $\alpha \in {E_7}^C$, the mapping $\tilde{\alpha} : {\mathfrak{e}_8}^C \to {\mathfrak{e}_8}^C$ is defined by
$$
\tilde{\alpha}(\varPhi, P, Q, r, s, t) = (\alpha\varPhi\alpha^{-1}, \alpha P, \alpha Q, r, s, t),
$$
then we have $\tilde{\alpha} \in {E_8}^C$, and so $\alpha$ and $\tilde{\alpha}$ will be identified. The group ${E_8}^C$ contains ${E_7}^C$ as a subgroup by
$$       {E_7}^C
=({E_8}^C)_{\tilde{1},1^-,1_-}\left( =\left\lbrace  \alpha \in ({E_8}^C) \relmiddle{|}\alpha \tilde{1}=\tilde{1},\alpha 1^-=1^-,\alpha 1_-=1_-\right\rbrace \right) .
$$
Hence we have the inclusion ${F_4}^C \hookrightarrow {E_6}^C \hookrightarrow {E_7}^C \hookrightarrow {E_8}^C$. For $ \alpha \in E_8 $, suppose $ \alpha 1_-=1_- $, we have $ \alpha \tilde{1}=\tilde{1} $ and $ \alpha 1^-=1^- $, so $ E_8 $ contains $ E_7 $ as a subgroup by
\begin{align*}
E_7=(E_8)_{1_-}.
\end{align*}
Hence we have the inclusion $F_4 \hookrightarrow E_6 \hookrightarrow E_7 \hookrightarrow E_8$.

\subsection{Useful lemma}

\begin{lemma}\label{lemma 3.6.1}
	For Lie groups $G, G' $,  let a mapping $\varphi : G \to G'$ be a homomorphism of Lie groups. When $G'$ is connected, if $\Ker\,\varphi$ is discrete and $\dim(\mathfrak{g})=\dim(\mathfrak{g}')$, $\varphi$ is surjective.
\end{lemma}
\begin{proof}
	The proof is omitted (cf. \cite[Lemma 0.6 (2)]{iy7}).
\end{proof}

\begin{lemma}[E. Cartan-Ra\v{s}evskii]\label{lemma 3.6.2}
	Let $G$ be a simply connected Lie group with a finite order automorphism $\sigma$
	of $G$. Then $G^\sigma$ is connected.
\end{lemma}
\begin{proof}
	The proof is omitted (cf. \cite[Lemma 0.7]{iy7}).
\end{proof}
After this, using these lemmas without permission each times, we often prove lemma, proposition or theorem.

\section{The group $ G_{2,\sH} $ }

Let $ \H $ and $ \H^C $ be the field of quaternion number
and its complexification. We consider the following groups
\begin{align*}
    (G_{2,\sH})^C&:=\left\lbrace \alpha \in \Iso_{C}(\H^C) \relmiddle{|}
    \alpha(xy)=(\alpha x)(\alpha y)\right\rbrace,
\\
     G_{2,\sH}&:=\left\lbrace \alpha \in \Iso_{\sR}(\H) \relmiddle{|}\alpha(xy)=(\alpha x)(\alpha y)\right\rbrace,
\end{align*}

\noindent In addition, we define a subgroup $ ((G_{2,\sH})^C)^\tau $ of $
(G_{2,\sH})^C $ by
\begin{align*}
   ((G_{2,\sH})^C)^\tau:=\left\lbrace\alpha \in (G_{2,\sH})^C \relmiddle{|}
   \tau\alpha=\alpha\tau \right\rbrace.
\end{align*}

Then we have the following proposition.

\begin{lemma}\label{lemma 4.1}
  The group $ ((G_{2,\sH})^C)^\tau $ coincides with the group $ G_{2,\sH} $
  {\rm :} $ ((G_{2,\sH})^C)^\tau=G_{2,\sH} $.
\end{lemma}
\begin{proof}
    Let $ \alpha \in G_{2,\sH} $. Then we define an action of $ \alpha $ to
    $ x:=x_1+ix_2 \in \H^C, x_i \in \H $ by
    \begin{align*}
        \alpha x=\alpha(x_1+ix_2)=\alpha x_1+i\alpha x_2 \in \H^C.
    \end{align*}
Hence, for $ x=x_1+ix_2,y=y_1+iy_2 \in \H^C, x_i,y_i \H $, it follows that
\begin{align*}
    (\alpha x)(\alpha y)&=\alpha(x_1+ix_2)\alpha(y_1+iy_2)=(\alpha
    x_1+i\alpha x_2)(\alpha y_1+i\alpha y_2)
    \\
    &=((\alpha x_1)(\alpha y_1)-(\alpha x_2)(\alpha y_2))+i((\alpha x_1)
    (\alpha y_2)+(\alpha x_2)(\alpha y_1))
    \\
    &=(\alpha(x_1y_1)-\alpha(x_2y_2))+i(\alpha(x_1y_2)+\alpha(x_2y_1))
    \\
    &=\alpha(x_1y_1-x_2y_2)+i\alpha(x_1y_2+x_2y_1)
    \\
    &=\alpha(x_1y_1-x_2y_2+i(x_1y_2+x_2y_1))
    \\
    &=\alpha((x_1+ix_2)(y_1+iy_2))
    \\
    &=\alpha(xy).
\end{align*}
Hence we have $ \alpha \in (G_{2,\sH})^C $. Moreover, it is easy to verify
that $ \tau\alpha=\alpha\tau $. Indeed, it follows that
\begin{align*}
    \tau\alpha x&=\tau\alpha(x_1+ix_2)=\tau(\alpha x_1+i\alpha
    x_2)=\alpha x_1-i\alpha x_2
    \\
    &=\alpha(x_1-ix_2)=\alpha\tau(x_1+ix_2)=\alpha\tau x,\,\,x=x_1+ix_2 \in
    \H^C,
\end{align*}
that is, $ \tau\alpha=\alpha\tau $. Thus we have $ \alpha \in ((G_{2,
\sH})^C)^\tau $.

Conversely, let $ \beta \in ((G_{2,\sH})^C)^\tau $. From $
\tau\beta=\beta\tau $, $ \beta $ induces an $ \R $-isomorphism of $ \H $.
Indeed, note that $ \tau y=y, y \in \H $, it follows that
\begin{align*}
    \tau(\beta y)=\beta\tau y=\beta y,
\end{align*}
so that $ \beta y \in \H $. Thus we have $ \beta \in G_{2,\sH} $. With
above, the proof of this lemma is completed.
\end{proof}

Here, in order to determine the structure of the group $ G_{2,\sH} $, we
need the following result.

\begin{proposition}{\rm (\cite[Proposition 1.3.1]{iy7})}
\label{proposition 4.2}
 The group $ (G_{2,\sH})^C $ is isomorphic to the group $ Sp(1,\H^C)/\Z_2,
 \allowbreak \Z_2=\{1,-1\}${\rm :} $ (G_{2,\sH})^C \cong Sp(1,\H^C)/\Z_2 $.
\end{proposition}

Now, we determine the structure of the group $ G_{2,\sH} $.

\begin{theorem}\label{theorem 4.3}
    The group $ G_{2,\sH} $ is isomorphic to the group $ Sp(1)/
    \Z_2, \Z_2=\{1,-1\} ${\rm :} $ G_{2,\sH} \cong Sp(1)/\Z_2 $.
\end{theorem}
\begin{proof}
 Let $ Sp(1)=\{q \in \H \,|\, q\ov{q}=1\} \subset Sp(1,\H^C)$. We
 define a mapping $ \varphi_{{}_{2,\scalebox{0.8}{$\sH$}}}: Sp(1) \to G_{2,
 \sH} $ by
 \begin{align*}
     \varphi_{{}_{2,\scalebox{0.8}{$\sH$}}}(q)m=qm\ov{q},\;\; m \in \H.
 \end{align*}
First, we can easily confirm that $ \varphi_{{}_{2,\scalebox{0.8}{$\sH$}}}
$ is well-defined. Indeed, it follows that
\begin{align*}
   (\varphi_{{}_{2,\scalebox{0.8}{$\sH$}}}(q)m)(\varphi_{{}_{2,
   \scalebox{0.8}{$\sH$}}}(q)n)=(qm\ov{q})(qn\ov{q})=q(mn)\ov{q}
   =\varphi_{{}_{2,\scalebox{0.8}{$\sH$}}}(q)(mn).
\end{align*}
Moreover, since this mapping is the restriction of the mapping $
\varphi_{{}_{2,\scalebox{0.8}{$\sH^C$}}}: Sp(1,\H^C) \to (G_{2,\sH})^C $,
it is clear that $ \varphi_{{}_{2,\scalebox{0.8}{$\sH$}}} $ is a
homomorphism.

Next, we will prove that $ \varphi_{{}_{2,\scalebox{0.8}{$\sH$}}} $ is
surjective. Let $ \alpha \in G_{2,\sH} \subset (G_{2,\sH})^C $. There
exists $ q \in Sp(1,\H^C) $ such that $ \alpha=\varphi_{{}_{2,
\scalebox{0.8}{$\sH^C$}}}(q) $ (Proposition \ref{proposition 4.2}). In
addition, $ \alpha $ satisfies the condition $ \tau\alpha=\alpha\tau
$ (Lemma \ref{lemma 4.1}), that is, $ \tau\varphi_{{}_{2,\scalebox{0.8}{$
\sH^C$}}}(q)
\tau=\varphi_{{}_{2,\scalebox{0.8}{$\sH^C$}}}(q) $, so that we have $
\varphi_{{}_{2,\scalebox{0.8}{$\sH^C$}}}(\tau q)=\varphi_{{}_{2,
\scalebox{0.8}{$\sH^C$}}}(q) $.
Hence we  have the following
\begin{align*}
    \tau q=q     \qquad \text{or} \qquad \tau q=-q.
\end{align*}
In the former case, we obtain $ q \in Sp(1) $. In the latter case, we have
$ q=iq', q' \in \H $. However, this case is impossible because of $
1=q\ov{q}=(iq')(i\ov{q'})=-q'\ov{q'} \leq 0 $. The proof of surjective is
completed.

Finally, we will determine $ \Ker\,\varphi_{{}_{2,\scalebox{0.8}{$\sH$}}}
$. Since $ \varphi_{{}_{2,\scalebox{0.8}{$\sH$}}} $ is the restriction of
the mapping $ \varphi_{{}_{2,\scalebox{0.8}{$\sH^C$}}} $, it is clear $
\Ker\,\varphi_{{}_{2,\scalebox{0.8}{$\sH$}}}=\Ker\,\varphi_{{}_{2,
\scalebox{0.8}{$\sH^C$}}}=\{1,-1 \} \cong \Z_2 $.

Therefore we have the required isomorphism
\begin{align*}
    G_{2,\sH} \cong Sp(1)/\Z_2.
\end{align*}
\end{proof}

After this, in all sections, we will determine the root system
and the Dynkin diagram of some Lie algebras, where some Lie algebras are defined later. Its final aim is to determine the type of the complex
Lie group $ (E_{8,\sH})^C $ as Lie algebras,

\section{The group $ F_{4,\sH} $ and the root system, the Dynkin
diagram of the Lie algebra $ ({\mathfrak{f}_{4}}^C)^{\varepsilon_1,
\varepsilon_2} $}

Let $\mathfrak{J}(3,\mathfrak{C})  $ be the exceptional Jordan algebra.
We consider an algebra $ \mathfrak{J}(3,\H) $ which are defined by
replacing $ \mathfrak{C} $ with $ \H $ in $ \mathfrak{J}(3,
\mathfrak{C}) $, and we briefly denote $ \mathfrak{J}(3,\H) $ by $
\mathfrak{J}_{\sH} $.

Then we consider the following group $ F_{4,\sH} $ which is defined by
replacing $ \mathfrak{J} $ with $ \mathfrak{J}_{\sH} $ in $ F_4 $:
\begin{align*}
F_{4,\sH}&=\left\lbrace \alpha \in \Iso_{\sR}(\mathfrak{J}_{\sH})
\relmiddle{|} \alpha(X \circ Y)=\alpha X \circ \alpha Y\right\rbrace
\\
&=\left\lbrace \alpha \in \Iso_{\sR}(\mathfrak{J}_{\sH})\relmiddle{|}
\alpha(X \times Y)=\alpha X \times \alpha Y\right\rbrace.
\end{align*}

This group has been investigated by Ichiro Yokota, so that the following
result was obtained.

\begin{theorem}{\rm (\cite[Proposition 2.11.3]{iy0})}\label{theorem 5.1}
   The group $ F_{4,\sH} $ is isomorphic to the group $ Sp(3)/\Z_2, \allowbreak \Z_2=
   \{E,-E\} ${\rm :} $ F_{4,\sH} \cong Sp(3)/\Z_2 $.
\end{theorem}

Now, we move the determination of the root system and the Dynkin diagram of the Lie algebra $ ({\mathfrak{f}_{4}}^C)^{\varepsilon_1,\varepsilon_2} $.

We arrange $ \R $-linear transformations of $ \mathfrak{C} $ below. By using the mapping $ \varphi_{{}_{G_2}}:Sp(1) \times Sp(1) \to (G_2)^\gamma $,
\begin{align*}
    \varphi_{{}_{G_2}}(p,q)(m+ne_4)=qm\ov{q}+(pn\ov{q})e_4,\;m+ne_4 \in \H \oplus \H e_4=\mathfrak{C}
\end{align*}
(Theorem \ref{theorem 2.1}), we define $ \R $-linear transformations $ \varepsilon_1,\varepsilon_2 $ of $ \mathfrak{C} $ by
\begin{align*}
    \varepsilon_1:=\varphi_{{}_{G_2}}(e_1,1),\quad
    \varepsilon_2:=\varphi_{{}_{G_2}}(e_2,1),
\end{align*}
where $ e_1,e_2 $ are two of basis in $ \H \subset \mathfrak{C} $.
Besides, the $ \R $-linear transformation $ \gamma $ of $ \mathfrak{C} $ defined in previous section is expressed by
\begin{align*}
\gamma=\varphi_{{}_{G_2}}(-1,1).
\end{align*}

These $ \R $-linear transformations of $ \mathfrak{C} $ are naturally extended to $ \R $-linear transformations of $ \mathfrak{J} $ as follows:
\begin{align*}
    \varepsilon_i
    \begin{pmatrix}
    \xi_1 & x_3 & \ov{x}_2 \\
    \ov{x}_3 & \xi_2 & x_1 \\
    x_2 & \ov{x}_1 & \xi_3
    \end{pmatrix}
= \begin{pmatrix}
    \xi_1 & \;\varepsilon_i x_3\; & \ov{\varepsilon_i x}_2 \\
    \ov{\varepsilon_i x}_3 & \xi_2 & \varepsilon_i x_1 \\
    \varepsilon_i x_2 & \ov{\varepsilon_i x}_1 & \xi_3
\end{pmatrix},i=1,2, \;
   \gamma
   \begin{pmatrix}
        \xi_1 & x_3 & \ov{x}_2 \\
       \ov{x}_3 & \xi_2 & x_1 \\
       x_2 & \ov{x}_1 & \xi_3
   \end{pmatrix}
   = \begin{pmatrix}
       \xi_1 & \gamma  x_3 & \ov{\gamma  x}_2 \\
       \ov{\gamma  x}_3 & \xi_2 & \gamma  x_1 \\
       \gamma  x_2 & \ov{\gamma  x}_1 & \xi_3
   \end{pmatrix},
\end{align*}
moreover these are also extended to $ C $-linear transformations of $ \mathfrak{J}^C $.
Then we have $ \varepsilon_1,\varepsilon_2,\gamma \in G_2 \subset F_4 \subset {F_4}^C $ and $ {\varepsilon_1}^2={\varepsilon_2}^2=\gamma, {\varepsilon_1}^4={\varepsilon_2}^4=1, \gamma^2=1$.

Now, we consider a subgroup $ ({F_4}^C)^{\varepsilon_1,\varepsilon_2} $ of $ {F_4}^C $ by
\begin{align*}
   ({F_4}^C)^{\varepsilon_1,\varepsilon_2}:=\left\lbrace \alpha \in {F_4}^C \relmiddle{|} \varepsilon_1\alpha=\alpha\varepsilon_1, \varepsilon_2\alpha=\alpha\varepsilon_2 \right\rbrace.
\end{align*}
Then the group $ ({F_4}^C)^{\varepsilon_1,\varepsilon_2} $ is isomorphic to the group $ Sp(3,\H^C) $: $ ({F_4}^C)^{\varepsilon_1,\varepsilon_2} \cong Sp(3,\H^C) $, however since this group does not use directly in this article, the proof of isomorphism is omitted.

Now, we move the determination of the root system and the Dynkin diagram of the Lie algebra $ ({\mathfrak{f}_{4}}^C)^{\varepsilon_1,\varepsilon_2} $ of the group $ ({F_4}^C)^{\varepsilon_1,\varepsilon_2} $.

We prove the following lemma needed later.

\begin{lemma}\label{lemma 5.2}
   The Lie algebra $ ({\mathfrak{f}_4}^C)^{\varepsilon_1,\varepsilon_2} $  of the group $ ({F_4}^C)^{\varepsilon_1,\varepsilon_2} $ is given by
   \begin{align*}
     ({\mathfrak{f}_4}^C)^{\varepsilon_1,\varepsilon_2}
     =\left\lbrace
     \begin{array}{l}
     \delta=(D_1,D_2,D_3) \vspace{1mm}\\
     \quad +\tilde{A}_1(a_1)+\tilde{A}_2(a_2)+\tilde{A}_3(a_3) \in {\mathfrak{f}_4}^C
    \end{array}
     \relmiddle{|}
     \begin{array}{l}
      D_1=d_{01}G_{01}+d_{02}G_{02}+d_{03}G_{03} \\
      \quad +d_{12}G_{12}+d_{13}G_{13}+d_{23}G_{23} \\
      \quad +d_{45}(G_{45}+G_{67})+d_{46}(G_{46}-G_{57}) \\
      \quad +d_{47}(G_{47}+G_{56}),d_{ij} \in C, \\
      D_2=\pi\kappa D_1, D_3=\kappa\pi D_1,\\
      a_k \in \H^C
     \end{array}
     \right\rbrace.
   \end{align*}

In particular, we have $ \dim_C( ({\mathfrak{f}_4}^C)^{\varepsilon_1,\varepsilon_2})=9+4 \times 3=21. $
\end{lemma}
\begin{proof}
    First, $ \varepsilon_1, \varepsilon_2 $ are expressed as elements of the group $ SO(8) $ as follows:
    \begin{align*}
       \varepsilon_1=\begin{pmatrix}1 & 0 & 0 & 0 & 0 & 0 & 0 & 0\\
          0 & 1 & 0 & 0 & 0 & 0 & 0 & 0\\
          0 & 0 & 1 & 0 & 0 & 0 & 0 & 0\\
          0 & 0 & 0 & 1 & 0 & 0 & 0 & 0\\
          0 & 0 & 0 & 0 & 0 & -1 & 0 & 0\\
          0 & 0 & 0 & 0 & 1 & 0 & 0 & 0\\
          0 & 0 & 0 & 0 & 0 & 0 & 0 & 1\\
          0 & 0 & 0 & 0 & 0 & 0 & -1 & 0\end{pmatrix},\quad
      \varepsilon_2=\begin{pmatrix}1 & 0 & 0 & 0 & 0 & 0 & 0 & 0\\
          0 & 1 & 0 & 0 & 0 & 0 & 0 & 0\\
          0 & 0 & 1 & 0 & 0 & 0 & 0 & 0\\
          0 & 0 & 0 & 1 & 0 & 0 & 0 & 0\\
          0 & 0 & 0 & 0 & 0 & 0 & 1 & 0\\
          0 & 0 & 0 & 0 & 0 & 0 & 0 & 1\\
          0 & 0 & 0 & 0 & -1 & 0 & 0 & 0\\
          0 & 0 & 0 & 0 & 0 & -1 & 0 & 0\end{pmatrix}.
    \end{align*}
For $ \delta \in {\mathfrak{f}_4}^C $, since it follows that
\begin{align*}
    \varepsilon_i\delta{\varepsilon_i}^{-1}=(\varepsilon_i D_1{\varepsilon_i}^{-1},\varepsilon_i D_2{\varepsilon_i}^{-1},\varepsilon_i D_3{\varepsilon_i}^{-1})+\tilde{A}_1(\varepsilon_i a_1)+\tilde{A}_2(\varepsilon_i a_2)+\tilde{A}_3(\varepsilon_i a_3),i=1,2,
\end{align*}
 we obtain the required result using the matrix representations of $ \varepsilon_1, \varepsilon_2 $ above.
\end{proof}

Here, we define a Lie subalgebra $ \mathfrak{h}_4 $ of $ ({\mathfrak{f}_4}^C)^{\varepsilon_1,\varepsilon_2} $ by
\begin{align*}
    \mathfrak{h}_4:=\left\lbrace \delta_4=(L_1,L_2,L_3) \in ({\mathfrak{f}_4}^C)^{\varepsilon_1,\varepsilon_2} \relmiddle{|}
    \begin{array}{l}
        L_1=\lambda_0(iG_{01})+\lambda_1(iG_{23})+\lambda_2(i(G_{45}+G_{67})) \\
        L_2=\pi\kappa L_1, L_3=\kappa\pi L_1,\lambda_i \in C
    \end{array}
    \right\rbrace,
\end{align*}
where the explicit forms of $ L_2, L_3 $ are given as follow:
\begin{align*}
    L_2&=\dfrac{1}{2}(-\lambda_0+\lambda_1+2\lambda_2)(iG_{01})+\dfrac{1}{2}(-\lambda_0+\lambda_1-2\lambda_2)(iG_{23})+\dfrac{1}{2}(-\lambda_0-\lambda_1)(i(G_{45}+G_{67})),
    \\
    L_3&=\dfrac{1}{2}(-\lambda_0-\lambda_1-2\lambda_2)(iG_{01})+\dfrac{1}
    {2}(\lambda_0+\lambda_1-2\lambda_2)(iG_{23})+\dfrac{1}{2}(\lambda_0-
    \lambda_1)(i(G_{45}+G_{67})).
\end{align*}
Then $ \mathfrak{h}_4 $ is a Cartan subalgebra of $ ({\mathfrak{f}_4}^C)^{\varepsilon_1,\varepsilon_2} $. Indeed, it is clear that $ ({\mathfrak{f}_4}^C)^{\varepsilon_1,\varepsilon_2} $ is abelian. Next, by doing straightforward computation, we can confirm that $ [\delta,\delta_4] \in \mathfrak{h}_4, \delta \in ({\mathfrak{f}_4}^C)^{\varepsilon_1,\varepsilon_2} $ implies $ \delta \in \mathfrak{h}_4 $ for any $ \delta_4 \in \mathfrak{h}_4 $.

\begin{theorem}\label{theorem 5.3}
    The rank of the Lie algebra $ ({\mathfrak{f}_4}^C)^{\varepsilon_1,\varepsilon_2} $ is three. The roots $ \varDelta $ of $ ({\mathfrak{f}_4}^C)^{\varepsilon_1,\varepsilon_2} $ relative to $ \mathfrak{h}_4 $ are given by
    \begin{align*}
        \varDelta=\left\lbrace
        \begin{array}{l}
         \pm(\lambda_0-\lambda_1), \;
         \pm(\lambda_0+\lambda_1),\; \pm 2\lambda_2, \;
         \pm\lambda_0, \;\pm\lambda_1,
         \vspace{1mm}\\
         \pm\dfrac{1}{2}(-\lambda_0-\lambda_1-2\lambda_2),\;
         \pm\dfrac{1}{2}(\lambda_0+\lambda_1-2\lambda_2)
        \end{array}
       \right\rbrace.
    \end{align*}
\end{theorem}
\begin{proof}
  Let $ \delta_4=(L_1,L_2,L_3) \in \mathfrak{h}_4 $ and $ \delta_{01}:=(U_{01},\pi\kappa U_{01}, \kappa\pi U_{01}) \in ({\mathfrak{f}_4}^C)^{\varepsilon_1,\varepsilon_2} $ where $ U_{01}:=(G_{02}+G_{13})+i(G_{03}-G_{12}) $. Then it follows that
  \begin{align*}
      [L_1,U_{01}]&=[\lambda_0(iG_{01})+\lambda_1(iG_{23})+\lambda_2(i(G_{45}+G_{67})),(G_{02}+G_{13})+i(G_{03}-G_{12}) ]
      \\
      &=(\lambda_0-\lambda_1)G_{02}+(\lambda_0-\lambda_1)G_{13}+(\lambda_0-\lambda_1)iG_{03}-(\lambda_0-\lambda_1)iG_{12}
      \\
      &=(\lambda_0-\lambda_1)((G_{02}+G_{13})+i(G_{03}-G_{12}))
      \\
      &=(\lambda_0-\lambda_1)U_{01},
  \end{align*}
that is, $ [L_1,U_{01}]=(\lambda_0-\lambda_1)U_{01} $. Note that $ [L_1,U_{01}]=(\lambda_0-\lambda_1)U_{01} $ implies
\begin{align*}
    [L_2,\pi\kappa U_{01}]&=[\pi\kappa L_1,\pi\kappa U_{01}]=(\lambda_0-\lambda_1)\pi\kappa U_{01},
    \\
    [L_3,\pi\kappa U_{01}]&=[\kappa\pi L_1,\kappa\pi U_{01}]=(\lambda_0-\lambda_1)\kappa\pi U_{01},
\end{align*}
so that we have $ [\delta_4,\delta_{01}]=(\lambda_0-\lambda_1)\delta_{01} $. Hence we see that $ \lambda_0-\lambda_1 $ is a root of $ ({\mathfrak{f}_4}^C)^{\varepsilon_1,\varepsilon_2} $ and $ \delta_{01} $ is a root vector associated with its root. Similarly, for the elements of $ ({\mathfrak{f}_4}^C)^{\varepsilon_1,\varepsilon_2} $
\begin{align*}
 \delta_{01}^-:&=(U_{01}^-,\pi\kappa U_{01}^-, \kappa\pi U_{01}^-),U_{01}^-:=i(G_{02}+G_{13})+(G_{03}-G_{12}),
 \\
 \varepsilon_{01}:&=(V_{01},\pi\kappa V_{01}, \kappa\pi V_{01}) ,V_{01}:=(G_{02}-G_{13})-i(G_{03}+G_{12}),
 \\
 \varepsilon_{01}^-:&=(V_{01}^-,\pi\kappa V_{01}^-, \kappa\pi V_{01}^-) ,V_{01}^-:=i(G_{02}-G_{13})-(G_{03}+G_{12}),
 \\
 \varepsilon_{23}:&=(V_{23},\pi\kappa V_{23}, \kappa\pi V_{23}),V_{23}:=(G_{46}-G_{57})-i(G_{47}+G_{56}),
 \\
 \varepsilon_{23}^-:&=(V_{23}^-,\pi\kappa V_{23}^-, \kappa\pi V_{23}^-),V_{23}^-:=i(G_{46}-G_{57})-(G_{47}+G_{56}),
\end{align*}
we have
\begin{align*}
    &[\delta_4,\delta_{01}^-]=-(\lambda_0-\lambda_1)\delta_{01}^-, [\delta_4,\varepsilon_{01}]=(\lambda_0+\lambda_1)\varepsilon_{01},
    [\delta_4,\varepsilon_{01}^-]=-(\lambda_0+\lambda_1)\varepsilon_{01}^-,
    \\
    &[\delta_4,\varepsilon_{23}]=2\lambda_2\varepsilon_{23},
    [\delta_4,\varepsilon_{23}^-]=-2\lambda_2\varepsilon_{23}^-.
\end{align*}
 The remainders of roots and root vectors associated with these roots are obtained as follows and together with the results above:
\begin{longtable}{c@{\qquad}l}
  roots & \hspace{15mm} root vectors associated with roots
\vspace{1mm}\cr
$ -(\lambda_0-\lambda_1) $ & $ (U_{01}^-,\pi\kappa U_{01}^-, \kappa\pi U_{01}^-),U_{01}^-=i(G_{02}+G_{13})+(G_{03}-G_{12}) $
\vspace{1mm}\cr
$ \lambda_0+\lambda_1 $ & $ (V_{01},\pi\kappa V_{01}, \kappa\pi V_{01}) ,V_{01}=(G_{02}-G_{13})-i(G_{03}+G_{12}) $
\vspace{1mm}\cr
$ -(\lambda_0+\lambda_1) $ & $ (V_{01}^-,\pi\kappa V_{01}^-, \kappa\pi V_{01}^-),V_{01}^-=i(G_{02}-G_{13})-(G_{03}+G_{12}) $
\vspace{1mm}\cr
$ 2\lambda_2 $ & $ (V_{23},\pi\kappa V_{23}, \kappa\pi V_{23}),V_{23}=(G_{46}-G_{57})-i(G_{47}+G_{56}) $
\vspace{1mm}\cr
$ -2\lambda_2 $ & $ (V_{23}^-,\pi\kappa V_{23}^-, \kappa\pi V_{23}^-),V_{23}^-=i(G_{46}-G_{57})-(G_{47}+G_{56}) $
\vspace{1mm}\cr
$ \lambda_0 $ & $ \tilde{A}_1(1-ie_1) $
\vspace{1mm}\cr
$ -\lambda_0 $ & $ \tilde{A}_1(1+ie_1) $
\vspace{1mm}\cr
$ \lambda_1 $ & $ \tilde{A}_1(e_2-ie_3) $
\vspace{1mm}\cr
$ -\lambda_1 $ & $ \tilde{A}_1(e_2+ie_3) $
\vspace{1mm}\cr
$ (1/2)(-\lambda_0+\lambda_1+2\lambda_2) $ & $ \tilde{A}_2(1-ie_1) $
\vspace{1mm}\cr
$ -(1/2)(-\lambda_0+\lambda_1+2\lambda_2) $ & $ \tilde{A}_2(1+ie_1) $
\vspace{1mm}\cr
$ (1/2)(-\lambda_0+\lambda_1-2\lambda_2) $ & $ \tilde{A}_2(e_2-ie_3) $
\vspace{1mm}\cr
$ -(1/2)(-\lambda_0+\lambda_1-2\lambda_2) $ & $ \tilde{A}_2(e_2+ie_3) $
\vspace{1mm}\cr
$ (1/2)(-\lambda_0-\lambda_1-2\lambda_2) $ & $ \tilde{A}_3(1-ie_1) $
\vspace{1mm}\cr
$ -(1/2)(-\lambda_0-\lambda_1-2\lambda_2) $ & $ \tilde{A}_3(1+ie_1) $
\vspace{1mm}\cr
$ (1/2)(\lambda_0+\lambda_1-2\lambda_2) $ & $ \tilde{A}_3(e_2-ie_3) $
\vspace{1mm}\cr
$ -(1/2)(\lambda_0+\lambda_1-2\lambda_2) $ & $ \tilde{A}_3(e_2+ie_3). $
\end{longtable}

Thus The Lie algebra $ ({\mathfrak{f}_4}^C)^{\varepsilon_1,\varepsilon_2} $ is spanned by $ \mathfrak{h}_4 $ and the roots vector associated with roots above, the roots obtained above are all.  The rank of the Lie algebra $ ({\mathfrak{f}_4}^C)^{\varepsilon_1,\varepsilon_2} $ follows from the dimension of $ \mathfrak{h}_4 $.
\end{proof}

Subsequently, we prove te following theorem.

\begin{theorem}\label{theorem 5.4}
  In the root system $ \varDelta $ of Theorem {\rm 5.3},
  \begin{align*}
      \varPi=\left\lbrace \alpha_1,\alpha_2, \alpha_3 \right\rbrace
  \end{align*}
is a fundamental root system of $  ({\mathfrak{f}_4}^C)^{\varepsilon_1,\varepsilon_2} $, where $ \alpha_1=(1/2)(-\lambda_0+\lambda_1+2\lambda_2), \alpha_2=(1/2)(-\lambda_0-\lambda_1-2\lambda_2), \alpha_3=\lambda_0+\lambda_1 $. The Dynkin diagram of $ ({\mathfrak{f}_4}^C)^{\varepsilon_1,\varepsilon_2} $ is given by
\vspace{-3mm}

   {
       \setlength{\unitlength}{1mm}
        \scalebox{1.0}
        
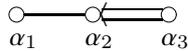
{\begin{picture}(100,20)
                \put(60,10){\circle{2}} \put(59,6){$\alpha_1$}
                \put(61,10){\line(1,0){8}}
                \put(70,10){\circle{2}} \put(69,6){$\alpha_2$}
                \put(70.7,9.2){$\langle$}
                \put(71.2,10.7){\line(1,0){8}}
                \put(71.2,9.3){\line(1,0){8}}
                \put(80,10){\circle{2}} \put(79,6){$\alpha_3$}

        \end{picture}}
    }
\end{theorem}
\begin{proof}
   The all positive roots are expressed by $ \alpha_1,\alpha_2,\alpha_3 $
   as follows:
   \begin{align*}
       -(\lambda_0-\lambda_1)&=2\alpha_1+2\alpha_2+\alpha_3,
       \\
       \lambda_0+\lambda_1&=\alpha_3,
       \\
       -2\lambda_2&=3\alpha_2+\alpha_3,
       \\
       -\lambda_0&=\alpha_1+\alpha_2,
       \\
       \lambda_1&=\alpha_1+\alpha_2+\alpha_3,
       \\
       (1/2)(-\lambda_0+\lambda_1+2\lambda_2)&=\alpha_1,
       \\
       (1/2)(-\lambda_0+\lambda_1-2\lambda_2)&=\alpha_1+2\alpha_2+\alpha_3,
       \\
       (1/2)(-\lambda_0-\lambda_1-2\lambda_2)&=\alpha_2,
       \\
       (1/2)(\lambda_0+\lambda_1-2\lambda_2)&=\alpha_2+\alpha_3.
   \end{align*}
Hence $ \varPi $ is a fundamental root system of $  ({\mathfrak{f}_4}
^C)^{\varepsilon_1,\varepsilon_2} $.

Then the Killing form $ B_4 $ of $ {\mathfrak{f}_4}^C $ is given by $
B_4(\delta,\delta')=3\tr(\delta\delta') $ (\cite[Theorem 2.5.3]{iy0}), so
that for $ \delta_{4}:=(L_1,L_2,l_3),{\delta_{4}}':=({L_1}',{L_2}',
{L_3}') \in \mathfrak{h}_4 $, we have
\begin{align*}
  B_4(\delta_{4},{\delta_{4}}')=3\tr(\delta_{4}
  {\delta_{4}}')=18(\lambda_0{\lambda_0}'+\lambda_1{\lambda_1}'
  +2\lambda_2{\lambda_2}').
\end{align*}
Indeed, it follows that
\begin{align*}
   B_4(\delta_{4},{\delta_{4}}')
&=3\tr(\delta_{4}{\delta_{4}}')
   \\
&=3\tr((L_1,L_2,L_3)({L_1}'{L_2}'{L_3}'))
   \\
&=3\tr(L_1{L_1}',L_2{L_2}',L_3{L_3}')
   \\
&=3(\tr(L_1{L_1}')+\tr(L_2{L_2}')+\tr(L_3{L_3}'))
   \\
&=3(3(2\lambda_0{\lambda_0}'+2\lambda_1{\lambda_1}'+4\lambda_2{\lambda_2}')
)
   \\
&=18(\lambda_0{\lambda_0}'+\lambda_1{\lambda_1}'
+2\lambda_2{\lambda_2}').
\end{align*}

Now, the canonical elements $ \delta_{\alpha_1}, \delta_{\alpha_2},
\delta_{\alpha_3} \in \mathfrak{h}_4$ corresponding to $ \alpha_1,\alpha_2,
\alpha_3 $ are determined as follows:
\begin{align*}
    \delta_{\alpha_1}&=-\dfrac{1}{36}(iG_{01})+\dfrac{1}{36}(iG_{23})+
    \dfrac{1}{36}(i(G_{45}+G_{67})),
    \\
    \delta_{\alpha_2}&=-\dfrac{1}{36}(iG_{01})-\dfrac{1}{36}(iG_{23})-
    \dfrac{1}{36}(i(G_{45}+G_{67})),
    \\
    \delta_{\alpha_3}&=\dfrac{1}{18}(iG_{01})+\dfrac{1}{18}(iG_{23}).
\end{align*}
Indeed, as for $ \delta_{\alpha_1} $, set $ \delta_{\alpha_1}:=(L_1, L_2,
L_3) \in \mathfrak{h}_4, L_1:=r_0(iG_{01})+r_1(iG_{23})+r_2(i(G_{45}
+G_{67}))
$, then since $ \delta_{\alpha_1} $ satisfies $ B_4(\delta_{\alpha_1},
\delta_4)=\alpha_1(\delta_4) $ for any $ \delta_4 \in \mathfrak{h}_4 $, it
follows from
\begin{align*}
    B_4(\delta_{\alpha_1},
    \delta_4)=18(r_0\lambda_0+r_1\lambda_1+2r_2\lambda_2),\;\;
    \alpha_1(\delta_4)&=(1/2)(-\lambda_0+\lambda_1+2\lambda_2)
\end{align*}
that
\begin{align*}
    r_0=-\dfrac{1}{36},\;\; r_1=\dfrac{1}{36},\;\; r_2=\dfrac{1}{36},
\end{align*}
that is, $ \delta_{\alpha_1}=-\dfrac{1}{36}(iG_{01})+\dfrac{1}{36}(iG_{23})
+
\dfrac{1}{36}(i(G_{45}+G_{67})) $. As for the others, we have the required
results by doing similar computations.

Hence we have the following
\begin{align*}
    (\alpha_1,\alpha_1)&=B_4(\delta_{\alpha_1},\delta_{\alpha_1})=18\left(
    \left(-\dfrac{1}{36} \right)^2+\left(\dfrac{1}{36} \right)^2+
    \left(\dfrac{1}{36} \right)^2 \right)=\dfrac{1}{18},
    \\
    (\alpha_1,\alpha_2)&=B_4(\delta_{\alpha_1},\delta_{\alpha_2})=18\left(
    \left(-\dfrac{1}{36} \right)^2
    +\left(\dfrac{1}{36} \right)\left(-\dfrac{1}{36} \right)+
     \left(\dfrac{1}{36} \right)\left(-\dfrac{1}{36} \right)
    \right)=-\dfrac{1}{12},
    \\
    (\alpha_1,\alpha_3)&=B_4(\delta_{\alpha_1},\delta_{\alpha_3})=18\left(
    \left(-\dfrac{1}{36} \right)\left( \dfrac{1}{18}\right)+\left(\dfrac{1}
    {36} \right)\left( \dfrac{1}{18}\right)  \right)=0,
    \\
    (\alpha_2,\alpha_2)&=B_4(\delta_{\alpha_2},\delta_{\alpha_2})=18\left(
    \left(-\dfrac{1}{36} \right)^2+\left(-\dfrac{1}{36} \right)^2+
    \left(-\dfrac{1}{36} \right)^2 \right)=\dfrac{1}{18},
    \\
    (\alpha_2,\alpha_3)&=B_4(\delta_{\alpha_2},\delta_{\alpha_3})=18\left(
    \left(-\dfrac{1}{36} \right)\left( \dfrac{1}{18}\right)+\left(-
    \dfrac{1}
    {36} \right)\left( \dfrac{1}{18}\right)  \right)=-\dfrac{1}{18},
    \\
    (\alpha_3,\alpha_3)&=B_4(\delta_{\alpha_3},\delta_{\alpha_3})=18\left(
    \left(\dfrac{1}{18} \right)^2+ \left(\dfrac{1}{18}
    \right)^2\right)=\dfrac{1}{9}.
\end{align*}
Thus, using the inner products above, we have
\begin{align*}
    \cos\theta_{12}&=\dfrac{(\alpha_1,\alpha_2)}{\sqrt{(\alpha_1,\alpha_1)
    (\alpha_2,\alpha_2)}}=-\dfrac{1}{2},\quad
    \cos\theta_{13}=\dfrac{(\alpha_1,\alpha_3)}{\sqrt{(\alpha_1,\alpha_1)
    (\alpha_3,\alpha_3)}}=0,
    \\
    \cos\theta_{23}&=\dfrac{(\alpha_2,\alpha_3)}{\sqrt{(\alpha_2,\alpha_2)
    (\alpha_3,\alpha_3)}}=-\dfrac{1}{\sqrt{2}},\quad (\alpha_2,
    \alpha_2)=\dfrac{1}{18} < \dfrac{1}{9}=(\alpha_3,\alpha_3),
\end{align*}
so that we can draw the Dynkin diagram.
\end{proof}

\section{The group $ E_{6,\sH} $ and the root system, the Dynkin
    diagram of the Lie algebra $ ({\mathfrak{e}_{6}}^C)^{\varepsilon_1,
    \varepsilon_2} $}

Let $\mathfrak{J}(3,\mathfrak{C}^C)  $ be the complex exceptional Jordan
algebra.
As in previous section, we consider an algebra $ \mathfrak{J}(3,\H^C) $
which
are defined by replacing $ \mathfrak{C} $ with $ \H $ in
$ \mathfrak{J}(3,\mathfrak{C}^C) $, and we briefly denote $
\mathfrak{J}(3,\H^C) $ by $ (\mathfrak{J}_{\sH})^C $.

Then we consider the following group $ E_{6,\sH} $ which is defined by
replacing $ \mathfrak{J}^C $ with $ (\mathfrak{J}_{\sH})^C $ in $ E_6 $:
\begin{align*}
    E_{6,\sH}&=\left\lbrace \alpha \in \Iso_{C}((\mathfrak{J}_{\sH})^C)
    \relmiddle{|} \det\,(\alpha X)=\det\,X,\langle \alpha X,\alpha Y
    \rangle=\langle X,Y \rangle \right\rbrace
    \\
    &=\left\lbrace \alpha \in \Iso_{C}((\mathfrak{J}_{\sH})^C)\relmiddle{|}
    \tau\alpha\tau(X \times Y)=\alpha X \times \alpha Y, \langle \alpha X,
    \alpha Y
    \rangle=\langle X,Y \rangle \right\rbrace.
\end{align*}

This group also has been investigated by Ichiro Yokota, so that the
following
result was obtained.

\begin{theorem}{\rm (\cite[Proposition 3.11.3]{iy0})}\label{theorem 6.1} The group $ E_{6,\sH} $ is isomorphic to the group $ SU(6)/\Z_2, \allowbreak \Z_2=\{E,-E\} ${\rm :} $ E_{6,\sH} \cong SU(6)/\Z_2 $.
\end{theorem}

We consider a subgroup $ ({E_6}^C)^{\varepsilon_1,\varepsilon_2} $ of $ {E_6}^C $ by
\begin{align*}
    ({E_6}^C)^{\varepsilon_1,\varepsilon_2}:=\left\lbrace \alpha
    \in {E_6}^C \relmiddle{|}
    \varepsilon_1\alpha=\alpha\varepsilon_1,
    \varepsilon_2\alpha=\alpha\varepsilon_2 \right\rbrace.
\end{align*}
Then the group $ ({E_6}^C)^{\varepsilon_1,\varepsilon_2} $  has been studied by author in \cite[Proposition 1.1.1]{miya1}, and its result was obtained as follows:
\begin{align*}
    ({E_6}^C)^{\varepsilon_1,\varepsilon_2} \cong SU^*(6,\C^C)(\cong SU(6,\C^C)).
\end{align*}

Now, we move the determination of the root system and the Dynkin diagram of the Lie algebra $ ({\mathfrak{e}_{6}}^C)^{\varepsilon_1,\varepsilon_2} $ of the group $ ({E_6}^C)^{\varepsilon_1,\varepsilon_2} $. After this, we will determine those by the same procedure as in $
({\mathfrak{f}_4}^C)^{\varepsilon_1,\varepsilon_2} $.

We prove the following lemma needed later.

\begin{lemma}\label{lemma 6.2}
    The Lie algebra $ ({\mathfrak{e}_6}
    ^C)^{\varepsilon_1,\varepsilon_2} $ of the group $ ({E_6}^C)^{\varepsilon_1,\varepsilon_2} $ is given by
    \begin{align*}
        ({\mathfrak{e}_6}^C)^{\varepsilon_1,\varepsilon_2}
        =\left\lbrace
        \begin{array}{l}
            \delta=(D_1,D_2,D_3) \vspace{1mm}\\
            \quad +\tilde{A}_1(a_1)+\tilde{A}_2(a_2)+\tilde{A}_3(a_3)\\
            \quad +(\tau_1E_1+\tau_2E_2+\tau_3E_3+F_1(t_1) \\
            \qquad\quad +F_2(t_2)+F_3(t_3))^\sim  \in {\mathfrak{e}_6}^C
        \end{array}
        \relmiddle{|}
        \begin{array}{l}
            D_1=d_{01}G_{01}+d_{02}G_{02}+d_{03}G_{03} \\
            \quad +d_{12}G_{12}+d_{13}G_{13}+d_{23}G_{23} \\
            \quad +d_{45}(G_{45}+G_{67})+d_{46}(G_{46}-G_{57}) \\
            \quad +d_{47}(G_{47}+G_{56}),d_{ij} \in C, \\
            D_2=\pi\kappa D_1, D_3=\kappa\pi D_1,\\
            a_k, t_k \in \H^C ,\\
            \tau_k \in C, \tau_1+\tau_2+\tau_3=0
        \end{array}
        \right\rbrace.
    \end{align*}

    In particular, we have $ \dim_C( ({\mathfrak{e}_6}^C)^{\varepsilon_1,
    \varepsilon_2})=9+4 \times 3+(3-1+4 \times 3)=35. $
\end{lemma}
\begin{proof}
    For $ \phi:=\delta + \tilde{T} \in {\mathfrak{e}_6}^C, \delta \in
    {\mathfrak{f}_4}^C, T \in (\mathfrak{J}^C)_0:=\{T \in \mathfrak{J}^C\,|\, \tr(T)=0 \} $, it follows that
    \begin{align*}
     \varepsilon_i\phi{\varepsilon_i}^{-1}= \varepsilon_i(\delta +
     \tilde{T}){\varepsilon_i}^{-1}=\varepsilon_i\delta{\varepsilon_i}^{-1}
     +(\varepsilon_iT)^\sim,\; i=1,2
    \end{align*}
using Lemma \ref{lemma 5.2}, so that we have the required result.
\end{proof}

Here, we define a Lie subalgebra $ \mathfrak{h}_6 $ of $ ({\mathfrak{e}_6}
^C)^{\varepsilon_1,\varepsilon_2} $ by
\begin{align*}
    \mathfrak{h}_6:=\left\lbrace
    \begin{array}{l}
    \phi_6=\delta_4 +(\mu_1E_1+\mu_2E_2+
    \mu_3E_3)^\sim \in ({\mathfrak{e}_6}^C)^{\varepsilon_1,\varepsilon_2}
    \end{array}
    \relmiddle{|}
    \begin{array}{l}
       \delta_4:=(L_1,L_2,L_3),
       \\
       \quad L_1=\lambda_0(iG_{01})+\lambda_1(iG_{23})
        \\
        \qquad\qquad\;\; +\lambda_2(i(G_{45}+G_{67})),
        \\
       \quad  L_2=\pi\kappa L_1, L_3=\kappa\pi L_1,\lambda_i \in C,
        \\
        \mu_k \in C, \mu_1+\mu_2+\mu_3=0
    \end{array}
    \right\rbrace.
\end{align*}

Then $ \mathfrak{h}_6 $ is a Cartan subalgebra of $ ({\mathfrak{e}_6}
^C)^{\varepsilon_1,\varepsilon_2} $. Indeed, it is clear that $
({\mathfrak{e}_6}^C)^{\varepsilon_1,\varepsilon_2} $ is abelian. Next, by
doing straightforward computation, we can confirm that $ [\phi,\phi_6]
\in \mathfrak{h}_6, \delta \in ({\mathfrak{e}_6}^C)^{\varepsilon_1,
\varepsilon_2} $ implies $ \phi \in \mathfrak{h}_6 $ for any $ \phi_6
\in \mathfrak{h}_6 $.

\begin{theorem}\label{theorem 6.3}
    The rank of the Lie algebra $ ({\mathfrak{e}_6}^C)^{\varepsilon_1,
    \varepsilon_2} $ is five. The roots $ \varDelta $ of $ ({\mathfrak{e}
    _6}^C)^{\varepsilon_1,\varepsilon_2} $ relative to $ \mathfrak{h}_6 $
    are given by
    \begin{align*}
        \varDelta=\left\lbrace
        \begin{array}{l}
            \pm(\lambda_0-\lambda_1), \;
            \pm(\lambda_0+\lambda_1),\; \pm 2\lambda_2, \;
            \vspace{1mm}\\
            \pm\dfrac{1}{2}(-2\lambda_0+\mu_2-\mu_3),\;
            \pm\dfrac{1}{2}(-2\lambda_0-\mu_2+\mu_3),
            \vspace{1mm}\\
            \pm\dfrac{1}{2}(-2\lambda_1+\mu_2-\mu_3),\;
            \pm\dfrac{1}{2}(-2\lambda_1-\mu_2+\mu_3),
            \vspace{1mm}\\
            \pm\dfrac{1}{2}(\lambda_0-\lambda_1-2\lambda_2-\mu_1+\mu_3),\;
            \pm\dfrac{1}{2}(\lambda_0-\lambda_1-2\lambda_2+\mu_1-\mu_3),
            \vspace{1mm}\\
            \pm\dfrac{1}{2}(\lambda_0-\lambda_1+2\lambda_2-\mu_1+\mu_3),\;
            \pm\dfrac{1}{2}(\lambda_0-\lambda_1+2\lambda_2+\mu_1-\mu_3),
            \vspace{1mm}\\
            \pm\dfrac{1}{2}(\lambda_0+\lambda_1+2\lambda_2+\mu_1-\mu_2),\;
            \pm\dfrac{1}{2}(\lambda_0+\lambda_1+2\lambda_2-\mu_1+\mu_2),
            \vspace{1mm}\\
            \pm\dfrac{1}{2}(-\lambda_0-\lambda_1+2\lambda_2+\mu_1-\mu_2),\;
            \pm\dfrac{1}{2}(-\lambda_0-\lambda_1+2\lambda_2-\mu_1+\mu_2).
        \end{array}
        \right\rbrace.
    \end{align*}
\end{theorem}
\begin{proof}
 First, the roots $ \pm(\lambda_0-\lambda_1),\pm(\lambda_0+\lambda_1),
 \pm 2\lambda_2 $ of $ ({\mathfrak{f}_4}
 ^C)^{\varepsilon_1,\varepsilon_2} $ are also the roots of
 $ ({\mathfrak{e}_6}^C)^{\varepsilon_1,\varepsilon_2} $.
 Indeed, for example, let $ \lambda_0-\lambda_1 $ and its associated root
 vector $ \delta_{01} $. Then it follows that
 \begin{align*}
   [\phi_6, \delta_{01}]
   &=[\delta_4+(\tau_1E_1+\tau_2E_2+\tau_3E_3)^\sim,\delta_{01}]
   \\
   &=[\delta_4,\delta_{01}]+[(\tau_1E_1+\tau_2E_2+\tau_3E_3)^\sim,
   \delta_{01}]\;\;([(\tau_1E_1+\tau_2E_2+\tau_3E_3)^\sim,\delta_{01}]=0)
   \\
   &=(\lambda_0-\lambda_1)(\delta_{01}).
 \end{align*}
The similar results are obtained for the others roots above.

We will determine the remainders of roots. Let $ \phi_6 \in \mathfrak{h}_6
$ and $ \tilde{A}_1(1+ie_1)+\tilde{F}_1(1+ie_1) \in ({\mathfrak{f}_4}
^C)^{\varepsilon_1,\varepsilon_2} $. Then it follows that
\begin{align*}
   [\phi_6,\tilde{A}_1(1+ie_1)+\tilde{F}_1(1+ie_1)]&=\![(L_1,L_2,L_3)+
   (\mu_1E_1+\mu_2E_2+\mu_3E_3)^\sim,\tilde{A}_1(1+ie_1)+\tilde{F}
   _1(1+ie_1)]
   \\
   &=\tilde{A}_1(L_1(1+ie_1))+\tilde{F}_1(L_1(1+ie_1))
   \\
   &\qquad +(1/2)\mu_2\tilde{F}_1(1+ie_1)+(1/2)\mu_2\tilde{A}_1(1+ie_1)
   \\
   &\qquad -(1/2)\mu_3\tilde{F}_1(1+ie_1)-(1/2)\mu_3\tilde{A}_1(1+ie_1)
   \\
   &=-\lambda_0\tilde{A}_1(1+ie_1))-\lambda_0\tilde{F}_1(L_1(1+ie_1))
   \\
   &\qquad +(1/2)\mu_2\tilde{F}_1(1+ie_1)+(1/2)\mu_2\tilde{A}_1(1+ie_1)
   \\
   &\qquad -(1/2)\mu_3\tilde{F}_1(1+ie_1)-(1/2)\mu_3\tilde{A}_1(1+ie_1)
   \\
   &=(1/2)(-2\lambda_0+\mu_2-\mu_3)(\tilde{A}_1(1+ie_1)+\tilde{F}
   _1(1+ie_1)),
\end{align*}
that is, $ [\phi_6,\tilde{A}_1(1+ie_1)+\tilde{F}_1(1+ie_1)]=(1/2)
(-2\lambda_0+\mu_2-\mu_3)(\tilde{A}_1(1+ie_1)+\tilde{F}
_1(1+ie_1)) $.
Hence we see that $ (1/2)(-2\lambda_0+\mu_2-\mu_3) $ is a root of $
({\mathfrak{e}_6}^C)^{\varepsilon_1,\varepsilon_2} $ and $ \tilde{A}
_1(1+ie_1)+\tilde{F}_1(1+ie_1) $ is a root vector associated with its root.
Similarly, we have that $ -(1/2)(-2\lambda_0+\mu_2-\mu_3) $ is a root of $
({\mathfrak{e}_6}^C)^{\varepsilon_1,\varepsilon_2} $ and $ \tilde{A}
_1(1-ie_1)-\tilde{F}_1(1-ie_1) $ is a root vector associated with its root.

The remainders of roots and root vectors associated with these roots are
obtained as follows:
\begin{longtable}{c@{\qquad}l}
   roots & root vectors associated with roots
   \vspace{0.3mm}\cr
   $ (1/2)(-2\lambda_0-\mu_2+\mu_3) $ & $ \tilde{A}_1(1+ie_1)
   -\tilde{F}_1(1+ie_1) $
   \vspace{0.3mm}\cr
   $ -(1/2)(-2\lambda_0-\mu_2+\mu_3) $ & $ \tilde{A}_1(1-ie_1)
   +\tilde{F}_1(1-ie_1) $
   \vspace{0.3mm}\cr
   $ (1/2)(-2\lambda_1+\mu_2-\mu_3) $ & $ \tilde{A}_1(e_2+ie_3)
   +\tilde{F}_1(e_2+ie_3) $
   \vspace{0.3mm}\cr
   $ -(1/2)(-2\lambda_1+\mu_2-\mu_3) $ & $ \tilde{A}_1(e_2-ie_3)
   -\tilde{F}_1(e_2-ie_3) $
   \vspace{0.3mm}\cr
   $ (1/2)(-2\lambda_1-\mu_2+\mu_3) $ & $ \tilde{A}_1(e_2+ie_3)
   -\tilde{F}_1(e_2+ie_3) $
   \vspace{0.3mm}\cr
   $ -(1/2)(-2\lambda_1-\mu_2+\mu_3) $ & $ \tilde{A}_1(e_2-ie_3)
   +\tilde{F}_1(e_2-ie_3) $
   \vspace{0.3mm}\cr
   $(1/2)(\lambda_0-\lambda_1-2\lambda_2-\mu_1+\mu_3)  $ & $ \tilde{A}
   _2(1+ie_1)+\tilde{F}_2(1+ie_1))  $
   \vspace{0.3mm}\cr
   $-(1/2)(\lambda_0-\lambda_1-2\lambda_2-\mu_1+\mu_3)  $ & $ \tilde{A}
   _2(1-ie_1)-\tilde{F}_2(1-ie_1))  $
   \vspace{0.3mm}\cr
   $(1/2)(\lambda_0-\lambda_1-2\lambda_2+\mu_1-\mu_3)  $ & $ \tilde{A}
   _2(1+ie_1)-\tilde{F}_2(1+ie_1))  $
   \vspace{0.3mm}\cr
   $-(1/2)(\lambda_0-\lambda_1-2\lambda_2+\mu_1-\mu_3)  $ & $ \tilde{A}
   _2(1-ie_1)+\tilde{F}_2(1-ie_1))  $
   \vspace{0.3mm}\cr
   $(1/2)(\lambda_0-\lambda_1+2\lambda_2-\mu_1+\mu_3)  $ & $ \tilde{A}
   _2(e_2+ie_3)+\tilde{F}_2(e_2+ie_3))  $
   \vspace{0.3mm}\cr
   $-(1/2)(\lambda_0-\lambda_1+2\lambda_2-\mu_1+\mu_3)  $ & $ \tilde{A}
   _2(e_2-ie_3)-\tilde{F}_2(e_2-ie_3))  $
   \vspace{0.3mm}\cr
   $(1/2)(\lambda_0-\lambda_1+2\lambda_2+\mu_1-\mu_3)  $ & $ \tilde{A}
   _2(e_2+ie_3)-\tilde{F}_2(e_2+ie_3))  $
   \vspace{0.3mm}\cr
   $-(1/2)(\lambda_0-\lambda_1+2\lambda_2+\mu_1-\mu_3)  $ & $ \tilde{A}
   _2(e_2-ie_3)+\tilde{F}_2(e_2-ie_3))  $
   \vspace{0.3mm}\cr
   $(1/2)(\lambda_0+\lambda_1+2\lambda_2+\mu_1-\mu_2)  $ & $ \tilde{A}
   _3(1+ie_1)+\tilde{F}_3(1+ie_1))  $
   \vspace{0.3mm}\cr
   $-(1/2)(\lambda_0+\lambda_1+2\lambda_2+\mu_1-\mu_2)  $ & $ \tilde{A}
   _3(1-ie_1)-\tilde{F}_3(1-ie_1))  $
   \vspace{0.3mm}\cr
   $(1/2)(\lambda_0+\lambda_1+2\lambda_2-\mu_1+\mu_2)  $ & $ \tilde{A}
   _3(1+ie_1)-\tilde{F}_3(1+ie_1))  $
   \vspace{0.3mm}\cr
   $-(1/2)(\lambda_0+\lambda_1+2\lambda_2-\mu_1+\mu_2)  $ & $ \tilde{A}
   _3(1-ie_1)+\tilde{F}_3(1-ie_1))  $
   \vspace{0.3mm}\cr
   $(1/2)(-\lambda_0-\lambda_1+2\lambda_2+\mu_1-\mu_2)  $ & $ \tilde{A}
   _3(e_2+ie_3)+\tilde{F}_3(e_2+ie_3))  $
   \vspace{0.3mm}\cr
   $-(1/2)(-\lambda_0-\lambda_1+2\lambda_2+\mu_1-\mu_2)  $ & $ \tilde{A}
   _3(e_2-ie_3)-\tilde{F}_3(e_2-ie_3))  $
   \vspace{0.3mm}\cr
   $(1/2)(-\lambda_0-\lambda_1+2\lambda_2-\mu_1+\mu_2)  $ & $ \tilde{A}
   _3(e_2+ie_3)-\tilde{F}_3(e_2+ie_3))  $
   \vspace{0.3mm}\cr
   $-(1/2)(-\lambda_0-\lambda_1+2\lambda_2-\mu_1+\mu_2)  $ & $ \tilde{A}
   _3(e_2-ie_3)+\tilde{F}_3(e_2-ie_3))  $.
\end{longtable}

Thus, since $ ({\mathfrak{e}_6}^C)^{\varepsilon_1,\varepsilon_2} $ is
spanned by $ \mathfrak{h}_6 $ and the root vectors associated with roots
above, the roots obtained above are all. The rank of the Lie algebra $
({\mathfrak{e}_6}^C)^{\varepsilon_1,\varepsilon_2} $ follows from the
dimension of $ \mathfrak{e}_6 $.
\end{proof}

Subsequently, we prove the following theorem.

\begin{theorem}\label{theorem 6.4}
In the root system $ \varDelta $ of Theorem {\rm \ref{theorem 6.3}}
 \begin{align*}
      \varPi=\left\lbrace \alpha_1,\alpha_2, \alpha_3, \alpha_4,\alpha_5
      \right\rbrace
  \end{align*}
is a fundamental root system of $  ({\mathfrak{e}_6}^C)^{\varepsilon_1,
\varepsilon_2} $, where $ \alpha_1=-(\lambda_0-\lambda_1),
\alpha_2=(1/2)
(\lambda_0-\lambda_1-2\lambda_2+\mu_1-\mu_3),
\alpha_3=2\lambda_2,
\alpha_4=-(1/2)
(\lambda_0+\lambda_1+2\lambda_2+\mu_1-\mu_2),
\alpha_5=\lambda_0+\lambda_1$.
The Dynkin diagram of $ ({\mathfrak{e}_6}^C)^{\varepsilon_1,
\varepsilon_2} $ is given by
\vspace{-3mm}
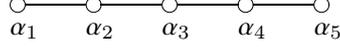
\begin{center}
\setlength{\unitlength}{1mm}
  \scalebox{1.0}
	{\begin{picture}(80,20)
	\put(0,9){}
	\put(20,10){\circle{2}} \put(19,6){$\alpha_1$}
	\put(21,10){\line(1,0){8}}
	\put(30,10){\circle{2}} \put(29,6){$\alpha_2$}
	\put(31,10){\line(1,0){8}}
	\put(40,10){\circle{2}} \put(39,6){$\alpha_3$}
	\put(41,10){\line(1,0){8}}
	\put(50,10){\circle{2}} \put(49,6){$\alpha_4$}
	\put(51,10){\line(1,0){8}}
	\put(60,10){\circle{2}} \put(59,6){$\alpha_5$}

	\end{picture}}
\end{center}
\end{theorem}
\begin{proof}
The all positive roots are expressed by $ \alpha_1,\alpha_2,\alpha_3,
\alpha_4,\alpha_5 $ as follows:
   \begin{align*}
-(\lambda_0-\lambda_1)&=\alpha_1,
\\
\lambda_0+\lambda_1&=\alpha_5,
\\
2\lambda_2&=\alpha_3,
\\
(1/2)(-2\lambda_0+\mu_2-\mu_3)&=\alpha_1+\alpha_2+\alpha_3+\alpha_4,
\\
-(1/2)(-2\lambda_0-\mu_2+\mu_3)&=\alpha_2+\alpha_3+\alpha_4+\alpha_5,
\\
(1/2)(-2\lambda\_1+\mu_2-\mu_3)&=\alpha_2+\alpha_3+\alpha_4,
\\
-(1/2)(-2\lambda_1-\mu_2+\mu_3)&=\alpha_1+\alpha_2+\alpha_3+\alpha_4+
\alpha_5,
\\
-(1/2)(\lambda_0-\lambda_1-2\lambda_2-\mu_1+\mu_3)&=\alpha_1+\alpha_2+
\alpha_3,
\\
(1/2)(\lambda_0-\lambda_1-2\lambda_2+\mu_1-\mu_3)&=\alpha_2,
\\
-(1/2)(\lambda_0-\lambda_1+2\lambda_2-\mu_1+\mu_3)&=\alpha_1+\alpha_2,
\\
(1/2)(\lambda_0-\lambda_1+2\lambda_2+\mu_1-\mu_3)&=\alpha_2+\alpha_3,
\\
-(1/2)(\lambda_0+\lambda_1+2\lambda_2+\mu_1-\mu_2)&=\alpha_4,
\\
(1/2)(\lambda_0+\lambda_1+2\lambda_2-\mu_1+\mu_2)&=\alpha_3+\alpha_4+
\alpha_5,
\\
-(1/2)(-\lambda_0-\lambda_1+2\lambda_2+\mu_1-\mu_2)&=\alpha_4+\alpha_5,
\\
(1/2)(-\lambda_0-\lambda_1+2\lambda_2-\mu_1+\mu_2)&=\alpha_3+\alpha_4.
   \end{align*}
Hence $ \varPi $ is a fundamental root system of $  ({\mathfrak{e}_6}
^C)^{\varepsilon_1,\varepsilon_2} $.

Then, for $ \phi,\phi' \in {\mathfrak{e}_6}^C$, the Killing form $ B_6 $ of
$ {\mathfrak{e}_6}^C $ is given by
\begin{align*}
B_6(\phi,\phi')=B_6(\delta+\tilde{T},\delta'+\tilde{T'})=\dfrac{4}{3}
B_4(\delta,
\delta')+12(T,T')
\end{align*}
(\cite[Theorem 3.5.3]{iy0}), so that for $ \phi_6:=\delta_4+(\mu_1E_1+
\mu_2E_2+\mu_3E_3)^\sim, {\phi_6}':={\delta_4}'+({\mu_1}'E_1+{\mu_2}'E_2+
{\mu_3}'E_3)^\sim \in \mathfrak{h}_6 $, we have
\begin{align*}
 B_6(\phi_{4},{\phi_{4}}')&=\dfrac{4}{3}B_4(\delta_{4},{\delta_{4}}')
 +12(\mu_1E_1+\mu_2E_2+\mu_3E_3,{\mu_1}'E_1+{\mu_2}'E_2+
{\mu_3}'E_3)
\\
&=\dfrac{4}{3}\cdot18(\lambda_0{\lambda_0}'+\lambda_1{\lambda_1}'
  +2\lambda_2{\lambda_2}')+12(\mu_1{\mu_1}'+\mu_2{\mu_2}'+\mu_3{\mu_3}')
\\
&=24(\lambda_0{\lambda_0}'+\lambda_1{\lambda_1}'
  +2\lambda_2{\lambda_2}')+12(\mu_1{\mu_1}'+\mu_2{\mu_2}'+\mu_3{\mu_3}').
\end{align*}

Now, the canonical elements $ \phi_{\alpha_1},
\phi_{\alpha_2}, \phi_{\alpha_3}, \phi_{\alpha_4},\phi_{\alpha_5} $
corresponding to $ \alpha_1, \alpha_2,\alpha_3,\alpha_4,\alpha_5 $
 are determined as follows:
\begin{align*}
\phi_{\alpha_1}&=-\dfrac{1}{24}(iG_{01})+\dfrac{1}{24}(iG_{23}),
\\
\phi_{\alpha_2}&=\dfrac{1}{48}(iG_{01})-\dfrac{1}{48}(iG_{23})-\dfrac{1}
{48}(i(G_{45}+G_{67}))+(\dfrac{1}{24}E_1-\dfrac{1}{24}E_3)^\sim,
\\
\phi_{\alpha_3}&=\dfrac{1}{24}(i(G_{45}+G_{67})),
\\
\phi_{\alpha_4}&=-\dfrac{1}{48}(iG_{01})-\dfrac{1}{48}(iG_{23})-\dfrac{1}
{48}(i(G_{45}+G_{67}))+(-\dfrac{1}{24}E_1+\dfrac{1}{24}E_2)^\sim,
\\
\phi_{\alpha_5}&=\dfrac{1}{24}(iG_{01})+\dfrac{1}{24}(iG_{23}).
\end{align*}
Indeed, as for $ \phi_{\alpha_1} $, set $ \phi_{\alpha_1}:=(L_1, L_2,
L_3)+(\nu_1E_1+\nu_2E_2+\nu_3E_3)^\sim \in \mathfrak{h}_6,
L_1:=r_0(iG_{01})+r_1(iG_{23})+r_2(i(G_{45}+G_{67}))
$, then since $ \phi_{\alpha_1} $ satisfies $ B_6(\phi_{\alpha_1},
\phi_6)=\alpha_1(\phi_6) $ for any $ \phi_6 \in \mathfrak{h}_6 $, it
follows from
\begin{align*}
    B_6(\phi_{\alpha_1},
    \phi_6)&=24(r_0\lambda_0+r_1\lambda_1+2r_2\lambda_2)+12(\nu_1\mu_1+
    \nu_2\mu_2+\nu_3\mu_3),
    \\
    \alpha_1(\phi_6)&=-\lambda_0+\lambda_1
\end{align*}
that
\begin{align*}
    r_0=-\dfrac{1}{24},\;\; r_1=\dfrac{1}{24},
\end{align*}
that is, $ \phi_{\alpha_1}=-\dfrac{1}{24}(iG_{01})+\dfrac{1}{24}(iG_{23})
$. As for the others, we have the required results by doing similar
computations.

Hence we have the following
\begin{align*}
    (\alpha_1,\alpha_1)&=B_6(\phi_{\alpha_1},\phi_{\alpha_1})=24\left(
    \left(-\dfrac{1}{24} \right)^2+\left(\dfrac{1}{24} \right)^2
    \right)=\dfrac{1}{12},
    \\
    (\alpha_1,\alpha_2)&=B_6(\phi_{\alpha_1},\phi_{\alpha_2})=24\left(
    \left(-\dfrac{1}{24} \right)\left(\dfrac{1}{48} \right)
    +\left(\dfrac{1}{24} \right)\left(-\dfrac{1}{48} \right)
    \right)=-\dfrac{1}{24},
    \\
    (\alpha_1,\alpha_3)&=B_6(\phi_{\alpha_1},\phi_{\alpha_3})=0,
    \\
    (\alpha_1,\alpha_4)&=B_6(\phi_{\alpha_1},\phi_{\alpha_4})=24\left(
    \left(-\dfrac{1}{24} \right)\left(-\dfrac{1}{48} \right)
    +\left(\dfrac{1}{24} \right)\left(-\dfrac{1}{48} \right)
    \right)=0,
    \\
    (\alpha_1,\alpha_5)&=B_6(\phi_{\alpha_1},\phi_{\alpha_5})=24\left(
    \left(-\dfrac{1}{24} \right)\left(\dfrac{1}{24} \right)+\left(\dfrac{1}
    {24} \right)^2 \right)=0,
    \\
    (\alpha_2,\alpha_2)&=B_6(\phi_{\alpha_2},\phi_{\alpha_2})=24\left(
    \left(\dfrac{1}{48} \right)^2+\left(-\dfrac{1}{48} \right)^2+
    2\left(-\dfrac{1}{48} \right)^2 \right)+12\left(
    \left(\dfrac{1}{24} \right)^2+\left(-\dfrac{1}{24} \right)^2
    \right)=\dfrac{1}{12},
    \\
    (\alpha_2,\alpha_3)&=B_6(\phi_{\alpha_2},
    \phi_{\alpha_3})=24\cdot 2\left(-\dfrac{1}{48} \right)\left(
    \dfrac{1}{24}\right)=-\dfrac{1}{24},
    \\
    (\alpha_2,\alpha_4)&=B_6(\phi_{\alpha_2},\phi_{\alpha_4})=24\left(
    \left(\dfrac{1}{48} \right)\left(-\dfrac{1}{48} \right)+\left(-
    \dfrac{1}{48} \right)^2+
    2\left(-\dfrac{1}{48} \right)^2 \right)+12
    \left(\dfrac{1}{24} \right)\left(-\dfrac{1}{24} \right)=0,
    \\
    (\alpha_2,\alpha_5)&=B_6(\phi_{\alpha_2},\phi_{\alpha_5})=24\left(
    \left(\dfrac{1}{48} \right)\left(\dfrac{1}{24} \right)+\left(-\dfrac{1}
    {48} \right)\left(\dfrac{1}{24} \right) \right)=0,
    \\
    (\alpha_3,\alpha_3)&=B_6(\phi_{\alpha_3},\phi_{\alpha_3})=24\cdot
    2\left(\dfrac{1}{24}\right)^2=\dfrac{1}{12},
    \\
    (\alpha_3,\alpha_4)&=B_6(\phi_{\alpha_3},\phi_{\alpha_4})=24\cdot
    2\left(\dfrac{1}{24}\right)\left(-\dfrac{1}{48} \right)=-\dfrac{1}
    {24},
    \\
    (\alpha_3,\alpha_5)&=B_6(\phi_{\alpha_3},\phi_{\alpha_5})=0,
    \\
    (\alpha_4,\alpha_4)&=B_6(\phi_{\alpha_4},\phi_{\alpha_4})=24\left(
    \left(-\dfrac{1}{48} \right)^2+\left(-\dfrac{1}{48} \right)^2+
    2\left(-\dfrac{1}{48} \right)^2 \right)+12\left(
    \left(-\dfrac{1}{24} \right)^2+\left(\dfrac{1}{24} \right)^2
    \right)=\dfrac{1}{12},
    \\
    (\alpha_4,\alpha_5)&=B_6(\phi_{\alpha_4},\phi_{\alpha_5})=24\left(
    \left(-\dfrac{1}{48} \right)\left(\dfrac{1}{24} \right)+\left(-
    \dfrac{1}{48} \right)\left(\dfrac{1}{24} \right) \right)=-\dfrac{1}
    {24},
    \\
    (\alpha_5,\alpha_5)&=B_6(\phi_{\alpha_5},\phi_{\alpha_5})=24\left(
    \left(\dfrac{1}{24} \right)^2+\left(\dfrac{1}{24} \right)^2
    \right)=\dfrac{1}{12}.
\end{align*}
Thus, using the inner products above, we have
\begin{align*}
    \cos\theta_{12}&=\dfrac{(\alpha_1,\alpha_2)}{\sqrt{(\alpha_1,\alpha_1)
    (\alpha_2,\alpha_2)}}=-\dfrac{1}{2},\quad
    \cos\theta_{13}=\dfrac{(\alpha_1,\alpha_3)}{\sqrt{(\alpha_1,\alpha_1)
    (\alpha_3,\alpha_3)}}=0,
    \\
    \cos\theta_{14}&=\dfrac{(\alpha_1,\alpha_4)}{\sqrt{(\alpha_1,\alpha_1)
    (\alpha_4,\alpha_4)}}=0,\quad
    \cos\theta_{15}=\dfrac{(\alpha_1,\alpha_5)}{\sqrt{(\alpha_1,\alpha_1)
    (\alpha_5,\alpha_5)}}=0,
    \\
    \cos\theta_{23}&=\dfrac{(\alpha_2,\alpha_3)}{\sqrt{(\alpha_2,\alpha_2)
    (\alpha_3,\alpha_3)}}=-\dfrac{1}{2},\quad
    \cos\theta_{24}=\dfrac{(\alpha_2,\alpha_4)}{\sqrt{(\alpha_2,\alpha_2)
    (\alpha_4,\alpha_4)}}=0,
    \\
    \cos\theta_{25}&=\dfrac{(\alpha_2,\alpha_5)}{\sqrt{(\alpha_2,\alpha_2)
    (\alpha_5,\alpha_5)}}=0,\quad
    \cos\theta_{34}=\dfrac{(\alpha_3,\alpha_4)}{\sqrt{(\alpha_3,\alpha_3)
    (\alpha_4,\alpha_4)}}=-\dfrac{1}{2},
    \\
     \cos\theta_{35}&=\dfrac{(\alpha_3,\alpha_5)}{\sqrt{(\alpha_3,\alpha_3)
    (\alpha_5,\alpha_5)}}=0,\quad
     \cos\theta_{45}=\dfrac{(\alpha_4,\alpha_5)}{\sqrt{(\alpha_4,\alpha_4)
    (\alpha_5,\alpha_5)}}=-\dfrac{1}{2}.
\end{align*}
so that we can draw the Dynkin diagram.
\end{proof}

\section{The group $ E_{7,\sH} $ and the root system, the Dynkin
diagram of the Lie algebra $ ({\mathfrak{e}_{7}}^C)^{\varepsilon_1,
\varepsilon_2} $}

Let $ \mathfrak{P}^C $ be the Freudenthal $ C $-vector space. We consider a $ 32 $ dimensional $ C $-vector space $ (\mathfrak{P}_{\sH})^C $ which is defined by replacing $ \mathfrak{C} $ with $ \H $ in $ \mathfrak{P}^C $:
\begin{align*}
(\mathfrak{P}_{\sH})^C:=(\mathfrak{J}_{\sH})^C \oplus (\mathfrak{J}_{\sH})^C \oplus C \oplus C.
\end{align*}

Then we consider the following complex Lie algebra $ (\mathfrak{e}_{7,\sH})^C $
which is given by replacing $ \mathfrak{C} $ with $ \H $ in the complex Lie
algebra $ {\mathfrak{e}_7}^C $:
\begin{align*}
(\mathfrak{e}_{7,\sH})^C=\left\lbrace \varPhi(\phi,A,B,\nu) \in \Hom_{C}((\mathfrak{P}_{\sH})^C) \relmiddle{|} \phi \in (\mathfrak{e}_{6,\sH})^C, A, B \in (\mathfrak{J}_{\sH})^C, \nu \in C \right\rbrace.
\end{align*}
In particular, we have $ \dim_C((\mathfrak{e}_{7,\sH})^C)=35+15\times
2+1=66 $. In $ (\mathfrak{e}_{7,\sH})^C $, we can define a Lie bracket and
can prove its simplicity as in $ {\mathfrak{e}_7}^C $.
In addition, we define a symmetric inner product $ (\varPhi_1,\varPhi_2)_7 $ by
\begin{align*}
(\varPhi_1,\varPhi_2)_7:=-2(\phi_1,\phi_2)_6-4(A_1,B_2)-4(A_2,B_1)-\dfrac{8}{3}\nu_1\nu_2,
\end{align*}
where $ \varPhi_i:=\varPhi(\phi_i,A_i,B_i,\nu_i) \in (\mathfrak{e}_{7,
\sH})^C,i=1,2 $. Then the symmetric inner product $ (\varPhi_1,\varPhi_2)_7 $ is $ (\mathfrak{e}_{7,\sH})^C $-adjoint invariant (cf. \cite[Lemma 4.5.1 (1)]{iy0}).

Then we prove the following theorem needed in Section 8.

\begin{theorem}\label{theorem 7.1}
    The killing form $ B_{7,\sH} $ of $ (\mathfrak{e}_{7,\sH})^C $ is given by
    \begin{align*}
    B_{7,\sH}(\varPhi_1,\varPhi_2)=-5(\varPhi_1,\varPhi_2)_7.
    \end{align*}
\end{theorem}
\begin{proof}
    Since $ (\mathfrak{e}_{7,\sH})^C $ is simple, there exist
    $ k \in C $ such that
    \begin{align*}
    B_{7,\sH}(\varPhi_1, \varPhi_2)=k (\varPhi_1, \varPhi_2)_7,\;\; \varPhi_i \in (\mathfrak{e}_{7,\sH})^C.
    \end{align*}
    We will determine $ k $. Let $ \varPhi_1=\varPhi_2=\varPhi(0,0,0,1) $. Then we have
    \begin{align*}
    (\varPhi_1, \varPhi_2)_7=-\dfrac{8}{3}.
    \end{align*}
     On the other hand, if follows from
     \begin{align*}
     (\ad \varPhi_1)(\ad \varPhi_2)\varPhi(\phi,A,B,\nu)&=[\varPhi_1,[\varPhi_2, \varPhi(\phi,A,B,\nu)]]
     \\
     &=[\varPhi_1,\varPhi(0,\dfrac{2}{3}A,-\dfrac{2}{3}B,0)]
     \\
     &=\varPhi(0,\dfrac{4}{9}A,\dfrac{4}{9}B,0)
     \end{align*}
     that
   \begin{align*}
      B_{7,\sH}(\ad \varPhi_1\ad \varPhi_2)=\tr(\ad \varPhi_1\ad \varPhi_2)=\dfrac{4}{9}\times 15 \times 2=\dfrac{40}{3}.
   \end{align*}
 Thus we have $ k=-5 $.

Therefore we have
\begin{align*}
    B_{7,\sH}(\varPhi_1,\varPhi_2)=-5(\varPhi_1,\varPhi_2)_7.
\end{align*}

\end{proof}

We will study a following Lie group $ E_{7,\sH} $ which is
defined by replacing $ \mathfrak{C} $ with $ \H $ in the compact Lie group
$ E_7 $:
\begin{align*}
    E_{7,\sH}:=\left\lbrace \alpha \in
    \Iso_{C}((\mathfrak{P}_{\sH})^C)\relmiddle{|} \alpha (P\times
    Q)\alpha^{-1}=\alpha P \times \alpha Q, \langle\alpha P,\alpha
    Q\rangle=\langle P,Q \rangle \right\rbrace,
\end{align*}
where $ \langle P,Q \rangle:=(\tau P,Q)=\{\tau\lambda P,Q\}  $.
The group $ E_{7,\sH} $ is a compact Lie group as a closed subgroup of the unitary group $ U(32)=U((\mathfrak{P}_{\sH})^C)=\left\lbrace  \alpha \in \Iso_{C}((\mathfrak{P}_{\sH})^C)\,|\, \langle\alpha P,\alpha Q
\rangle=\langle P,Q \rangle \right\rbrace $.

\begin{lemma}\label{lemma 7.2}
The Lie algebra $ \mathfrak{e}_{7,\sH} $ of the group $ E_{7,\sH} $ is
given by
\begin{align*}
\mathfrak{e}_{7,\sH}=\left\lbrace \varPhi(\phi,A,-\tau A,\nu) \relmiddle{|}
\phi \in \mathfrak{e}_{6,\sH},A \in (\mathfrak{J}_{\sH})^C,\nu \in
i\R \right\rbrace .
\end{align*}
\end{lemma}

In particular, we have $ \dim(\mathfrak{e}_{7,\sH})=35+30+1 =66$.
\begin{proof}
As in $ {\mathfrak{e}_7}^C $, the complex Lie algebra $ (\mathfrak{e}_{7,
\sH})^C $ of the group $ (E_{7,\sH})^C $ is given by
\begin{align*}
(\mathfrak{e}_{7,\sH})^C=\left\lbrace \varPhi(\phi,A, B,\nu) \relmiddle{|}
\phi \in (\mathfrak{e}_{6,\sH})^C, A,B \in (\mathfrak{J}_{\sH})^C,\nu \in
C \right\rbrace,
\end{align*}
where the group $ (E_{7,\sH})^C $ is defined by
\begin{align*}
(E_{7,\sH})^C =\left\lbrace \alpha \in
    \Iso_{C}((\mathfrak{P}_{\sH})^C)\relmiddle{|} \alpha (P\times
    Q)\alpha^{-1}=\alpha P \times \alpha Q  \right\rbrace.
\end{align*}
Hence, using $ ((E_{7,\sH})^C )^{\tau\lambda}=E_{7,\sH} $ shown in
Proposition \ref{proposition 7.3} (1), we obtain the required result.
\end{proof}


The immediate aim is to prove the connectedness of $ E_{7,\sH} $.
We explain its procedure. First, since we have the isomorphism $ E_{6,
\sH} \cong SU(6)/\Z_2 $ (Theorem \ref{theorem 6.1}), the connectedness of $ E_{6,\sH} $ follows from
its result. After this, we will prove that the homogeneous space $ E_{7,
\sH}/E_{6,\sH} $ is homeomorphic to a space $ (\mathfrak{M}_{\sH})_1 $,
so that the connectedness of $ E_{7,\sH} $ is finally shown, where the
space $ (\mathfrak{M}_{\sH})_1 $ is defined later.

Now, we will make some preparations.
First, we define a space $ (\mathfrak{M}_{\sH})^C $ by
\begin{align*}
    (\mathfrak{M}_{\sH})^C :
    &=\left\lbrace
    P \in (\mathfrak{P}_{\sH})^C\relmiddle{|}
    P \times P=0, P\not=0  \right\rbrace
    \\
    &=\left\lbrace P=(X,Y,\xi,\eta)\relmiddle{|}
    \begin{array}{l}
        X \vee Y=0,X \times X=\eta Y \\
        Y \times Y=\xi X,(X,Y)=3\xi\eta
    \end{array}\right\rbrace.
\end{align*}
Here, for example, the following elements of $ (\mathfrak{M}_{\sH})^C $
\begin{align*}
    (X,\dfrac{1}{\eta}(X \times X), \dfrac{1}{\eta^2}\det\,X, \eta),\;\;
    (\dfrac{1}{\xi}(Y \times Y), Y, \xi, \dfrac{1}{\xi^2}\det\,Y)\;\;
    \dot{1}=(0,0,1,0),\;\;\text{\d{$ 1 $}}=(0,0,0,1)
\end{align*}
belong to $ (\mathfrak{M}_{\sH})^C $, where $ \xi\not=0, \eta\not=0 $.
Obviously, the group $ E_{7,\sH} $ acts on $ (\mathfrak{M}_{\sH})^C $.
\vspace{1mm}

Here, we will study the following subgroup $ (E_{7,\sH})_{\dot{1}} $ of $
E_{7,
\sH} $:
\begin{align*}
(E_{7,\sH})_{\dot{1}}:=\left\lbrace \alpha \in E_{7,\sH} \relmiddle{|}
\alpha \dot{1}=\dot{1} \right\rbrace.
\end{align*}

\begin{proposition}\label{proposition 7.3}
{\rm (1)} The group $ E_{7,\sH} $ coincides with the group $ ((E_{7,
\sH})^C)^{\tau\lambda} ${\rm :}$ E_{7,\sH}\!=((E_{7,\sH})^C)^{\tau\lambda}
$\!.

{\rm (2)} If $ \alpha \in E_{7,\sH} $ satisfies $ \alpha
\dot{1}
=\dot{1} $, then $ \alpha $ also satisfies $ \alpha
\text{\d{$ 1 $}}=\text{\d{$ 1 $}} $ and conversely.
\end{proposition}
\begin{proof}
(1)
Let $ \alpha \in E_{7,\sH} $. Then, note that $ \{\alpha P, \alpha Q\}=
\{P,Q\}, P,Q \in (\mathfrak{P}_{\sH})^C $, so that it follows from
\begin{align*}
(P,\lambda Q)=\{P,Q \}=\{\alpha P,\alpha Q \}=(\alpha P,\lambda\alpha
Q)=(P,{}^t\alpha\lambda\alpha Q)
\end{align*}
that $ \lambda={}^t\alpha\lambda\alpha $, that is, $ {}^t\alpha^{-1}
=\lambda\alpha\lambda^{-1} $. In addition, we have $ {}^t\alpha \tau
\alpha=\tau $ from $ \langle \alpha
P,\alpha Q \rangle=\langle P,Q \rangle $. Hence, together with the formula
$ {}^t\alpha^{-1}=\lambda\alpha\lambda^{-1} $ above, that is, $
{}^t\alpha=\lambda\alpha^{-1}\lambda^{-1} $, we have $(\tau\lambda)
\alpha=\alpha(\tau\lambda) $, so that $ E_{7,\sH} \subset ((E_{7,
\sH})^C)^{\tau\lambda} $.

Conversely, let $ \beta \in ((E_{7,\sH})^C)^{\tau\lambda} $. Then it
follows from $ \langle P, Q \rangle=\{\tau\lambda P,Q \} $ that
\begin{align*}
  \langle \beta P,\beta Q \rangle=\{(\tau\lambda)\beta P,\beta Q\}=
\{\beta(\tau\lambda)P,\beta Q\}=\{\tau\lambda P,Q\}=\langle P, Q \rangle.
\end{align*}
Hence we have $ ((E_{7,\sH})^C)^{\tau\lambda} \subset E_{7,\sH} $. Thus we
have $ E_{7,\sH}=((E_{7,\sH})^C)^{\tau\lambda} $.

(2) By using the result of (1) above, it follows that
$\alpha \text{\d{$ 1 $}}=\alpha (\tau\lambda)(-\dot{1})=(\tau\lambda)
\alpha(-\dot{1})=\tau\lambda (-\dot{1})=\text{\d{$ 1 $}}$.
\end{proof}
Then we have the following theorem.

\begin{proposition}\label{proposition 7.4}
   The group $ (E_{7,\sH})_{\dot{1}} $ is isomorphic to the group $
   E_{6,\sH} ${\rm :} $ (E_{7,\sH})_{\dot{1}} \cong E_{6,\sH}$.

In particular, the group $ (E_{7,\sH})_{\dot{1}} $ is connected.
\end{proposition}
\begin{proof}
We associate an element $ \beta \in E_{6,\sH} $  with the element
  \begin{align*}
      \tilde{\beta} =
      \begin{pmatrix}
          \beta  &  0 &  0  & 0  \\
          0 & \tau\beta\tau & 0 & 0 \\
          0 & 0 & 1 & 0 \\
          0 & 0 & 0 & 1
      \end{pmatrix}
      \in (E_{7,\sH})_{\dot{1}}.
  \end{align*}
  Then, for $ P=(X,Y,\xi,\eta), Q=(Z,W,\zeta,\omega) \in
  (\mathfrak{P}_{\sH})^C$,
  using the formula $ \beta X \vee \tau\beta\tau W=\beta(X \vee W)
  \beta^{-1} $, it follows that
  \begin{align*}
      \tilde{\beta} P \times \tilde{\beta} Q&=(\beta X,
      \tau\beta\tau Y, \xi,\eta) \times (\beta Z,\tau\beta\tau W,\zeta,
      \omega)
      \\
      &=\left(
      \begin{array}{l}
         (-1/2)(\beta X \vee \tau\beta\tau W+\beta Z \vee \tau\beta\tau Y)
         \\[1mm]
        (-1/4)(2\tau\beta\tau Y \times \tau\beta\tau W-\xi\beta
        Z-\zeta\beta X)
        \\[1mm]
        (1/4)(2\beta X \times \beta Z-\eta\tau\beta\tau W-\omega
        \tau\beta\tau Y)
        \\[1mm]
        (-1/8)((\beta X, \tau\beta\tau W)+(\beta Z, \tau\beta\tau
        Y)-3(\xi\omega+\zeta\eta))
      \end{array}
      \right)
      \\[1mm]
      &=\left(
       \begin{array}{l}
           (-1/2)(\beta (X \vee W)\beta^{-1}+\beta (Z \vee Y)
           \beta^{-1}
           \\[1mm]
           (-1/4)(\beta (2Y \times  W-\xi Z-\zeta X))
           \\[1mm]
           (1/4)(\tau\beta\tau(2X \times Z-\eta W-\omega Y))
           \\[1mm]
           (-1/8)((X, W)+(Z, Y)-3(\xi\omega+\zeta\eta))
       \end{array}
      \right).
  \end{align*}
Here, set $ P \times Q=:\varPhi(\phi,A,B,\nu) $, where
\begin{align*}
    &\phi=(-1/2)(X \vee W+Z \vee Y),\;A=(-1/4)(2Y \times W-\xi Z=\zeta
    X),
    \\
    & B=(1/4)(2X\times Z-\eta W-\omega Y),\;\nu=(1/8)((X,W)+
    (Z,Y)-3(\xi\omega+\zeta\eta)),
\end{align*}
then we have $ \tilde{\beta} P \times \tilde{\beta} Q
=\varPhi(\beta\phi\beta^{-1},\beta A,\tau\beta\tau B,\nu) $.

On the other hand, for $ R:=(S,T,\varsigma,\varrho) \in
(\mathfrak{P}_{\sH})^C $, using the formula $ (\tau\beta\tau){}
^t\phi(\tau\beta^{-1}\tau)={}^t(\beta\phi\beta^{-1}) $, it
follows that
  \begin{align*}
      \tilde{\beta}(P \times Q)\tilde{\beta}^{-1}R
      &=\tilde{\beta}(P \times Q)(\beta^{-1} S,
       \tau\beta^{-1}\tau T, \varsigma,\varrho)
      \\
      &=:\tilde{\beta}\varPhi(\phi,A,B,\nu)(\beta^{-1} S,\tau\beta^{-1}\tau
      T, \varsigma,\varrho)\;\;(P \times Q=:\varPhi(\phi,A,B,
      \nu))
      \\
      &=\beta\left(
      \begin{array}{c}
        \phi(\beta^{-1}S)-(1/3)\nu(\beta^{-1}S)+2B \times (\tau\beta^{-1}
        \tau T)+\varrho A \\
        2A \times (\beta^{-1}S)-{}^t\phi(\tau\beta^{-1}\tau T)+(1/3)
        \nu(\tau\beta^{-1}\tau T)+\varsigma B \\
        (A,\tau\beta^{-1}\tau T)+\nu\varsigma \\
        (B,\beta^{-1}S)-\nu\varrho
      \end{array}\right)
      \\
       &=\left(
       \begin{array}{c}
          \beta\phi\beta^{-1}S-(1/3)\nu S+\beta(2B \times (\tau\beta^{-1}
          \tau T))+\varrho \beta A
          \\
          \tau\beta\tau(2A \times (\beta^{-1}S))-(\tau\beta\tau){}
         ^t\phi(\tau\beta^{-1}\tau )T+(1/3)\nu T+\varsigma {}^t\beta^{-1} B
          \\
           (\beta A, T)+\nu\varsigma
          \\
           (\tau\beta\tau B,S)-\nu\varrho
       \end{array}\right)
       \\
       &=\left(
       \begin{array}{c}
           \beta\phi\beta^{-1}S-(1/3)\nu S+2(\tau\beta\tau) B \times T+
           \varrho \beta A \\
           2\beta A \times S-{}^t(\beta\phi\beta^{-1}) T+(1/3)\nu T+
           \varsigma (\tau\beta\tau) B \\
           (\beta A, T)+\nu\varsigma \\
           ((\tau\beta\tau)B,S)-\nu\varrho
       \end{array}\right)
    \\
    &=\varPhi(\beta\phi\beta^{-1},\beta A,(\tau\beta\tau) B,\nu)(S,T,
    \varsigma,\varrho),
  \end{align*}
that is, $ \tilde{\beta}(P \times Q)\tilde{\beta}^{-1}
=\varPhi(\beta\phi\beta^{-1},\beta A,{}^t\beta^{-1}B,\nu) $.

\noindent Hence we have $ \tilde{\beta} P \times \tilde{\beta} Q
=\tilde{\beta}(P \times Q)\tilde{\beta}^{-1} $. Moreover, it is
easy to verify that $ \langle\tilde{\beta} P,\tilde{\beta} Q\rangle=\langle
P,Q \rangle $. Indeed, for $ P=(X,Y,\xi,\eta), Q=(Z,W,\zeta,\omega) \in
(\mathfrak{P}_{\sH})^C $, note that $ \tau\beta\tau \in E_{7,\sH} $, it
follows that
\begin{align*}
\langle\tilde{\beta} P,\tilde{\beta} Q\rangle
&=\langle (\beta X,(\tau\beta\tau)Y,\xi,\eta), (\beta Z,(\tau\beta\tau)W,
\zeta,\omega) \rangle
\\
&=\langle \beta X,\beta Z \rangle +\langle (\tau\beta\tau)Y,
(\tau\beta\tau)W \rangle +(\tau\xi)\eta +(\tau\eta)\omega
\\
&=\langle X, Z \rangle +\langle Y,W \rangle +(\tau\xi)\eta +(\tau\eta)
\omega
\\
&=\langle P,Q \rangle,
\end{align*}
so that we have $ \tilde{\beta} \in E_{7,\sH} $. Moreover, since it is
trivial that $ \tilde{\beta}\dot{1}=\dot{1} $, we obtain $ \tilde{\beta}
\in (E_{7,\sH})_{\dot{1}} $.

Conversely, let $ \alpha  \in (E_{7,\sH})_{\dot{1}} $. Then $ \alpha $ is
of the form
\begin{align*}
\alpha=
\begin{pmatrix}
    \beta  &  \varepsilon  & 0  &  0 \\
    \delta & \beta' & 0 & 0     \\
    0 & 0 & 1 & 0               \\
    0 & 0 & 0 & 1
\end{pmatrix},\;\;\beta,\beta',\delta,\varepsilon \in
\Hom_{C}((\mathfrak{J}_{\sH})^C).
\end{align*}
Indeed, set
\begin{align*}
\alpha=
\begin{pmatrix}
    \beta & \varepsilon &  C  &  A   \\
    \delta & \beta' & B & D     \\
    c & a & \nu & \lambda       \\
    b & d & \kappa & \mu
\end{pmatrix},\;\;
\begin{array}{l}
 \beta,\beta',\delta,\varepsilon \in \Hom_{C}((\mathfrak{J}_{\sH})^C), \\
 a,b.c,d \in \Hom_{C}((\mathfrak{J}_{\sH})^C,C), \\
 A,B,C,D \in (\mathfrak{J}_{\sH})^C, \\
 \nu,\mu,\lambda,\kappa \in C.
\end{array}
\end{align*}
Then, since $ \alpha \dot{1}=\dot{1} $ implies $ \alpha \text{\d{$ 1 $}}
=\text{\d{$ 1 $}} $ (Lemma \ref{lemma 7.2}(2)), it follows that
\begin{align*}
\dot{1}&=\alpha
\dot{1}=\dot{C}+\text{\d{$ B $}}+\dot{\nu}+\text{\d{$ \kappa $}}
\\
\text{\d{$ 1 $}}&=\alpha \text{\d{$ 1 $}}=\dot{A}+
\text{\d{$ D $}}+\dot{\lambda}+\text{\d{$ \mu $}}.
\end{align*}
Hence we have $ A=B=C=D=\bm{O}{(\text{$ \bm{O} $ is zero matrix})},
\kappa=\lambda=0 $ and $ \nu=\mu=1 $. Moreover, it follows from
\begin{align*}
    \{\alpha \dot{X}, \alpha \dot{1}\}=\{\alpha \dot{X}, \alpha
    \text{\d{$ 1 $}}\}=0,\;\;
    \{\alpha \text{\d{$ X $}}, \alpha \dot{1}\}=\{\alpha \text{\d{$ X
    $}} , \alpha \text{\d{$ 1 $}}\}=0
\end{align*}
that
\begin{align*}
0&=\{\alpha \dot{X}, \alpha \dot{1}\}=\{(\beta X, \delta X,c(X),
b(X)),(0,0,1,0)\}=-b(X),
\\
0&=\{\alpha \dot{X}, \alpha \text{\d{$ 1 $}}\}=\{(\beta X, \delta
X,c(X), b(X)),(0,0,0,1)\}=c(X),
\\
0&=\{\alpha \text{\d{$ 1 $}}, \alpha \dot{1}\}=\{(\varepsilon X,
\beta'
X,a(X), d(X)),(0,0,1,0)\}=-d(X),
\\
0&=\{\alpha \text{\d{$ 1 $}}, \alpha \text{\d{$ 1 $}}\}=\{(\varepsilon
X, \beta' X,a(X), d(X)),(0,0,0,1)\}=a(X),
\end{align*}
that is, $ -b(X)=c(X)=-d(X)=a(X)=0 $ hold for all $ X \in (\mathfrak{J}
_{\sH})^C $. Hence we have $ a=b=c=d=0 $. With above, $ \alpha $ is of the form above.

The group $ (E_{7,\sH})_{\dot{1}} $ act on the space
$ (\mathfrak{M}_{\sH})^C $, obviously. Hence,
for $ (X, \dfrac{1}{\eta}(X \times X), \dfrac{1}{\eta^2}\det X, \eta)
\in (\mathfrak{M}_{\sH})^C,\eta \not=0 $, since it follows that
\begin{align*}
    (*)\cdots  \alpha(X, \dfrac{1}{\eta}(X \times X), \dfrac{1}
    {\eta^2}\det X, \eta)=(\beta X+\dfrac{1}{\eta}\varepsilon(X \times
    X),
    \delta X+\dfrac{1}{\eta}\beta'(X \times X), \dfrac{1}{\eta^2}\det
    X,
    \eta) \in (\mathfrak{M}_{\sH})^C,
\end{align*}
the following formula
\begin{align*}
    (\beta X+\dfrac{1}{\eta}\varepsilon(X \times X))\times (\beta X+
    \dfrac{1}{\eta}\varepsilon(X \times X))=\eta(\delta X+\dfrac{1}
    {\eta}
    \beta'(X \times X))
\end{align*}
holds for all $ \eta \in C, \eta \not=0 $.  Hence the formula above
holds for $ \eta=1,-1,2,-2,1/4 $, so that we have that $ \delta X=0 $
holds for all $ X \in (\mathfrak{J}_{\sH})^C $. Thus we obtain $ \delta=0
$.
Moreover, for $ (\dfrac{1}{\xi}(Y \times Y),Y,\xi, \dfrac{1}{\xi^2}\det Y)
\in (\mathfrak{M}_{\sH})^C, \xi \not=0 $, by doing a similar computation
above, we have that the formula $ \beta' Y \times \beta' Y=\beta(Y \times
Y)+\xi\varepsilon (Y) $ holds for all $ \xi \in C, \xi \not=0 $. Hence
the formula above holds for $ \xi=1,-1 $, so that we have that $
\varepsilon (Y)=0 $ holds for all $ Y \in (\mathfrak{J}_{\sH})^C $. Thus
we obtain $ \varepsilon=0 $. Therefore $ \alpha $ is of the form
\begin{align*}
    \alpha=
    \begin{pmatrix}
        \beta  & 0 & 0 & 0 \\
        0 & \beta' & 0 & 0     \\
        0 & 0 & 1 & 0               \\
        0 & 0 & 0 & 1
    \end{pmatrix},\;\;\beta,\beta'\in \Hom_{C}((\mathfrak{J}_{\sH})^C).
\end{align*}

\noindent Again, in the formula $ (*) $ above, set $ \eta=1 $. Then
the condition $ \alpha(X,X \times X,\det\,X,1) \in (\mathfrak{M}_{\sH})^C$
implies
\begin{align*}
    \left\lbrace
    \begin{array}{l}
        \beta X \times \beta X=\beta'(X \times X)
        \\[1mm]
        (\beta X,\beta'(X \times X))=3\det X.
    \end{array}
    \right.
\end{align*}
Hence it follows from two formulas above
that
\begin{align*}
    3\det \, (\beta X)=(\beta X,\beta X \times \beta X)=(\beta
    X,\beta'(X \times X))=3\det \, X,
\end{align*}
that is, $ \det\,(\beta X)=\det\,X $. Moreover, $ \langle \alpha \dot{X},
\alpha \dot{X'} \rangle =\langle \dot{X}, \dot{X'} \rangle$ implies $
\langle \beta X, \beta X'\rangle=\langle X, X'\rangle $. Thus we have $
\beta \in E_{6,\sH} $.

Next, as for $ \beta' $, using the formula $ \beta X \times \beta
X=\beta'(X \times X) $, it follows from $ \beta \in E_{6, \sH} $ that
\begin{align*}
    \beta'(X \times X)=\beta X \times \beta X=\tau\beta\tau(X \times X),
\end{align*}
that is $ \beta'(X \times X)=\tau\beta\tau(X \times X) \cdots(**) $.

Therefore we can obtain
\begin{align*}
    \beta'=\tau\beta\tau.
\end{align*}
Indeed, set $ X \times X $ instead of $ X $ in the formula $ (**) $,
then we have the following formula
\begin{align*}
    (\det\, X)\beta' X=(\det\,X)\tau\beta\tau X.
\end{align*}
In the case where $ \det\,X \not=0$, we have $ \beta'X=\tau\beta\tau X $
for all $ X \in (\mathfrak{J}_{\sH})^C $, that is $ \beta'=\tau\beta\tau
$.
In the case $ \det\,X=0 $. We consider a space $ S:=\left\lbrace X \in
(\mathfrak{J}_{\sH})^C \relmiddle{|} \det\, X \not=0 \right\rbrace  $.
Then it follows from $ \ov{S}=(\mathfrak{J}_{\sH})^C $ that $ S $ is dense
in $ (\mathfrak{J}_{\sH})^C $, and since $ \beta', \tau\beta\tau \in
\Hom_{C}((\mathfrak{J}_{\sH})^C) $  are continuous mappings and agree on $
S $,these also agree on $ (\mathfrak{J}_{\sH})^C $. Namely, $
\beta'=\tau\beta\tau $ holds even if $ \det\,X=0 $.

Finally, the connectedness of the group $ (E_{7,\sH})_{\dot{1}} $ follows
from the connectedness of the group $ E_{6,\sH} $.

Consequently, the proof of this proposition is completed.
\end{proof}

Here, we prove the useful lemma in order to prove Theorem \ref{theorem 7.6}
below.

\begin{lemma}\label{lemma 7.5}
Any element $ P \in (\mathfrak{M}_{\sH})^C $ can be transformed to a
diagonal form by some element $ \alpha \in (E_{7,\sH})_0 ${\rm :}
\begin{align*}
\alpha P=(X,Y,\xi,\eta),\;\; X,Y {\text{{\rm :} diagonal forms}}, \xi >0.
\end{align*}
\end{lemma}
\begin{proof}
Before proving this lemma, we prove the following claim.
\vspace{-3mm}
\begin{quote}
\begin{claim}\label{claim}
Any element $ P \in (\mathfrak{J}_{\sH})^C $ can be transformed to a
diagonal form by some element $ \beta \in E_{6,\sH} ${\rm :}
\begin{align*}
\beta X=\begin{pmatrix}
\xi_1 & 0 & 0 \\
0 & \xi_2 & 0 \\
0 & 0 & \xi_3
\end{pmatrix},\;\; \xi_i \in C.
\end{align*}
\end{claim}
\begin{proof}
The proof of this claim can be proved in exactly the same way as the proof
of \cite[Proposition 3.8.2]{iy0}, so its proof is omitted.
\end{proof}
\end{quote}

We start to prove the proof and will first show that $ P $ can be
transformed to a diagonal form with $ \xi \not=0 $.

Let $ P=(X,Y,\xi,\eta) \in (\mathfrak{M}_{\sH})^C $.

Case (i) where $ P=(X,Y,\xi,\eta), \xi\not=0 $. Then we choose $ \beta \in
E_{6,\sH} $ such that $ \tau\beta\tau Y $ is a diagonal form (Claim). Hence
since it follows from $ X=(1/\xi)(Y \times Y) $ that
\begin{align*}
\beta X=\dfrac{1}{\xi}\beta (Y \times Y)=\dfrac{1}{\xi}(\tau\beta\tau Y
\times \tau\beta\tau Y),
\end{align*}
$ \beta X $ is also a diagonal form. Thus $ P $ is a diagonal form.

Case (ii) where $ P=(X,Y,0,\eta), Y \not=0 $. As in the Case (i), we choose
$ \beta \in
E_{6,\sH} $ such that $ \tau\beta\tau Y $ is a diagonal form:
\begin{align*}
\tau\beta\tau Y:=\diag(\eta_1,\eta_2,\eta_3),\;\;\eta_i \in C.
\end{align*}
Hence, from $ Y \not=0 $, some $ \eta_i $ is non-zero: $ \eta_i \not=0 $.
For this $ \eta_i \in C $, we choose $ \alpha_i(a):=\exp(\varPhi(0,aE_i,
\allowbreak-\tau aE_i,0)) \in (E_{7,\sH})_0, a \in C, a \not=0 $ (Lemma
\ref{lemma 7.2}). The action of $ \alpha_i(a) $ to $ (\mathfrak{P}_{\sH})^C
$ is given as follows:
\begin{align*}
 \alpha_i(a)(X,Y,\xi,\eta)
=\left(
\begin{array}{c}
(1+(\cos|a|-1)p_i)X-2\tau a\dfrac{\sin|a|}{|a|}E_i \times Y+\eta
a\dfrac{\sin|a|}{|a|}E_i
\\
2a\dfrac{\sin|a|}{|a|}E_i \times X+(1+(\cos|a|-1)p_i)Y-\xi\tau
a\dfrac{\sin|a|}{|a|}E_i
\\
(a\dfrac{\sin|a|}{|a|}E_i,Y)+\xi \cos|a|
\\
(-\tau a\dfrac{\sin|a|}{|a|}E_i,X )+\eta \cos|a|
\end{array}\right),
\end{align*}
where $ p_i:(\mathfrak{J}_{\sH})^C \to (\mathfrak{J}_{\sH})^C $ is defined by $ p_i(X)=(X,E_i)E_i+4E_i\times (E_i \times X) $
(cf. \cite[Lemma 4.3.4]{iy7}).

Apply $ \alpha_i(-\pi/2) $ on $ \beta
P=(\beta X, \tau\beta\tau Y,0,\eta) $, then we have the following
\begin{align*}
\alpha_i\left(-\dfrac{\pi}{2}\right)(\beta P)&=\alpha_i\left(-\dfrac{\pi}
{2}\right)(\beta X, \tau\beta\tau Y,0,\eta)
\\
&=\left(
\begin{array}{c}
(1-p_i)(\beta X)+2E_i \times (\tau\beta\tau Y)-\eta E_i \\
-2E_i \times (\beta X)+(1-p_i)(\tau\beta\tau Y) \\
(-E_i,\tau\beta\tau Y) \\
(E_i,\beta X)
\end{array}
\right)
\\
&=\left(
\begin{array}{c}
 * \\
 * \\
-\eta_i \\
 *
\end{array}\right),\,\,\eta_i\not=0.
\end{align*}
Hence this case reduces to Case (i).

Case (iii) where $ P=(X,0,0,\eta), X \not=0 $. We choose $ \beta \in E_{6,
\sH} $ such that $ \beta X=\diag(\xi_1,\xi_2,\xi_3), \allowbreak \xi_i \in C $. Then
some $ \xi_i $ is non-zero from $ \beta X\not=0 $: $ \xi_i\not=0 $. Apply $
\alpha_{i+1}(-\pi/2) $ on $ \beta P=(\beta X,0,0,\eta) $, then we have the
following
\begin{align*}
\alpha_{i+1}\left( -\dfrac{\pi}{2}\right)(\beta P)
&=\alpha_{i+1}\left( -
\dfrac{\pi}{2}\right)(\beta X,0,0,\eta)
\\
&=\alpha_{i+1}\left(-\dfrac{\pi}{2}\right)(\xi_iE_i+\xi_{i+1}E_{i+1}+
\xi_{i+2}E_{i+2},0,0,
\eta)
\\
&=((1-p_{i+1})(\xi_iE_i+\xi_{i+1}E_{i+1}+
\xi_{i+2}E_{i+2})-\eta E_{i+1},
\\
&\quad -2E_{i+1}\times (\xi_iE_i+\xi_{i+1}E_{i+1}+
\xi_{i+2}E_{i+2}),0, (E_{i+1},\xi_iE_i+\xi_{i+1}E_{i+1}+
\xi_{i+2}E_{i+2}) )
\\
&=(*,-\xi_iE_{i+2}-\xi_{i+2}E_i,0,*),\;\;\xi_i\not=0.
\end{align*}
Hence this case reduces to Case (ii).

Case (iv) where $ P=(0,0,0,\eta), \eta \not=0 $.
We choose $ \alpha_1(a) \in (E_{7,\sH})_0 $, set $ a=-\pi/2 $. Apply $
\alpha_1(-\pi/2) $ on $ P $, then we have the following
\begin{align*}
\alpha_1\left(-\dfrac{\pi}{2} \right)P&=\alpha_1\left(-\dfrac{\pi}{2}
\right)(0,0,0,\eta)
\\
&=(\eta E_1,0,0,0),\;\;\eta\not=0.
\end{align*}
Hence this case reduces to Case (iii).

With above, it was shown that  $ P $ can be transformed to a diagonal form
with $ \xi \not=0 $. Finally, for $ \theta \in U(1):=\{\theta \in C \,|\,
(\tau\theta)\theta=1\} $, we define a mapping
$ \phi:(\mathfrak{P}
_{\sH})^C \to (\mathfrak{P}_{\sH})^C $ by
\begin{align*}
\phi(\theta)(X,Y,\xi,\eta)=(\theta^{-1}X,\theta Y, \theta^3\xi,\theta^{-3}
\eta).
\end{align*}
Then we have $ \phi(\theta) \in E_{7,\sH} $. For a given $ P=(X,Y,\xi,\eta
) $ transformed to a diagonal form with $ \xi \not=0 $, since we can choose
some $ \theta \in U(1) $ such that $ \xi\theta^3 >0 $, we obtain $ P \in
(\mathfrak{M}_{\sH})^C $ of the required form. The proof is completed.
\end{proof}

We define a space $ ((\mathfrak{M}_{\sH})^C)_1  $ by
\begin{align*}
({\mathfrak{M}_{\sH}}^C)_1 :=\left\lbrace P \in (\mathfrak{M}_{\sH})^C
\relmiddle{|} \langle P, P \rangle=1 \right\rbrace.
\end{align*}

Then we have the following theorem.

\begin{theorem}\label{theorem 7.6}
The homogeneous space $ E_{7,\sH}/E_{6,\sH} $ is homeomorphic to the space
$ ((\mathfrak{M}_{\sH})^C)_1 ${\rm:} \allowbreak $ E_{7,\sH}/E_{6,\sH}
\allowbreak \simeq ((\mathfrak{M}_{\sH})^C)_1 $.

In particular, the group $ E_{7,\sH} $ is connected.
\end{theorem}
\begin{proof}
For $\alpha \in E_{7,\sH}$ and $P \in ((\mathfrak{M}_{\sH})^C)_1$, we have
$\alpha P \in ((\mathfrak{M}_{\sH})^C)_1$.
Hence $E_{7,\sH}$ acts on $((\mathfrak{M}_{\sH})^C)_1$. We will prove that
the group $(E_7)_0$ acts transitively on $((\mathfrak{M}_{\sH})^C)_1$. In
order to prove this, it is sufficient to show that any
element $P \in ((\mathfrak{M}_{\sH})^C)_1$ can be transformed to $(0, 0, 1,
0) \in ((\mathfrak{M}_{\sH})^C)_1$ by
some $\alpha \in (E_{7,\sH})_0$. Now, $P \in((\mathfrak{M}_{\sH})^C)_1$ can
be transformed to a diagonal form by $\alpha \in (E_{7,\sH})_0$:
\begin{align*}
\alpha P=(\dfrac{1}{\xi}\begin{pmatrix}\eta_2\eta_3 & 0 & 0 \\
                                       0 & \eta_3\eta_1 & 0 \\
                                       0 & 0 & \eta_1 \eta_2
                         \end{pmatrix},
                \begin{pmatrix}\;\eta_1\; & \;0\; & \;0\; \\
                                         0 & \eta_2 & 0 \\
                                         0 & 0 & \eta_3
                \end{pmatrix},
           \xi, \dfrac{1}{\xi^2}\eta_1\eta_2\eta_3 ),\;\; \xi > 0
\end{align*}
(Lemma \ref{lemma 7.5}). From the condition $\langle \alpha P,\alpha P
\rangle = \langle P, P \rangle = 1,$  we have
\begin{align*}
 \dfrac{1}{\xi^2}(|\eta_2\eta_3|^2 + |\eta_3\eta_1|^2 + |\eta_1\eta_2|^2) +
(|\eta_1|^2 + |\eta_2|^2 + |\eta_3|^2) + \xi^2 + \dfrac{1}{\xi^4}|
\eta_1\eta_2\eta_3|^2 = 1,
\end{align*}
that is,
\begin{equation*}
\displaylines{\hfill
      \Big(1 + \dfrac{|\eta_1|^2}{\xi^2}\Big)\Big(1 +
      \frac{|\eta_2|^2}{\xi^2}\Big)\Big(1 + \dfrac{|\eta_3|^2}{\xi^2}\Big)
      = \dfrac{1}{\xi^2}.
\hfill{(*)}}
\end{equation*}
We choose $r_1, r_2, r_3 \in \R$, $0 \le r_i < \pi/2$ such that
 $  \tan r_i = |\eta_i|/\xi,\; i = 1, 2, 3 $,
then $ (*) $ becomes
$$
           \xi = \cos r_1\cos r_2\cos r_3.$$
By setting $  a_i :=( \eta_i/|\eta_i|)r_i,\; i = 1, 2, 3 $
(if $\eta_i = 0$, then $a_i$ means $0$), we have
\begin{align*}
 r_i=|a_i|,\;\; \eta_i =\dfrac{1}{|a_i|}\dfrac{\eta_i}{|\eta_i|}r_i
\dfrac{|\eta_i|}{\xi}\xi = \dfrac{a_i}{|a_i|}\tan r_i\cos r_1\cos r_2\cos
r_3.
\end{align*}
Hence we see that $\alpha P$ is of the form
\begin{align*}
&\quad  \alpha P
\\
&=\left(  \begin{array}{c}
\begin{pmatrix}
 \cos|a_1|a_2\dfrac{\sin|a_2|}{|a_2|}a_3\dfrac{\sin|a_3|}{|a_3|} & 0 & 0 \\
 0 & a_1\dfrac{\sin|a_1|}{|a_1|}\cos|a_2|a_3\dfrac{\sin|a_3|}{|a_3|} & 0 \\
 0 & 0 & a_1\dfrac{\sin|a_1|}{|a_1|}a_2\dfrac{\sin|a_2|}{|a_2|}\cos|a_3|
\end{pmatrix}
\vspace{1mm}\\
\begin{pmatrix}
a_1\dfrac{\sin|a_1|}{|a_1|}\cos|a_2|\cos|a_3| & 0 & 0 \\
0 & \cos|a_1|a_2\dfrac{\sin|a_2|}{|a_2|}\cos|a_3| & 0 \\
0 & 0 &\cos|a_1|\cos|a_2|a_3\dfrac{\sin|a_3|}{|a_3|}
\end{pmatrix}
\vspace{1mm}\\
         \cos|a_1|\cos|a_2|\cos|a_3|
\vspace{2mm}\\
        a_1\dfrac{\sin|a_1|}{|a_1|}a_2\dfrac{\sin|a_2|}{|a_2|}
        a_3\dfrac{\sin|a_3|}{|a_3|}
\end{array}\right)
\end{align*}
Thus, using $ \alpha_i(a) $ defined in the proof of Lemma \ref{lemma 7.5},
$ \alpha P $ is expressed by
\begin{align*}
\alpha P=\alpha_3(a_3)\alpha_2(a_2)\alpha_1(a_1)(0, 0, 1, 0),
\end{align*}
that is,
\begin{align*}
\alpha_1(a_1)^{-1}\alpha_2(a_2)^{-1}\alpha_3(a_3)^{-1}(\alpha P) = (0,
      0, 1, 0).
\end{align*}
This shows the transitivity of $(E_{7,\sH})_0$. Since we have $
((\mathfrak{M}_{\sH})^C)_1=(E_{7,\sH})_0(0, 0, 1, 0)$, $((\mathfrak{M}
_{\sH})^C)_1$ is connected. The group $E_{7,\sH}$ acts transitively on
$((\mathfrak{M}_{\sH})^C)_1$ and the isotropy subgroup of $E_{7,\sH}$ at $
\dot{1}=(0, 0, 1, 0) \in ((\mathfrak{M}_{\sH})^C)_1$ is
$E_{6,\sH}$ (Proposition \ref{proposition 7.3}).
Therefore we have the required homeomorphism
\begin{align*}
E_{7,\sH}/E_{6,\sH} \simeq ((\mathfrak{M}_{\sH})^C)_1.
\end{align*}

Finally, the connectedness of $E_{7,\sH}$ follows from the connectedness
of $((\mathfrak{M}_{\sH})^C)_1$ and $E_{6,\sH}$.
\end{proof}

Here, we prove the following lemma.

\begin{lemma}\label{lemma 7.7}
The group $ (E_{7,\sH})^C $ is an algebraic subgroup of the general linear group $ GL(32,C)\!=\Iso_{C}((\mathfrak{P}_{\sH})^C) $ and satisfies the condition that $ \alpha \in (E_{7,\sH})^C $ implies $ \alpha^*=(\tau\lambda)\alpha^{-1}(\lambda^{-1}\tau) \in (E_{7,\sH})^C  $, where $ \alpha^* $ is the conjugate transpose of $ \alpha $ with respect to the Hermite inner product $ \langle P,Q \rangle ${\rm :}$ \langle \alpha^* P,Q \rangle=\langle P,\alpha Q \rangle $.
\end{lemma}
\begin{proof}
First, it is easy to verify that $ \alpha^*=(\tau\lambda)\alpha^{-1}(\lambda^{-1}\tau) $. Indeed, for any $ P,Q \in (\mathfrak{P}_{\sH})^C $, it follows from
\begin{align*}
\langle \alpha^* P,Q \rangle&=\langle P,\alpha Q \rangle =\{\tau\lambda P, \alpha Q \}=\{\alpha \alpha^{-1}\tau\lambda P, \alpha Q \}
\\
&=\{\alpha^{-1}\tau\lambda P, Q \}=\{(\tau\lambda)(\lambda^{-1}\tau)\alpha^{-1}\tau\lambda P, Q \}
\\
&=\langle(\lambda^{-1}\tau)\alpha^{-1}(\tau\lambda) P, Q \rangle
\end{align*}
that $ \alpha^*=(\lambda^{-1}\tau)\alpha^{-1}(\tau\lambda) $,
that is, $ \alpha^*=(\tau\lambda)\alpha^{-1}(\lambda^{-1}\tau) $. Subsequently, we will show $ \alpha^* \in (E_{7,\sH})^C $. It is clear that $ \alpha^* \in \Iso_{C}((\mathfrak{P}_{\sH})^C) $ and it follows from $ \tau P \times \tau Q=\tau (P \times Q)\tau $ and $ \lambda \in (E_{7,\sH})^C $
that
\begin{align*}
\alpha^*P \times \alpha^* Q
&=(\tau\lambda)\alpha^{-1}(\lambda^{-1}\tau)P \times (\tau\lambda)\alpha^{-1}(\lambda^{-1}\tau)Q
\\
&=\tau(\lambda\alpha^{-1}(\lambda^{-1}\tau)P \times \lambda\alpha^{-1}(\lambda^{-1}\tau)Q)\tau
\\
&=\tau\lambda(\alpha^{-1}(\lambda^{-1}\tau)P \times \alpha^{-1}(\lambda^{-1}\tau)Q)\lambda^{-1}\tau
\\
&=\tau\lambda\alpha^{-1}((\lambda^{-1}\tau)P \times (\lambda^{-1}\tau)Q)\alpha\lambda^{-1}\tau
\\
&=(\tau\lambda)\alpha^{-1}(\lambda\tau)(P \times Q)(\tau\lambda^{-1})\alpha(\lambda^{-1}\tau)
\\
&=\alpha^*(P \times Q){\alpha^*}^{-1}.
\end{align*}
Hence we have $ \alpha^* \in (E_{7,\sH})^C $.

Finally, the group $ (E_{7,\sH})^C $ is defined by the algebraic relation $ \alpha P \times \alpha Q=\alpha(P \times Q)\alpha^{-1} $, so that $ (E_{7,\sH})^C $ is algebraic.
\end{proof}

Let $ O((\mathfrak{P}_{\sH})^C) $ be the orthogonal subgroup $ GL(32,C)=\Iso_{C}((\mathfrak{P}_{\sH})^C) $:
\begin{align*}
O(32,C)=O((\mathfrak{P}_{\sH})^C):=\left\lbrace \alpha \in \Iso_{C}((\mathfrak{P}_{\sH})^C) \relmiddle{|} \langle \alpha P,\alpha Q\rangle=\langle P,Q \rangle\right\rbrace.
\end{align*}

Then we have the following
\begin{align*}
(E_{7,\sH})^C \cap O((\mathfrak{P}_{\sH})^C)&=\left\lbrace \alpha \in \Iso_{C}((\mathfrak{P}_{\sH})^C) \relmiddle{|} \alpha P \times \alpha Q=\alpha(P \times Q)\alpha^{-1}, \langle \alpha P,\alpha Q\rangle=\langle P,Q \rangle\right\rbrace
\\
&=E_{7,\sH}.
\end{align*}

Using Chevalley's lemma (\cite[Lemma 2]{che}), we have a homeomorphism
\begin{align*}
(E_{7,\sH})^C &\simeq ((E_{7,\sH})^C \cap O((\mathfrak{P}_{\sH})^C)) \times \R^d
\\
&=E_{7,\sH} \times \R^d,
\end{align*}
where the dimension $ d $ of the Euclidian part is computed by Lemma \ref{lemma 7.2} as follows.
\begin{align*}
d=\dim((E_{7,\sH})^C)-\dim(E_{7,\sH})=66 \times 2-66=66.
\end{align*}

With above, we have the following theorem needed in the next section.

\begin{theorem}\label{theorem 7.8}
The group $ (E_{7,\sH})^C $ is homeomorphic to the topological product of the group $ E_{7,\sH} $ and a $ 66 $-dimensional Euclidian space $ \R^{66} $:
\begin{align*}
  (E_{7,\sH})^C \simeq E_{7,\sH} \times \R^{66}.
\end{align*}

In particular, the group $ (E_{7,\sH})^C $ is connected.
\end{theorem}
\begin{proof}
The first half has already been shown above. The connectedness of the group $ (E_{7,\sH})^C  $ follows from the connectedness of the group $ E_{7,\sH} $ (Theorem \ref{theorem 7.6}).
\end{proof}

Using the $ C $-linear transformations $ \varepsilon_1,\varepsilon_2 $
defined in previous section and the semi-linear transformation  $
\tau\lambda $ of $ (\mathfrak{P}_{\sH})^C $,
we define a $ 12 $-dimensional $ \R $-vector space $ V^{12} $ by
\begin{align*}
    V^{12}
    &:=\left\lbrace P \in \mathfrak{P}^C \relmiddle{|} \varepsilon_1
    P=-iP, \varepsilon_2(\tau\lambda) P=P \right\rbrace
    \\
    &=\left\lbrace P=(X,-\varepsilon_2\tau X,0,0)
    \relmiddle{|}X=F_1(x_1)+F_2(x_2)+F_3(x_3), x_i \in
    (\H^Ce_4)_{\varepsilon_1} \right\rbrace
    \\
    &=\left\lbrace P=(
    \begin{pmatrix}
     0 &  x_3 & \ov{x_2} \\
     \ov{x_3} & 0 & x_1 \\
     x_2 & \ov{x_1} & 0 \\
    \end{pmatrix},
    \begin{pmatrix}
  0 & -\varepsilon_2\tau x_3 & -\ov{\varepsilon_2\tau x_2} \\
  -\ov{\varepsilon_2\tau x_3} & 0 & -\varepsilon_2\tau x_1  \\
  -\varepsilon_2\tau x_2  & -\ov{\varepsilon_2\tau x_1} & 0
    \end{pmatrix},
    0,0)\relmiddle{|} x_i \in (\H^Ce_4)_{\varepsilon_1} \right\rbrace
\end{align*}
with the norm
\begin{align*}
    (P,P)_{\varepsilon_2}:=\dfrac{1}{8}\left\lbrace P,\varepsilon_2 P
    \right\rbrace ,
\end{align*}
where $ (\H^Ce_4)_{\varepsilon_1}:=\left\lbrace x \in \mathfrak{C}^C
\relmiddle{|} x=p(e_4+ie_5)+q(e_6-ie_7),p,q \in C\right\rbrace  $. Let $
x=p(e_4+ie_5)+q(e_6-ie_7) \in  (\H^Ce_4)_{\varepsilon_1} $, then note that
$ -\varepsilon_2\tau x=-\tau q(e_4+ie_5)+\tau p(e_6-ie_7) $. Hence, for $
P \in V^{12} $, the explicit form of $  (P,P)_{\varepsilon_2} $ as $
x_i:=p_i(e_4+ie_5)+q_i(e_6-ie_7),i=1,2,3 $ is given by
\begin{align*}
     (P,P)_{\varepsilon_2}=(p_1(\tau p_1)+q_1(\tau q_1))+(p_2(\tau
     p_2)+q_2(\tau q_2))+(p_3(\tau p_3)+q_3(\tau q_3)).
\end{align*}

Next, we move on to study a group $ (E_7)^{\varepsilon_1,\varepsilon_2} $ defined below.

We define a subgroup $ (E_7)^{\varepsilon_1,\varepsilon_2} $ of $ E_7 $ by
\begin{align*}
    (E_7)^{\varepsilon_1,\varepsilon_2}:=\left\lbrace \alpha \in E_7
    \relmiddle{|} \varepsilon_1\alpha=\alpha\varepsilon_1,
    \varepsilon_2\alpha=\alpha\varepsilon_2 \right\rbrace,
\end{align*}

Then we have the following lemma and proposition.

\begin{lemma}\label{lemma 7.9}
  The Lie algebra $ (\mathfrak{e}_7)^{\varepsilon_1,\varepsilon_2} $ of
  the group $ (E_7)^{\varepsilon_1,\varepsilon_2} $ is given by
  \begin{align*}
   (\mathfrak{e}_7)^{\varepsilon_1,\varepsilon_2}=\left\lbrace
   \varPhi(\phi,A,-\tau A,\nu) \in \mathfrak{e}_7 \relmiddle{|}
   \begin{array}{l}
       \phi=(D_1,D_2,D_3)+\tilde{A}_1(a_1)+\tilde{A}_2(a_2)+\tilde{A}
       _3(a_3)+i\tilde{T}
       \\
       \quad
       D_1=d_{01}G_{01}+d_{02}G_{02}+d_{03}G_{03}+d_{12}G_{12}+d_{13}G_{13}
       \\
       \qquad +d_{23}G_{23}+d_{45}(G_{45}+G_{67})+d_{46}(G_{46}-G_{57})
       \\
       \qquad +d_{47}(G_{47}+G_{56}),d_{ij} \in \R,
       \\
       \quad D_2=\pi\kappa D_1,D_3=\kappa\pi D_1,
       \\
       \quad a_i \in \H, T \in (\mathfrak{J}_{\sH})_0,
       \\[1mm]
       A \in (\mathfrak{J}_{\sH})^C.
       \\[1mm]
       \nu \in i\R
   \end{array} \right\rbrace.
  \end{align*}

 In particular, we have $ \dim(
 (\mathfrak{e}_7)^{\varepsilon_1,\varepsilon_2})=(9+4\times
 3+15-1)+30+1=66 $.
\end{lemma}
\begin{proof}
     Using the result of \cite[Theorem 4.3.4]{iy0}, by straightforward
     computation, we can obtain the required result.
\end{proof}

\begin{proposition}\label{proposition 7.10}
    The group $ (E_7)^{\varepsilon_1,\varepsilon_2} $ is connected.
\end{proposition}
\begin{proof}
 First, we will confirm that $ \tau\lambda $ commutes with $
 \varepsilon_i,i=1,2 $, respectively. Indeed, it follows that
 \begin{align*}
     (\tau\lambda)\varepsilon_i(X,Y,\xi,\eta)
     &=(\tau\lambda)(\varepsilon_iX,\varepsilon_iY,\xi,\eta),(X,Y,\xi,\eta)
      \in \mathfrak{P}^C
     \\
     &=(\varepsilon_i\tau Y,-\varepsilon_i\tau X , \tau\eta, -\tau\xi)
     \\
     &=\varepsilon_i(\tau\lambda)(X,Y,\xi,\eta),
 \end{align*}
that is, $ (\tau\lambda)\varepsilon_i=\varepsilon_i(\tau\lambda) $. Hence
we have
\begin{align*}
    (({E_7}^C)^{\varepsilon_1,\varepsilon_2})^{\tau\lambda}=(({E_7}
    ^C)^{\tau\lambda})^{\varepsilon_1,\varepsilon_2}=(E_7)^{\varepsilon_1,
    \varepsilon_2}.
\end{align*}
Thus, since $ ({E_7}^C)^{\varepsilon_1,\varepsilon_2} (\cong Spin(12,C)) $
is simply connected (\cite[Proposition 1.1.7]{miya1}), the group $
(E_7)^{\varepsilon_1,\varepsilon_2}\allowbreak =(({E_7}^C)^{\varepsilon_1,
\varepsilon_2})^{\tau\lambda}
 $ is connected.
\end{proof}

Now, we will construct the spinor group $ Spin(12) $ in $ E_7 $.

\begin{proposition}\label{proposition 7.11}
  The group $ (E_7)^{\varepsilon_1,\varepsilon_2} $ is isomorphic to the
  group $ Spin(12) ${\rm :}$  (E_7)^{\varepsilon_1,\varepsilon_2} \cong
  Spin(12) $.
\end{proposition}
\begin{proof}
Let the orthogonal group $ O(12)=O(V^{12})=\{\beta \in
\Iso_{\sR}(V^{12})\,|\,(\beta P,\beta
Q)_{\varepsilon_2}=(P,Q)_{\varepsilon_2} \} $. We consider the restriction
$ \alpha \big|_{V^{12}} $ of $ \alpha \in
(E_7)^{\varepsilon_1,\varepsilon_2} $ to $ V^{12} $, then since $
(E_7)^{\varepsilon_1,\varepsilon_2} $ acts on $ V^{12} $, we see $ \alpha
\big|_{V^{12}} \in O(12) $. Moreover, since $
(E_7)^{\varepsilon_1,\varepsilon_2} $ is connected (Proposition
\ref{proposition 7.10}), we can define a homomorphism $ p:
(E_7)^{\varepsilon_1,\varepsilon_2} \to SO(V^{12})=SO(12) $ by
\begin{align*}
    p(\alpha)=\alpha \big|_{V^{12}}.
\end{align*}

We will determine $ \Ker\,p $. However, since $ p $ is the restriction of
the mapping $ \pi $ defined in \cite[Proposition 1.1.7]{miya1}, we easily
obtain $ \Ker\,p=\Ker\,\pi=\{1,-\gamma \} \cong \Z_2 $.

Finally, we will prove that $ p $ is surjective.
Since the group $ (E_7)^{\varepsilon_1,\varepsilon_2} $ ic connected
(Proposition \ref{proposition 7.10}) and $ \Ker\,p $ is discrete,
together with $
\dim((\mathfrak{e}_7)^{\varepsilon_1,\varepsilon_2})=66=\dim(\mathfrak{so}
(12))$ (Lemma \ref{lemma 7.7}), we see that $ p $ is surjective.
Thus we have the isomorphism
\begin{align*}
 (E_7)^{\varepsilon_1,\varepsilon_2}/\Z_2 \cong SO(12).
\end{align*}

Therefore the group $ (E_7)^{\varepsilon_1,\varepsilon_2} $ is isomorphic
to $ Spin(12) $ as the universal covering group of $ SO(12) $, that is,
\begin{align*}
    (E_7)^{\varepsilon_1,\varepsilon_2} \cong Spin(12).
\end{align*}
\end{proof}

Let $ \alpha \in (E_7)^{\varepsilon_1,\varepsilon_2} $. Then, as in $
({E_7}^C)^{\varepsilon_1,\varepsilon_2} $, for $ P \in
(\mathfrak{P}_{\sH})^C $, we have $ \alpha P \in (\mathfrak{P}_{\sH})^C $
because of $ (E_7)^{\varepsilon_1,\varepsilon_2} \subset (E_7)^\gamma $.
Hence $ \alpha $ induces a $ C $-linear isomorphism $ \alpha:
(\mathfrak{P}_{\sH})^C \to (\mathfrak{P}_{\sH})^C $.
\vspace{2mm}

Now, we will determine the structure of the group $ E_{7,\sH} $.

\begin{theorem}\label{theorem 7.12}
    The group $ E_{7,\sH} $ is isomorphic to the semi-spinor group $
    Ss(12) ${\rm :} $ E_{7,\sH} \cong Ss(12) $.
\end{theorem}
\begin{proof}
   Let $ Spin(12) $ as the group $ (E_7)^{\varepsilon_1,\varepsilon_2} $
   (Proposition \ref{proposition 7.11}). Then we define a mapping $ f_7:
   (E_7)^{\varepsilon_1,\varepsilon_2} \to E_{7,\sH} $ by
  \begin{align*}
      f_7(\alpha)=\alpha \big |_{(\mathfrak{P}_{\sH})^C}.
  \end{align*}
  As mentioned above, since $ \alpha \in (E_7)^{\varepsilon_1,
  \varepsilon_2} $ induces the $ C $-linear isomorphism $ \alpha:
  (\mathfrak{P}_{\sH})^C \to (\mathfrak{P}_{\sH})^C $, we easily see $
  \alpha
  \big |_{(\mathfrak{P}_{\sH})^C} \in E_{7,\sH}  $. Hence $ f_7 $ is well-
  defined. Moreover, it is clear that $ f_7 $ is a homomorphism.

Next, we will determine $ \Ker\, f_7 $. In order to determine it, we will
prove the following claim.
\vspace{-3mm}

    \begin{quote}
          \begin{claim1}\label{claim 1} The kernel of the differential
          mapping $
          f_{7_*}:(\mathfrak{e}_7)^{\varepsilon_1,\varepsilon_2} \to
          \mathfrak{e}_{7,\sH} $ of $ f_7 $  is trivial: $ \Ker\,
          f_{7_*}=\{0\}$.
          \end{claim1}
          \begin{proof}
         Let $ \varPhi:=\varPhi(\phi,A,-\tau A,\nu) \in \Ker\,
           f_{7_*} $. Then $ \varPhi $ satisfies the condition
          \begin{align*}
          \varPhi P=0\,\,\text{for all}\,\,P \in (\mathfrak{P}_{\sH})^C.
          \end{align*}

    Set $ P:=(0,0,0,1) $, then it follows from $ 0=\varPhi
    P=(A,0,0,-\nu) $ that $ A=0 $ and $ \nu=0 $.
    Hence $ \varPhi $ is of the form  $ \varPhi(\phi,0,0,0) $: $
    \varPhi=\varPhi(\phi,0,0,0) $.
    In addition, set $ P:=(F_1(x),0,0,0), x \in \H^C $, then as $
    \phi=(D_1,D_2,D_3)+\tilde{A}_1(a_1)+\tilde{A}_2(a_2)+\tilde{A}_3(a_3)+
    (\tau_1E_1+\tau_2E_2+\tau_3E_3+F_1(t_1)+F_2(t_2)+F_3(t_3))^\sim
     $, it follows from
   \begin{align*}
    0=\varPhi P
    &=(\phi F_1(x), 0, 0, 0)
    \\
    &=(((a_1,x)+(t_1,x))E_2+((a_1,x)-(t_1,x))E_3
    \\
    &\quad
    +F_1(D_1x-(1/2)\tau_1x)+(1/2)F_2(\ov{a_3x}+\ov{t_3x})+(1/2)F_3(-
    \ov{xa_2}+\ov{xt_2}),0,0,0)
   \end{align*}
   that
   \begin{align*}
   &(a_1,x)+(t_1,x)=(a_1,x)-(t_1,x)=0,\;\; D_1x-(1/2)\tau_1x=0,
   \\
   &\ov{a_3x}+\ov{t_3x}=0,\;\; -\ov{xa_2}+\ov{xt_2}=0 \,\,\text{for all}\,
   \, x \in (\mathfrak{P}_{\sH})^C.
   \end{align*}
   Moreover, set $ P:=(F_2(x),0,0,0),(F_3(x),0,0,0), x \in \H^C $,
   then similarly it follows from $ \varPhi P=0 $ that
   \begin{align*}
  & \left\lbrace
   \begin{array}{l}
   (a_2,x)+(t_2,x)=(a_2,x)-(t_2,x)=0,\;\; D_2x-(1/2)\tau_2x=0,
   \\[1mm]
   \ov{a_1x}+\ov{t_1x}=0, \;\;-\ov{xa_3}+\ov{xt_3}=0,
   \end{array}\right.
   \\
   &\left\lbrace
   \begin{array}{l}
   (a_3,x)+(t_3,x)=(a_3,x)-(t_3,x)=0,\;\; D_3x-(1/2)\tau_3x=0,
   \\[1mm]
   \ov{a_2x}+\ov{t_2x}=0, \;\;-\ov{xa_1}+\ov{xt_1}=0
   \end{array}\right.
   \end{align*}
   for all $ x \in (\mathfrak{P}_{\sH})^C $ , respectively.

   \noindent Hence we have
   \begin{align*}
   (*)\,\left\lbrace
   \begin{array}{l}
    a_1=a_2=a_3=0,\;\;t_1=t_2=t_3=0,
    \\[1mm]
    D_1x-(1/2)\tau_1x=0,\;\;D_2x-(1/2)\tau_2x=0,\;\;D_3x-(1/2)\tau_3x=0.
   \end{array}\right.
  \end{align*}

  Finally, set $ P:=(0,F_k(x),0,0),x \in \H^C $, then as results
  different from $ (*) $, we have
  \begin{align*}
  D_1x+(1/2)\tau_1x=0,\;\;D_2x+(1/2)\tau_2x=0,\;\;D_3x+(1/2)\tau_3x=0
  \end{align*}
  Therefore, together with $ (*) $, we have
  \begin{align*}
  D_1=D_2=D_3=0,\;\;\tau_1=\tau_2=\tau_3=0.
  \end{align*}

With above, we obtain the required result.
\end{proof}

\begin{claim2}\label{claim 2}
The center $ z((E_7)^{\varepsilon_1,\varepsilon_2}) $ of the group $
(E_7)^{\varepsilon_1,\varepsilon_2} $ coincides with the center $ z(({E_7}^C)^{\varepsilon_1,\varepsilon_2}) $ of the group $ ({E_7}
^C)^{\varepsilon_1,\varepsilon_2} ${\rm  :} $
z((E_7)^{\varepsilon_1,\varepsilon_2})=\{1,\gamma\}\times \{1,-\gamma\} \cong \Z_2 \times \Z_2 $.
\end{claim2}
\begin{proof}
First, from Proposition \ref{proposition 7.11} we can confirm
\begin{align*}
z((E_7)^{\varepsilon_1,\varepsilon_2}) \cong z(Spin(12)) \cong \Z_2 \times \Z_2,
\end{align*}
and moreover from \cite[Proposition 1.1.6]{miya1} and $
(E_7)^{\varepsilon_1,\varepsilon_2} \subset ({E_7}^C)^{\varepsilon_1,
\varepsilon_2} $ we have
\begin{align*}
z((E_7)^{\varepsilon_1,\varepsilon_2})
\supset z(({E_7}^C)^{\varepsilon_1,\varepsilon_2}&=\{1,\gamma\}\times \{1,-\gamma\}
\\
&\cong \Z_2 \times \Z_2.
\end{align*}
Hence we obtain the required result.
\end{proof}
\end{quote}
\vspace{2mm}

Here, the group $ (E_7)^{\varepsilon_1,\varepsilon_2} $ is connected
(Proposition \ref{proposition 7.10}) and $ \Ker\,f_7 $ is discrete (Claim
1), so that together with Claim 2, we have
\begin{align*}
\Ker\,f_7 \subset z((E_7)^{\varepsilon_1,\varepsilon_2})=\{1,\gamma\}\times \{1,-\gamma\}.
\end{align*}
However, since $ -1, -\gamma $ do not leave $ (\mathfrak{P}_{\sH})^C $
, we see that $ -1, -\gamma \notin \Ker\,f_7  $. Thus we have
\begin{align*}
\Ker\,f_7=\{1,\gamma\}=\Z_2.
\end{align*}


Finally, we will prove that $ f_7 $ is surjective. Since the group $
E_{7,\sH} $ is connected (Theorem \ref{theorem 7.6}) and $
\Ker\,f_7 $ is discrete, together with $ \dim((\mathfrak{e}
_7)^{\varepsilon_1,\varepsilon_2})=66=\dim(\mathfrak{e}_{7,
\sH}) $(Lemmmas \ref{lemma 7.7}, \ref{lemma 7.2}), $ f_7 $ is
surjective.

Therefore, from Proposition \ref{proposition 7.11}, we have the isomorphism
\begin{align*}
    E_{7,\sH} \cong Spin(12)/\Z_2,
\end{align*}
that is, we have the required isomorphism
\begin{align*}
    E_{7,\sH} \cong Ss(12).
\end{align*}
\end{proof}

We consider a subgroup $ ({E_7}^C)^{\varepsilon_1,\varepsilon_2} $ of $ {E_7}^C $ by
\begin{align*}
({E_7}^C)^{\varepsilon_1,\varepsilon_2}:=\left\lbrace \alpha \in {E_7}^C \relmiddle{|} \varepsilon_1\alpha=\alpha\varepsilon_1, \varepsilon_2\alpha=\alpha\varepsilon_2
\right\rbrace.
\end{align*}
Then the group $ ({E_7}^C)^{\varepsilon_1,\varepsilon_2} $ has been studied by author in  \cite[Proposition 1.1.7]{miya1}, and its result was obtained as follows:
\begin{align*}
    ({E_7}^C)^{\varepsilon_1,\varepsilon_2} \cong Spin(12,C).
\end{align*}

Now, we move the determination of the root system and the Dynkin diagram of the Lie algebra $ ({\mathfrak{e}_{7}}^C)^{\varepsilon_1,\varepsilon_2} $ of the group $({E_7}^C)^{\varepsilon_1,\varepsilon_2} $.
After this, we will determine those by the same procedure as in $
({\mathfrak{f}_4}^C)^{\varepsilon_1,\varepsilon_2} $.

Then we need the following lemma used later.

\begin{lemma}{\rm(\cite[Lemma 1.1.2]{miya1})}\label{lemma 7.13}
    The Lie algebra $ ({\mathfrak{e}_7}
    ^C)^{\varepsilon_1,\varepsilon_2} $ of the group $ ({E_7}^C)^{\varepsilon_1,\varepsilon_2} $ is given by
    \begin{align*}
        ({\mathfrak{e}_7}^C)^{\varepsilon_1,\varepsilon_2}
        =\left\lbrace
        \begin{array}{l}
            \varPhi(\phi,A,B,\nu)  \in {\mathfrak{e}_7}^C
        \end{array}
        \relmiddle{|}
        \begin{array}{l}
            \phi=(D_1,D_2,D_3) \vspace{1mm}\\
            \quad +\tilde{A}_1(a_1)+\tilde{A}_2(a_2)+\tilde{A}_3(a_3)\\
            \quad +(\tau_1E_1+\tau_2E_2+\tau_3E_3+F_1(t_1) \\
            \quad +F_2(t_2)+F_3(t_3))^\sim,
            \\
            \qquad D_1=d_{01}G_{01}+d_{02}G_{02}+d_{03}G_{03} \\
            \qquad +d_{12}G_{12}+d_{13}G_{13}+d_{23}G_{23} \\
            \qquad +d_{45}(G_{45}+G_{67})+d_{46}(G_{46}-G_{57}) \\
            \qquad +d_{47}(G_{47}+G_{56}),d_{ij} \in C, \\
            \qquad D_2=\pi\kappa D_1, D_3=\kappa\pi D_1,\\
            \qquad a_k, t_k \in \H^C ,\\
            \qquad \tau_k \in C, \tau_1+\tau_2+\tau_3=0,\\
          A,B \in (\mathfrak{J}_{\sH})^C, \\
          \nu \in C
        \end{array}
        \right\rbrace.
    \end{align*}

    In particular, we have $ \dim_C( ({\mathfrak{e}_7}^C)^{\varepsilon_1,
    \varepsilon_2})=(9+4 \times 3+(3-1)+4 \times 3)+15 \times 2+1=66. $
\end{lemma}
\vspace{-3mm}

Here, we define a Lie subalgebra $ \mathfrak{h}_7 $ of $ ({\mathfrak{e}_7}
^C)^{\varepsilon_1,\varepsilon_2} $ by
\begin{align*}
    \mathfrak{h}_7:=\left\lbrace
\varPhi_7=\varPhi(\phi,0,0,\mu) \in ({\mathfrak{e}_7}^C)^{\varepsilon_1,
\varepsilon_2}
    \relmiddle{|}
    \begin{array}{l}
\phi=\delta +(\mu_1E_1+\mu_2E_2+\mu_3E_3)^\sim
\\
\quad \delta:=(L_1,L_2,L_3),
\\
\qquad L_1=\lambda_0(iG_{01})+\lambda_1(iG_{23})
\\
\qquad\qquad\;\; +\lambda_2(i(G_{45}+G_{67})),
\\
\qquad  L_2=\pi\kappa L_1, L_3=\kappa\pi L_1,\lambda_i \in C,
\\
\quad \mu_k \in C, \mu_1+\mu_2+\mu_3=0,
\\
\mu \in C
    \end{array}
    \right\rbrace.
\end{align*}

Then $ \mathfrak{h}_7 $ is a Cartan subalgebra of $ ({\mathfrak{e}_7}
^C)^{\varepsilon_1,\varepsilon_2} $. Indeed, it is clear that $
({\mathfrak{e}_7}^C)^{\varepsilon_1,\varepsilon_2} $ is abelian. Next, by
doing straightforward computation, we can confirm that $ [\varPhi,
\varPhi_7]\in \mathfrak{h}_7, \varPhi \in ({\mathfrak{e}_7}
^C)^{\varepsilon_1,\varepsilon_2} $ implies $ \varPhi \in \mathfrak{h}_7 $
for
any $ \varPhi_7 \in \mathfrak{h}_7 $.

\begin{theorem}\label{theorem 7.14}
    The rank of the Lie algebra $ ({\mathfrak{e}_7}^C)^{\varepsilon_1,
    \varepsilon_2} $ is six. The roots $ \varDelta $ of $ ({\mathfrak{e}
    _7}^C)^{\varepsilon_1,\varepsilon_2} $ relative to $ \mathfrak{h}_7 $
    are given by
    \begin{align*}
        \varDelta=\left\lbrace
        \begin{array}{l}
            \pm(\lambda_0-\lambda_1), \;
            \pm(\lambda_0+\lambda_1),\; \pm 2\lambda_2, \;
            \vspace{1mm}\\
            \pm\dfrac{1}{2}(-2\lambda_0+\mu_2-\mu_3),\;
            \pm\dfrac{1}{2}(-2\lambda_0-\mu_2+\mu_3),
            \vspace{1mm}\\
            \pm\dfrac{1}{2}(-2\lambda_1+\mu_2-\mu_3),\;
            \pm\dfrac{1}{2}(-2\lambda_1-\mu_2+\mu_3),
            \vspace{1mm}\\
            \pm\dfrac{1}{2}(\lambda_0-\lambda_1-2\lambda_2-\mu_1+\mu_3),\;
            \pm\dfrac{1}{2}(\lambda_0-\lambda_1-2\lambda_2+\mu_1-\mu_3),
            \vspace{1mm}\\
            \pm\dfrac{1}{2}(\lambda_0-\lambda_1+2\lambda_2-\mu_1+\mu_3),\;
            \pm\dfrac{1}{2}(\lambda_0-\lambda_1+2\lambda_2+\mu_1-\mu_3),
            \vspace{1mm}\\
            \pm\dfrac{1}{2}(\lambda_0+\lambda_1+2\lambda_2+\mu_1-\mu_2),\;
            \pm\dfrac{1}{2}(\lambda_0+\lambda_1+2\lambda_2-\mu_1+\mu_2),
            \vspace{1mm}\\
            \pm\dfrac{1}{2}(-\lambda_0-\lambda_1+2\lambda_2+\mu_1-\mu_2),\;
            \pm\dfrac{1}{2}(-\lambda_0-\lambda_1+2\lambda_2-\mu_1+\mu_2),
\vspace{1mm}\\
\pm(\mu_1+\dfrac{2}{3}\mu),\;\pm(\mu_2+\dfrac{2}{3}\mu),\;\pm(\mu_3+
\dfrac{2}{3}\mu),
\vspace{1mm}\\
\pm(-\lambda_0-\dfrac{1}{2}\mu_1+\dfrac{2}{3}\mu),\,
\pm(\lambda_0-\dfrac{1}
{2}\mu_1+\dfrac{2}{3}\mu),
\pm(-\lambda_1-\dfrac{1}{2}\mu_1+\dfrac{2}{3}
\mu),\,
\pm(\lambda_1-\dfrac{1}{2}\mu_1+\dfrac{2}{3}\mu),
\vspace{1mm}\\
\pm(-\dfrac{1}{2}(-\lambda_0+\lambda_1+2\lambda_2)-\dfrac{1}{2}\mu_2+
\dfrac{2}{3}\mu),\;
\pm(\dfrac{1}{2}(-\lambda_0+\lambda_1+2\lambda_2)-\dfrac{1}{2}\mu_2+
\dfrac{2}{3}\mu),
\vspace{1mm}\\
\pm(-\dfrac{1}{2}(-\lambda_0+\lambda_1-2\lambda_2)-\dfrac{1}{2}\mu_2+
\dfrac{2}{3}\mu),\;
\pm(\dfrac{1}{2}(-\lambda_0+\lambda_1-2\lambda_2)-
\dfrac{1}{2}\mu_2+
\dfrac{2}{3}\mu),
\vspace{1mm}\\
\pm(-\dfrac{1}{2}(-\lambda_0-\lambda_1-2\lambda_2)-\dfrac{1}{2}\mu_3+
\dfrac{2}{3}\mu),\;
\pm(\dfrac{1}{2}(-\lambda_0-\lambda_1-2\lambda_2)-\dfrac{1}{2}\mu_3+
\dfrac{2}{3}\mu),
\vspace{1mm}\\
\pm(-\dfrac{1}{2}(\lambda_0+\lambda_1-2\lambda_2)-\dfrac{1}{2}\mu_3+
\dfrac{2}{3}\mu),\;
\pm(\dfrac{1}{2}(\lambda_0+\lambda_1-2\lambda_2)-\dfrac{1}{2}\mu_3+
\dfrac{2}{3}\mu)
        \end{array}
        \right\rbrace.
    \end{align*}
\end{theorem}
\begin{proof}
The roots of $ ({\mathfrak{e}_6}^C)^{\varepsilon_1,\varepsilon_2} $ is also
the roots of $ ({\mathfrak{e}_7}^C)^{\varepsilon_1,\varepsilon_2} $.
Indeed, let the root $ \alpha $ of $ ({\mathfrak{e}_6}^C)^{\varepsilon_1,
\varepsilon_2} $ and its associated root vector $ V \in ({\mathfrak{e}_6}
^C)^{\varepsilon_1,\varepsilon_2} \subset ({\mathfrak{e}
_7}^C)^{\varepsilon_1,\varepsilon_2}$. Then it follows that
\begin{align*}
[\varPhi_7, V]&=[\varPhi(\phi,0,0,\mu), \varPhi(V,0,0,0)]=\varPhi([\phi,V],
0,0,0)
\\
&=\varPhi(\alpha(\phi)V,0,0,0)=\alpha(\phi)\varPhi(V,0,0,0)
\\
&=\alpha(\varPhi_7)V.
\end{align*}

We will determine the remainders of roots and give a few examples for that.
    First, let $ \varPhi_7=\varPhi(\phi,0,\allowbreak 0,\nu) \in
    \mathfrak{h}_7 $ and $ \varPhi(0,E_1,0,0) \in ({\mathfrak{e}_7}
    ^C)^{\varepsilon_1,\varepsilon_2}$. Then it
    follows that
    \begin{align*}
    [\varPhi_7, \varPhi(0,E_1,0,0)]&=[\varPhi(\phi, 0,0,\mu),
    \varPhi(0,E_1,0,0)]
    \\
    &=\varPhi(0,(\phi+(2/3)\mu)E_1,0,0)
    \\
    &=\varPhi(0,(\mu_1+(2/3)\mu)E_1,0,0)
    \\
    &=(\mu_1+(2/3)\mu)\varPhi(0,E_1,0,0),
    \end{align*}
that is, $ [\varPhi_7, \varPhi(0,E_1,0,0)]=(\mu_1+(2/3)\mu)\varPhi(0,E_1,
0,0) $. Hence we see that $ \mu_1+(2/3)\mu $ is a root and $
\varPhi(0,E_1,0,0) $ is an associated root vector. Next, let $ \varPhi(0,
F_1(1+ie_1),0,0) \in ({\mathfrak{e}_7}^C)^{\varepsilon_1,\varepsilon_2} $.
Then
it follows that
    \begin{align*}
    [\varPhi, \varPhi(0,F_1(1+ie_1),0,0)]&=[\varPhi(\phi, 0,0,\mu),
    \varPhi(0,F_1(1),0,0)]
    \\
    &=\varPhi(0,(\phi+(2/3)\mu)F_1(1+ie_1),0,0)
    \\
    &=\varPhi(0,\phi F_1(1+ie_1)+(2/3)\mu F_1(1+ie_1), 0,0)
    \\
    &=\varPhi(0,(-\lambda_0+(1/2)(\mu_2+\mu_3))F_1(1+ie_1)+(2/3)\mu
    F_1(1+ie_1),0,0)
    \\
    &=(-\lambda_0+(1/2)(\tau_2+\tau_3)+(2/3)\mu)\varPhi(0,F_1(1+ie_1),0,0),
    \end{align*}
    that is, $ [\varPhi_7, \varPhi(0,F_1(1+ie_1),0,0)]=(-\lambda_0-(1/2)
    \mu_1+(2/3)\mu)\varPhi(0,F_1(1+ie_1),0,0) $. Hence we see that $-
    \lambda_0-(1/2)\mu_1+(2/3)\mu $ is a root and $  \varPhi(0,F_1(1+ie_1),
    0,0) $ is an associated root vector.

Together with two cases above, the remainders of roots and root vectors
associated with these roots are obtained as follows:
\begin{longtable}[c] {l@{\qquad\qquad}l}
   \hspace{20mm} roots & \hspace{-10mm}root vectors associated with roots
   \vspace{0.5mm}\cr
   $ \mu_1+(2/3)\mu $ & $ \varPhi(0,E_1,0,0) $
   \vspace{0.5mm}\cr
   $ -(\mu_1+(2/3)\mu) $ & $ \varPhi(0,0,E_1,0) $
   \vspace{0.5mm}\cr
   $ \mu_2+(2/3)\mu $ & $ \varPhi(0,E_2,0,0) $
   \vspace{0.5mm}\cr
   $ -(\mu_2+(2/3)\mu) $ & $ \varPhi(0,0,E_2,0) $
   \vspace{0.5mm}\cr
   $ \mu_3+(2/3)\mu $ & $ \varPhi(0,E_3,0,0) $
   \vspace{0.5mm}\cr
   $ -(\mu_3+(2/3)\mu) $ & $ \varPhi(0,0,E_3,0) $
   \vspace{0.5mm}\cr
   $ -\lambda_0-(1/2)\mu_1+(2/3)\mu $ & $ \varPhi(0,F_1(1+ie_1),0,0) $
   \vspace{0.5mm}\cr
   $ -(-\lambda_0-(1/2)\mu_1+(2/3)\mu) $ & $ \varPhi(0,0,F_1(1-ie_1),0) $
   \vspace{0.5mm}\cr
   $ \lambda_0-(1/2)\mu_1+(2/3)\mu $ & $ \varPhi(0,F_1(1-ie_1),0,0) $
   \vspace{0.5mm}\cr
   $ -(\lambda_0-(1/2)\mu_1+(2/3)\mu) $ & $ \varPhi(0,0,F_1(1+ie_1),0) $
   \vspace{0.5mm}\cr
   $ -\lambda_1-(1/2)\mu_1+(2/3)\mu $ & $ \varPhi(0,F_1(e_2+ie_3),0,0) $
   \vspace{0.5mm}\cr
   $ -(-\lambda_1-(1/2)\mu_1+(2/3)\mu) $ & $ \varPhi(0,0,F_1(e_2-ie_3),0) $
   \vspace{0.5mm}\cr
   $ \lambda_1-(1/2)\mu_1+(2/3)\mu $ & $ \varPhi(0,F_1(e_2-ie_3),0,0) $
   \vspace{0.5mm}\cr
   $ -(\lambda_1-(1/2)\mu_1+(2/3)\mu) $ & $ \varPhi(0,0,F_1(e_2+ie_3),0) $
   \vspace{0.5mm}\cr
   $ (-1/2)(-\lambda_0+\lambda_1+2\lambda_2)-(1/2)\mu_2+(2/3)\mu $ & $
   \varPhi(0,F_2(1+ie_1),0,0) $
   \vspace{0.5mm}\cr
   $ -((-1/2)(-\lambda_0+\lambda_1+2\lambda_2)-(1/2)\mu_2+(2/3)\mu) $ & $
   \varPhi(0,0,F_2(1-ie_1),0) $
   \vspace{0.5mm}\cr
   $ (1/2)(-\lambda_0+\lambda_1+2\lambda_2)-(1/2)\mu_2+(2/3)\mu $ & $
   \varPhi(0,F_2(1-ie_1),0,0) $
   \vspace{0.5mm}\cr
   $ -((1/2)(-\lambda_0+\lambda_1+2\lambda_2)-(1/2)\mu_2+(2/3)\mu) $ & $
   \varPhi(0,0,F_2(1+ie_1),0) $
   \vspace{0.5mm}\cr
   $ (-1/2)(-\lambda_0+\lambda_1-2\lambda_2)-(1/2)\mu_2+(2/3)\mu $ & $
   \varPhi(0,F_2(e_2+ie_3),0,0) $
   \vspace{0.5mm}\cr
   $ -((-1/2)(-\lambda_0+\lambda_1-2\lambda_2)-(1/2)\mu_2+(2/3)\mu) $ & $
   \varPhi(0,0,F_2(e_2-ie_3),0) $
   \vspace{0.5mm}\cr
   $ (1/2)(-\lambda_0+\lambda_1-2\lambda_2)-(1/2)\mu_2+(2/3)\mu $ & $
   \varPhi(0,F_2(e_2-ie_3),0,0) $
   \vspace{0.5mm}\cr
   $ -((1/2)(-\lambda_0+\lambda_1-2\lambda_2)-(1/2)\mu_2+(2/3)\mu) $ & $
   \varPhi(0,0,F_2(e_2+ie_3),0) $
   \vspace{0.5mm}\cr
   $ (-1/2)(-\lambda_0-\lambda_1-2\lambda_2)-(1/2)\mu_3+(2/3)\mu $ & $
   \varPhi(0,F_3(1+ie_1),0,0) $
   \vspace{0.5mm}\cr
   $ -((-1/2)(-\lambda_0-\lambda_1-2\lambda_2)-(1/2)\mu_3+(2/3)\mu) $ & $
   \varPhi(0,0,F_3(1-ie_1),0) $
   \vspace{0.5mm}\cr
   $ (1/2)(-\lambda_0-\lambda_1-2\lambda_2)-(1/2)\mu_3+(2/3)\mu $ & $
   \varPhi(0,F_3(1-ie_1),0,0) $
   \vspace{0.5mm}\cr
   $ -((1/2)(-\lambda_0-\lambda_1-2\lambda_2)-(1/2)\mu_3+(2/3)\mu) $ & $
   \varPhi(0,0,F_3(1+ie_1),0) $
   \vspace{0.5mm}\cr
   $ (-1/2)(\lambda_0+\lambda_1-2\lambda_2)-(1/2)\mu_3+(2/3)\mu $ & $
   \varPhi(0,F_3(e_2+ie_3),0,0) $
   \vspace{0.5mm}\cr
   $ -((-1/2)(\lambda_0+\lambda_1-2\lambda_2)-(1/2)\mu_3+(2/3)\mu) $ & $
   \varPhi(0,0,F_3(e_2-ie_3),0) $
   \vspace{0.5mm}\cr
   $ (1/2)(\lambda_0+\lambda_1-2\lambda_2)-(1/2)\mu_3+(2/3)\mu $ & $
   \varPhi(0,F_3(e_2-ie_3),0,0) $
   \vspace{0.5mm}\cr
   $ -((1/2)(\lambda_0+\lambda_1-2\lambda_2)-(1/2)\mu_3+(2/3)\mu) $ & $
   \varPhi(0,0,F_3(e_2+ie_3),0) $.
\end{longtable}

Thus, since $ ({\mathfrak{e}_7}^C)^{\varepsilon_1,\varepsilon_2} $ is
spanned by $ \mathfrak{h}_7 $ and the root vectors associated with roots
above, the roots obtained above are all. The rank of the Lie algebra $
({\mathfrak{e}_7}^C)^{\varepsilon_1,\varepsilon_2} $ follows from the
dimension of $ \mathfrak{e}_7 $.
\end{proof}

Subsequently, we prove the following theorem.

\begin{theorem}\label{theorem 7.15}
In the root system $ \varDelta $ of Theorem {\rm \ref{theorem 7.14}}
 \begin{align*}
      \varPi=\left\lbrace \alpha_1,\alpha_2, \alpha_3, \alpha_4,\alpha_5,
      \alpha_6 \right\rbrace
  \end{align*}
is a fundamental root system of $  ({\mathfrak{e}_7}^C)^{\varepsilon_1,
\varepsilon_2} $, where
$
\alpha_1=-(\lambda_0-\lambda_1),
\alpha_2=(1/2)(\lambda_0-\lambda_1-2\lambda_2+\mu_1-\mu_3),
\alpha_3=2\lambda_2,
\alpha_4=-(1/2)(\lambda_0+\lambda_1+2\lambda_2+\mu_1-\mu_2),
\alpha_5=\lambda_0+\lambda_1,
\alpha_6=-(\mu_2+(2/3)\mu)$.
The Dynkin diagram of $ ({\mathfrak{e}_7}^C)^{\varepsilon_1,\varepsilon_2}
$ is given by
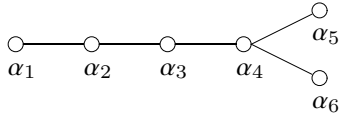
\begin{center}
\setlength{\unitlength}{1mm}
  \scalebox{1.0}
	{\begin{picture}(80,20)
	\put(0,9){}
	\put(20,10){\circle{2}} \put(19,6){$\alpha_1$}
	\put(21,10){\line(1,0){8}}
	\put(30,10){\circle{2}} \put(29,6){$\alpha_2$}
	\put(31,10){\line(1,0){8}}
	\put(40,10){\circle{2}} \put(39,6){$\alpha_3$}
	\put(41,10){\line(1,0){8}}
	\put(50,10){\circle{2}} \put(49,6){$\alpha_4$}
	\put(51,10){\line(2,1){8}}
	\put(60,14.5){\circle{2}} \put(59,10.5){$\alpha_5$}
   \put(51,10){\line(2,-1){8}}
	\put(60,5.5){\circle{2}} \put(59,1.5){$\alpha_6$}

	\end{picture}}
\end{center}
\end{theorem}
\begin{proof}
The all positive roots are expressed by $ \alpha_1,\alpha_2,\alpha_3,
\alpha_4,\alpha_5,\alpha_6 $ as follows:
   \begin{align*}
-(\lambda_0-\lambda_1)&=\alpha_1,
\\
\lambda_0+\lambda_1&=\alpha_5,
\\
2\lambda_2&=\alpha_3,
\\
(1/2)(-2\lambda_0+\mu_2-\mu_3)&=\alpha_1+\alpha_2+\alpha_3+\alpha_4,
\\
-(1/2)(-2\lambda_0-\mu_2+\mu_3)&=\alpha_2+\alpha_3+\alpha_4+\alpha_5,
\\
(1/2)(-2\lambda\_1+\mu_2-\mu_3)&=\alpha_2+\alpha_3+\alpha_4,
\\
-(1/2)(-2\lambda_1-\mu_2+\mu_3)&=\alpha_1+\alpha_2+\alpha_3+\alpha_4+
\alpha_5,
\\
-(1/2)(\lambda_0-\lambda_1-2\lambda_2-\mu_1+\mu_3)&=\alpha_1+\alpha_2+
\alpha_3,
\\
(1/2)(\lambda_0-\lambda_1-2\lambda_2+\mu_1-\mu_3)&=\alpha_2,
\\
-(1/2)(\lambda_0-\lambda_1+2\lambda_2-\mu_1+\mu_3)&=\alpha_1+\alpha_2,
\\
(1/2)(\lambda_0-\lambda_1+2\lambda_2+\mu_1-\mu_3)&=\alpha_2+\alpha_3,
\\
-(1/2)(\lambda_0+\lambda_1+2\lambda_2+\mu_1-\mu_2)&=\alpha_4,
\\
(1/2)(\lambda_0+\lambda_1+2\lambda_2-\mu_1+\mu_2)&=\alpha_3+\alpha_4+
\alpha_5,
\\
-(1/2)(-\lambda_0-\lambda_1+2\lambda_2+\mu_1-\mu_2)&=\alpha_4+\alpha_5,
\\
(1/2)(-\lambda_0-\lambda_1+2\lambda_2-\mu_1+\mu_2)&=\alpha_3+\alpha_4
\\
-(\mu_1+(2/3)\mu)&=\alpha_3+2\alpha_4+\alpha_5+\alpha_6,
\\
-(\mu_2+(2/3)\mu)&=\alpha_6,
\\
-(\mu_3+(2/3)\mu)&=\alpha_1+2\alpha_2+2\alpha_3+2\alpha_4+\alpha_5+
\alpha_6,
\\
-(-\lambda_0-(1/2)\mu_1+(2/3)\mu)&=\alpha_2+\alpha_3+\alpha_4+\alpha_5+
\alpha_6,
\\
-(\lambda_0-(1/2)\mu_1+(2/3)\mu)&=\alpha_1+\alpha_2+\alpha_3+\alpha_4+
\alpha_6,
\\
-(-\lambda_1-(1/2)\mu_1+(2/3)\mu)&=\alpha_1+\alpha_2+\alpha_3+\alpha_4+
\alpha_5+\alpha_6,
\\
-(\lambda_1-(1/2)\mu_1+(2/3)\mu)&=\alpha_2+\alpha_3+\alpha_4+\alpha_6,
\\
-((-1/2)(-\lambda_0+\lambda_1+2\lambda_2)-(1/2)\mu_2+(2/3)\mu)&=\alpha_1+
\alpha_2+2\alpha_3+2\alpha_4+\alpha_5+\alpha_6,
\\
-((1/2)(-\lambda_0+\lambda_1+2\lambda_2)-(1/2)\mu_2+(2/3)\mu)&=\alpha_2+
\alpha_3+\2\alpha_4+\alpha_5+\alpha_6,
\\
-((-1/2)(-\lambda_0+\lambda_1-2\lambda_2)-(1/2)\mu_2+(2/3)\mu)
&=\alpha_1\alpha_2+\alpha_3+2\alpha_4+\alpha_5+\alpha_6,
\\
-((1/2)(-\lambda_0+\lambda_1-2\lambda_2)-(1/2)\mu_2+(2/3)\mu)
&=\alpha_2+2\alpha_3+2\alpha_4+\alpha_5+\alpha_6,
\\
-((-1/2)(-\lambda_0-\lambda_1-2\lambda_2)-(1/2)\mu_3+(2/3)\mu)
&=\alpha_4+\alpha_6,
\\
-((1/2)(-\lambda_0-\lambda_1-2\lambda_2)-(1/2)\mu_3+(2/3)\mu)
&=\alpha_3+\alpha_4+\alpha_5+\alpha_6,
\\
-((-1/2)(\lambda_0+\lambda_1-2\lambda_2)-(1/2)\mu_3+(2/3)\mu)
&=\alpha_4+\alpha_5+\alpha_6,
\\
-((1/2)(\lambda_0+\lambda_1-2\lambda_2)-(1/2)\mu_3+(2/3)\mu)
&=\alpha_3+\alpha_4+\alpha_6.
\end{align*}
Hence $ \varPi $ is a fundamental root system of $  ({\mathfrak{e}_7}
^C)^{\varepsilon_1,\varepsilon_2} $.

Then, for $ \varPhi,\varPhi' \in {\mathfrak{e}_7}^C$, the Killing form $
B_7 $ of $ {\mathfrak{e}_7}^C $ is given by
\begin{align*}
B_7(\varPhi,\varPhi')=B_7(\varPhi(\phi,A,B.\nu),\varPhi(\phi,A,B.
\nu))=\dfrac{3}{2} B_6(\phi,\phi')+36(A,B')+36(A',B)+24\nu\nu'
\end{align*}
(\cite[Theorem 4.5.2]{iy0}), so that for $
\varPhi_7:=\varPhi(\phi,0,0,\mu),
{\varPhi_7}':=\varPhi(\phi',0,0,\mu') \in \mathfrak{h}_7$, we have
\begin{align*}
 B_7(\varPhi_7,{\varPhi_7}')&=\dfrac{3}{2}B_6(\phi,\phi')+24\mu\mu'
\\
&=\dfrac{3}{2}(24(\lambda_0{\lambda_0}'+\lambda_1{\lambda_1}'
  +2\lambda_2{\lambda_2}')+12(\mu_1{\mu_1}'+\mu_2{\mu_2}'+\mu_3{\mu_3}'))
  +24\mu\mu'
\\
&=36(\lambda_0{\lambda_0}'+\lambda_1{\lambda_1}'
  +2\lambda_2{\lambda_2}')+18(\mu_1{\mu_1}'+\mu_2{\mu_2}'+\mu_3{\mu_3}')
  +24\mu\mu',
\end{align*}
where $ \phi:=\lambda_0(iG_{01})+
 \lambda_1(iG_{23})+\lambda_2(i(G_{45}+G_{67})+(\mu_1E_1+\mu_2E_2+
 \mu_3E_3)^\sim ,\phi'={\lambda_0}'(iG_{01})+{\lambda_1}'(iG_{23})+
{\lambda_2}'(i(G_{45}+G_{67})+({\mu_1}'E_1+{\mu_2}'E_2+{\mu_3}'E_3)^\sim $.
Now, the canonical elements $ \varPhi_{\alpha_1},
\varPhi_{\alpha_2}, \allowbreak \varPhi_{\alpha_3}, \varPhi_{\alpha_4},
\varPhi_{\alpha_5},\varPhi_{\alpha_6}
$ corresponding to $ \alpha_1, \alpha_2,\alpha_3,\alpha_4,\alpha_5,\alpha_6
$ are determined as follows:
\begin{align*}
\varPhi_{\alpha_1}&=\varPhi(-\dfrac{1}{36}(iG_{01})+\dfrac{1}{36}(iG_{23}),
0,0,0),
\\
\varPhi_{\alpha_2}&=\varPhi(\dfrac{1}{72}(iG_{01})-\dfrac{1}{72}(iG_{23})-
\dfrac{1}{72}(i(G_{45}+G_{67}))+(\dfrac{1}{36}E_1-\dfrac{1}{36}E_3)^\sim,
0,0,0),
\\
\varPhi_{\alpha_3}&=\varPhi(\dfrac{1}{36}(i(G_{45}+G_{67})),0,0,0),
\\
\varPhi_{\alpha_4}&=\varPhi(-\dfrac{1}{72}(iG_{01})-\dfrac{1}{72}(iG_{23})-
\dfrac{1}{72}(i(G_{45}+G_{67}))+(-\dfrac{1}{36}E_1+\dfrac{1}{36}E_2)^\sim,
0,0,0),
\\
\varPhi_{\alpha_5}&=\varPhi(\dfrac{1}{36}(iG_{01})+\dfrac{1}{36}(iG_{23}),
0,0,0),
\\
\varPhi_{\alpha_6}&=\varPhi(\dfrac{1}{54}(E_1-2E_2+E_3)^\sim,
0,0,\dfrac{1}{36}).
\end{align*}

Hence we have the following
\begin{align*}
    (\alpha_1,\alpha_1)&=B_7(\varPhi_{\alpha_1},
    \varPhi_{\alpha_1})=36\left(
    \left(-\dfrac{1}{36} \right)^2+\left(\dfrac{1}{36} \right)^2
    \right)=\dfrac{1}{18},
    \\
    (\alpha_1,\alpha_2)&=B_7(\varPhi_{\alpha_1},
    \varPhi_{\alpha_2})=36\left(
    \left(-\dfrac{1}{36} \right)\left(\dfrac{1}{72} \right)
    +\left(\dfrac{1}{36} \right)\left(-\dfrac{1}{72} \right)
    \right)=-\dfrac{1}{36},
    \\
    (\alpha_1,\alpha_3)&=B_7(\varPhi_{\alpha_1},\varPhi_{\alpha_3})=0,
    \\
    (\alpha_1,\alpha_4)&=B_7(\varPhi_{\alpha_1},
    \varPhi_{\alpha_4})=36\left(
    \left(-\dfrac{1}{36} \right)\left(-\dfrac{1}{72} \right)
    +\left(\dfrac{1}{36} \right)\left(-\dfrac{1}{72} \right)
    \right)=0,
    \\
    (\alpha_1,\alpha_5)&=B_7(\varPhi_{\alpha_1},
    \varPhi_{\alpha_5})=36\left(
    \left(-\dfrac{1}{36} \right)\left(\dfrac{1}{36} \right)+\left(\dfrac{1}
    {36} \right)^2 \right)=0,
    \\
   (\alpha_1,\alpha_6)&=B_7(\varPhi_{\alpha_1},\varPhi_{\alpha_6})=0,
    \\
    (\alpha_2,\alpha_2)&=B_7(\varPhi_{\alpha_2},
    \varPhi_{\alpha_2})=36\left(
    \left(\dfrac{1}{72} \right)^2+\left(-\dfrac{1}{72} \right)^2+
    2\left(-\dfrac{1}{72} \right)^2 \right)+18\left(
    \left(\dfrac{1}{36} \right)^2+\left(-\dfrac{1}{36} \right)^2
    \right)=\dfrac{1}{18},
    \\
    (\alpha_2,\alpha_3)&=B_7(\varPhi_{\alpha_2},
    \varPhi_{\alpha_3})=36\cdot 2\left(-\dfrac{1}{72} \right)\left(
    \dfrac{1}{36}\right)=-\dfrac{1}{36},
    \\
    (\alpha_2,\alpha_4)&=B_6(\varPhi_{\alpha_2},
    \varPhi_{\alpha_4})=36\left(
    \left(\dfrac{1}{72} \right)\left(-\dfrac{1}{72} \right)+\left(-
    \dfrac{1}{72} \right)^2+
    2\left(-\dfrac{1}{72} \right)^2 \right)+18
    \left(\dfrac{1}{36} \right)\left(-\dfrac{1}{36} \right)=0,
    \\
    (\alpha_2,\alpha_5)&=B_6(\varPhi_{\alpha_2},
    \varPhi_{\alpha_5})=36\left(
    \left(\dfrac{1}{72} \right)\left(\dfrac{1}{36} \right)+\left(-\dfrac{1}
    {72} \right)\left(\dfrac{1}{36} \right) \right)=0,
    \\
    (\alpha_2,\alpha_6)&=B_7(\varPhi_{\alpha_2},
    \varPhi_{\alpha_6})=18\left(
    \left(\dfrac{1}{36} \right)\left(\dfrac{1}{54} \right)+\left(-\dfrac{1}
    {36} \right)\left(\dfrac{1}{54} \right) \right)=0,
    \\
    (\alpha_3,\alpha_3)&=B_7(\varPhi_{\alpha_3},\varPhi_{\alpha_3})=36\cdot
    2\left(\dfrac{1}{36}\right)^2=\dfrac{1}{18},
    \\
    (\alpha_3,\alpha_4)&=B_7(\varPhi_{\alpha_3},\varPhi_{\alpha_4})=36\cdot
    2\left(\dfrac{1}{36}\right)\left(-\dfrac{1}{72} \right)=-\dfrac{1}{36},
    \\
    (\alpha_3,\alpha_5)&=B_7(\varPhi_{\alpha_3},\varPhi_{\alpha_5})=0,
    \\
    (\alpha_3,\alpha_6)&=B_7(\varPhi_{\alpha_3},\varPhi_{\alpha_6})=0,
    \\
    (\alpha_4,\alpha_4)&=B_7(\varPhi_{\alpha_4},
    \varPhi_{\alpha_4})=36\left(
    \left(-\dfrac{1}{72} \right)^2+\left(-\dfrac{1}{72} \right)^2+
    2\left(-\dfrac{1}{72} \right)^2 \right)+18\left(
    \left(-\dfrac{1}{36} \right)^2+\left(\dfrac{1}{36} \right)^2
    \right)=\dfrac{1}{18},
    \\
    (\alpha_4,\alpha_5)&=B_7(\varPhi_{\alpha_4},
    \varPhi_{\alpha_5})=36\left(
    \left(-\dfrac{1}{72} \right)\left(\dfrac{1}{36} \right)+\left(-
    \dfrac{1}{72} \right)\left(\dfrac{1}{36} \right) \right)=-\dfrac{1}
    {36},
    \\
    (\alpha_4,\alpha_6)&=B_7(\varPhi_{\alpha_4},
    \varPhi_{\alpha_6})=18\left(
    \left(-\dfrac{1}{36} \right)\left(\dfrac{1}{54} \right)+\left(\dfrac{1}
    {36} \right)\left(-\dfrac{2}{54} \right) \right)=-\dfrac{1}{36},
    \\
    (\alpha_5,\alpha_5)&=B_7(\varPhi_{\alpha_5},
    \varPhi_{\alpha_5})=36\left(
    \left(\dfrac{1}{36} \right)^2+\left(\dfrac{1}{36} \right)^2
    \right)=\dfrac{1}{18},
    \\
    (\alpha_5,\alpha_6)&=B_7(\varPhi_{\alpha_5},\varPhi_{\alpha_6})=0,
    \\
    (\alpha_6,\alpha_6)&=B_7(\varPhi_{\alpha_6},
    \varPhi_{\alpha_6})=18\left(
    \left(\dfrac{1}{54} \right)^2+\left(-\dfrac{2}{54} \right)^2+
    \left(\dfrac{1}{54} \right)^2 \right)+24\left(\dfrac{1}{36} \right)^2
    =\dfrac{1}{18}.
\end{align*}

Thus, using the inner products above, we have
\begin{align*}
    \cos\theta_{12}&=\dfrac{(\alpha_1,\alpha_2)}{\sqrt{(\alpha_1,\alpha_1)
    (\alpha_2,\alpha_2)}}=-\dfrac{1}{2},\quad
    \cos\theta_{13}=\dfrac{(\alpha_1,\alpha_3)}{\sqrt{(\alpha_1,\alpha_1)
    (\alpha_3,\alpha_3)}}=0,
    \\
    \cos\theta_{14}&=\dfrac{(\alpha_1,\alpha_4)}{\sqrt{(\alpha_1,\alpha_1)
    (\alpha_4,\alpha_4)}}=0,\quad
    \cos\theta_{15}=\dfrac{(\alpha_1,\alpha_5)}{\sqrt{(\alpha_1,\alpha_1)
    (\alpha_5,\alpha_5)}}=0,
    \\
    \cos\theta_{16}&=\dfrac{(\alpha_1,\alpha_6)}{\sqrt{(\alpha_1,\alpha_1)
    (\alpha_6,\alpha_6)}}=0, \quad
    \cos\theta_{23}=\dfrac{(\alpha_2,\alpha_3)}{\sqrt{(\alpha_2,\alpha_2)
    (\alpha_3,\alpha_3)}}=-\dfrac{1}{2},
    \\
    \cos\theta_{24}&=\dfrac{(\alpha_2,\alpha_4)}{\sqrt{(\alpha_2,\alpha_2)
    (\alpha_4,\alpha_4)}}=0,\quad \cos\theta_{25}=\dfrac{(\alpha_2,
    \alpha_5)}{\sqrt{(\alpha_2,\alpha_2)
    (\alpha_5,\alpha_5)}}=0
    \\
    \cos\theta_{26}&=\dfrac{(\alpha_2,\alpha_6)}{\sqrt{(\alpha_2,\alpha_2)(\alpha_6,\alpha_6)}}=0,\quad
    \cos\theta_{34}=\dfrac{(\alpha_3,\alpha_4)}{\sqrt{(\alpha_3,\alpha_3)
    (\alpha_4,\alpha_4)}}=-\dfrac{1}{2},
\\
    \cos\theta_{35}&=\dfrac{(\alpha_3,\alpha_5)}{\sqrt{(\alpha_3,\alpha_3)(\alpha_5,\alpha_5)}}=0,\quad
     \cos\theta_{36}=\dfrac{(\alpha_3,\alpha_6)}{\sqrt{(\alpha_3,\alpha_3)(\alpha_6,\alpha_6)}}=0,
\\
     \cos\theta_{45}&=\dfrac{(\alpha_4,\alpha_5)}{\sqrt{(\alpha_4,\alpha_4)(\alpha_5,\alpha_5)}}=-\dfrac{1}{2},\quad
    \cos\theta_{46}=\dfrac{(\alpha_4,\alpha_6)}{\sqrt{(\alpha_4,\alpha_4)
    (\alpha_6,\alpha_6)}}=-\dfrac{1}{2},
\\
    \cos\theta_{56}&=\dfrac{(\alpha_5,\alpha_6)}{\sqrt{(\alpha_5,\alpha_5)(\alpha_6,\alpha_6)}}=0.
\end{align*}
so that we can draw the required Dynkin diagram.
\end{proof}

\section{The group $ E_{8,\sH} $ and the root system, the Dynkin
diagram of the Lie algebra $ ({\mathfrak{e}_{8}}^C)^{\varepsilon_1,
\varepsilon_2} $}

We consider the following complex Lie algebra $ (\mathfrak{e}_{8,\sH})^C $
which is given by replacing $ \mathfrak{C} $ with $ \H $ in the complex Lie
algebra $ {\mathfrak{e}_8}^C $:
\begin{align*}
(\mathfrak{e}_{8,\sH})^C&=(\mathfrak{e}_{7,\sH})^C \oplus (\mathfrak{P}
_{\sH})^C \oplus (\mathfrak{P}_{\sH})^C \oplus C \oplus C \oplus C
\\
&=\left\lbrace R=(\varPhi,P,Q,r,s,t) \relmiddle{|}
\begin{array}{l}
\varPhi \in (\mathfrak{e}_{7,\sH})^C, P,Q \in (\mathfrak{P}_{\sH})^C, \\
r,s,t \in C
\end{array}
\right\rbrace
\end{align*}
In particular, we have $ \dim_C((\mathfrak{e}_{8,\sH})^C)=66+32\times
2+3=133 $. In $ (\mathfrak{e}_{8,\sH})^C $, we can define a Lie bracket and
can prove its simplicity as in $ {\mathfrak{e}_8}^C $.
In addition, we define a symmetric inner product $ (R_1,R_2)_8 $ by
\begin{align*}
(R_1,R_2)_8:=(\varPhi_1,\varPhi_2)_7-\{Q_1, P_2\}+
\{P_1,Q_2\}-8r_1r_2-4t_1s_2-4s_1t_2,
\end{align*}
where $ R_i:=(\varPhi_i,P_i,Q_i,r_i,s_i,t_i) \in (\mathfrak{e}_{8,
\sH})^C,i=1,2 $. Then $ (\mathfrak{e}_{8,\sH})^C $ leaves the symmetric
inner product $ (R_1,R_2)_8 $ invariant (cf. \cite[Lemma 5.3.1]{iy0}).

Then we have the following theorem.

\begin{theorem}\label{theorem 8.1}
    The Killing form $ B_{8,\sH} $ of $ (\mathfrak{e}_{8,\sH})^C $ is
    given by
    \begin{align*}
    B_{8,\sH}(R_1,R_2)&=-9(R_1,R_2)_8
    \\
    &=-9(\varPhi_1,\varPhi_2)_7+9\{Q_1,
    P_2\}-9\{P_1,Q_2\}+72r_1r_2+36t_1s_2+36s_1t_2
    \\
    &=\dfrac{9}{5}B_{7,\sH}(\varPhi_1,\varPhi_2)+9\{Q_1,
    P_2\}-9\{P_1,Q_2\}+72r_1r_2+36t_1s_2+36s_1t_2.
    \end{align*}
\end{theorem}
\begin{proof}
    Since $ (\mathfrak{e}_{8,\sH})^C $ is simple, there exists $k \in C$
    such that
    \begin{align*}
    B_{8,\sH}(R_1, R_2) = k(R_1, R_2)_8, \quad R_i \in
    (\mathfrak{e}_{8,\sH})^C.
    \end{align*}
    We will determine $k$. Let $ R_1 = R_2= (0, 0, 0, 1, 0, 0) =
    \tilde{1}$. Then we have
    \begin{align*}
    (\tilde{1},\tilde{1})_8=-8.
    \end{align*}
    On the other hand, it follows from
    \begin{align*}
    (\ad\,\tilde{1})(\ad\,\tilde{1})R
    &=[\tilde{1}, [\tilde{1}, (\varPhi, P, Q, r, s, t)]\,],\,\, R \in
    (\mathfrak{e}_{8,\sH})^C
    \\
    &= [\tilde{1}, (0, P, - Q, 0, 2s, - 2t)]
    \\
    &= (0, P, Q, 0, 4s, 4t)
    \end{align*}
    that
    \begin{align*}
    B_{8,\sH}(\tilde{1}, \tilde{1})=\tr((\ad\,\tilde{1})(\ad\,\tilde{1})) = 32 \times 2 + 4 \times 2 = 72.
    \end{align*}
    Thus we have $k = - 9$.

    Therefore we have
    \begin{align*}
    B_{8,\sH}(R_1, R_2) = -9(R_1, R_2)_8
    \end{align*}
    and together with Theorem \ref{theorem 7.1}, the remainders of formulas can be easily obtained.
\end{proof}


Now, we consider a Lie group $ E_{8,\sH} $ which is defined by
\begin{align*}
E_{8,\sH}=\left\lbrace \alpha \in \Iso_{C}((\mathfrak{e}_{8,\sH})^C) \relmiddle{|}\alpha[R_1,R_2]=[\alpha R_1. \alpha R_2],
\langle \alpha R_1,\alpha R_2 \rangle=\langle R_1,R_2\rangle
\right\rbrace,
\end{align*}
where $ \langle R_1,R_2\rangle:=(-1/15)B_{8,\sH}((\tau\lambda_\omega)R_1,R_2) $. The group $ E_{8,\sH} $ is a compact Lie group as a closed subgroup of the unitary group $ U(133)=U((\mathfrak{e}_{8,\sH})^C)=\{\alpha \in \Iso_{C}((\mathfrak{e}_{8,\sH})^C) \,|\, \langle \alpha R_1,\alpha R_2 \rangle=\langle R_1,R_2\rangle \} $. Moreover, $ E_{8,\sH} $ is nothing but the group which is replaced $ \mathfrak{C} $ with $ \H $ in the definition of the compact Lie group of type $ E_8 $ (Subsection 3.5). The main aim of this section is to determine the structure of the group $ E_{8,\sH} $.
\vspace{1mm}

We give the explicit form of the Lie algebra $ \mathfrak{e}_{8,\sH} $ of the group $ E_{8,\sH} $ as lemma.

\begin{lemma}\label{lemma 8.2}
The Lie algebra $ \mathfrak{e}_{8,\sH} $ of the group $ E_{8,\sH} $ is given by
\begin{align*}
\mathfrak{e}_{8,\sH}=\left\lbrace R=(\varPhi,P, -\tau\lambda P,r,s,-\tau s) \relmiddle{|} \varPhi \in \mathfrak{e}_{7,\sH},P \in (\mathfrak{P}_{\sH})^C, r \in i\R, s \in C \right\rbrace.
\end{align*}

In particular, we have $ \dim(\mathfrak{e}_{8,\sH})=66+32 \times 2+1+2=133 $.
\end{lemma}
\begin{proof}
We can confirm $ ((\mathfrak{e}_{8,\sH})^C)^{\tau\lambda_\omega}=\mathfrak{e}_{8,\sH} $ as that in $ ({\mathfrak{e}_8}^C)^{\tau\lambda_\omega}=\mathfrak{e}_8 $. For $ R=(\varPhi,P,Q,r,s,t) \in (\mathfrak{e}_{8,\sH})^C $, we have the required result by getting $ R $ satisfying $ (\tau\lambda_\omega)R=R $.
\end{proof}

First, we will study the complexification $ (E_{8,\sH})^C $ of $ E_{8,\sH} $:
\begin{align*}
(E_{8,\sH})^C&=\Aut((\mathfrak{e}_{8,\sH})^C)
\\
&=\left\lbrace \alpha \in \Iso_{C}((\mathfrak{e}_{8,\sH})^C)
\relmiddle{|}\alpha[R_1,R_2]=[\alpha R_1,\alpha R_2]\right\rbrace.
\end{align*}

The immediate aim is to prove the connectedness of the group $
(E_{8,\sH})^C $. In order to prove it, we will use the manner used in
\cite{iy7}.

First, we consider a subgroup $
((E_{8,\sH})^C)_{\tilde{1},1^-,1_-} $ of $ (E_{8,\sH})^C $:
\begin{align*}
((E_{8,\sH})^C)_{\tilde{1},1^-,1_-}=\left\lbrace \alpha \in (E_{8,\sH})^C
\relmiddle{|} \alpha \tilde{1}=\tilde{1},\alpha 1^-=1^-, \alpha
1_-=1_-\right\rbrace.
\end{align*}

Then we have the following proposition.

\begin{proposition}\label{proposition 8.3}
    The group $ ((E_{8,\sH})^C)_{\tilde{1},1^-,1_-} $ is isomorphic to the group $ (E_{7,\sH})^C ${\rm :} \\ $ ((E_{8,\sH})^C)_{\tilde{1},1^-,1_-}
    \allowbreak \cong (E_{7,\sR})^C $.
\end{proposition}
\begin{proof}
    This proposition can be proved as in the proof of \cite[Proposition
    5.7.1]{iy0} by replacing $ \mathfrak{C} $ with $ \H $.
\end{proof}

Subsequently, we consider a subgroup $ ((E_{8,\sH})^C)_{1_-}
$ of $ (E_{8,\sH})^C $:
\begin{align*}
((E_{8,\sH})^C)_{1_-}=\left\lbrace  \alpha \in (E_{8,\sH})^C
\relmiddle{|}\alpha 1_-=1_- \right\rbrace.
\end{align*}

Then we prove the following lemma.

\begin{lemma}\label{lemma 8.4}


    The Lie algebra $ ((\mathfrak{e}_{8,\sH})^C)_{1_-} $ of the
    group $ ((E_{8,\sH})^C)_{1_-} $ is given by
    \begin{align*}
    ((\mathfrak{e}_{8,\sH})^C)_{1_-}
    &=\left\lbrace R \in (\mathfrak{e}_{8,\sH})^C \relmiddle{|} [R, 1_-]=0
    \right\rbrace
    \\
    &=\left\lbrace (\varPhi, 0, Q, 0, 0, t) \in (\mathfrak{e}_{8,\sH})^C
    \relmiddle{|}
    \begin{array}{l}
    \varPhi \in (\mathfrak{e}_{7,\sH})^C,
    Q \in (\mathfrak{P}_{\sH})^C,
    t \in C
    \end{array} \right\rbrace .
    \end{align*}

    In particular, we have $ \dim_C(((\mathfrak{e}_{8,\sH})^C)_{1_-}) = 66+32+1 = 99 $.
\end{lemma}
\begin{proof}

     By the straightforward computation, we can easily obtain the required results.
\end{proof}

We prove the following proposition needed in the proof of connectedness.

\begin{proposition}\label{proposition 8.5}
    The group $((E_{8,\sH})^C)_{1_-}$ is a semi-direct product of groups
    $\exp(\ad(((\mathfrak{P}_{\sH})^C)_- \oplus C_- ))$ and $
    (E_{7,\sH})^C ${\rm:}
    \begin{align*}
    ((E_{8,\sH})^C)_{1_-}=\exp(\ad(((\mathfrak{P}_{\sH})^C)_- \oplus C_-
    )) \rtimes (E_{7,\sH})^C.
    \end{align*}

    In particular, the group $((E_{8,\sH})^C)_{1_-}$ is connected.
\end{proposition}
\begin{proof}
     Let $((\mathfrak{P}_{\sH})^C)_- \oplus C_- = \{(0, 0, L, 0, 0, v) \,
     | \, L \in (\mathfrak{P}_{\sH})^C, v \in C\}$ be a Lie subalgebra of
     the Lie algebra
    $((\mathfrak{e}_{8,\sH})^C)_{1_-}$ (Lemma \ref{lemma 8.4}).
    Since it follows from $[L_-, v_-] = 0$ that $\ad(L_-)$ commutes with
    $\ad(v_-)$, we have $\exp(\ad(L_- + v_-)) =
    \exp(\ad(L_-))\exp(\ad(v_-))$. Hence the group
    $\exp(\ad(((\mathfrak{P}_{\sH})^C)_- \oplus C_-))$ is the connected
    subgroup of the group $((E_{8,\sH})^C)_{1-}$.

    Now, let $\alpha \in ((E_{8,\sH})^C)_{1-}$ and set $ \alpha\tilde{1}:=
    (\varPhi_1, P_1, Q_1, r_1, s_1, t_1), \alpha1^-:= (\varPhi_2, P_2, Q_2, r_2, s_2,t_2) $. Then, it follows from
    \begin{align*}
    [\alpha\tilde{1}, 1_-]&= \alpha[\tilde{1}, 1_-] = -2\alpha 1_- =
    -21_-=(0,0,0,0,0,-2),
    \\
    [\alpha\tilde{1},
    1_-]&=[(\varPhi_1,P_1,Q_1,r_1,s_1,t_1),(0,0,0,0,0,1)]=(0,0,-P_1,s_1,0,-2r_1)
    \end{align*}
    that $ P_1=0, r_1=1, s_1=0 $. Similarly, from $ [\alpha1^-, 1_-] =
    \alpha[1^-, 1_-] = \alpha\tilde{1} $, we have $ \varPhi_1=0, P_2=-Q_1,
    s_2=r_1,  s_1=0,t_1=-2r_2$. Hence we have
    \begin{align*}
    \alpha \tilde{1}=(0,0,Q_1,1,0,t_1),\quad \alpha
    1^-=(\varPhi_2,-Q_1,Q_2,-\dfrac{t_1}{2},1,t_2).
    \end{align*}
    Moreover, from $[\alpha\tilde{1}, \alpha1^-] = \alpha[\tilde{1}, 1^-]
    = 2\alpha1^-$, we obtain the following
\begin{align*}
\varPhi_2 = \dfrac{1}{2}Q_1 \times Q_1, \; Q_2 = - \dfrac{t}{2}Q_1 -\dfrac{1}{3}\varPhi_2Q_1, \; t_2 = -\dfrac{{t_1}^2}{4} - \dfrac{1}{16}\{Q_1,Q_2\},
\end{align*}
that is,
\begin{align*}
\alpha 1^-=(\dfrac{1}{2}Q_1 \times Q_1,-Q_1,-\dfrac{t}{2}Q_1 -\dfrac{1}{6}(Q_1\times Q_1)Q_1 ,-\dfrac{t_1}{2},1,-\dfrac{{t_1}^2}{4} + \dfrac{1}{96}\{Q_1, (Q_1\times Q_1)Q_1).
\end{align*}
    On the other hand, for $ \exp(\ad((v/2)_-))\exp(\ad(L_-)) \in
    \exp(\ad(((\mathfrak{P}_{\sH})^C)_- \oplus C_-)) $, we have
    \begin{align*}
    &\quad \exp\Big(\ad(\Big(\dfrac{v}{2}\Big)_-)\Big)\exp(\ad(L_-))1^-
    \\
    &=(\dfrac{1}{2}L \times L,- L,
    - \dfrac{t}{2}L - \dfrac{1}{6}(L \times L)L ,-\dfrac{v}{2}, 1,
    -\dfrac{v^2}{4} + \dfrac{1}{96}\{L, (L \times L)L\}),
    \end{align*}
so that set $ \delta:= \exp(\ad((t_1/2)_-))\exp(\ad({Q_1}_-)) $, then we see
\begin{align*}
(\delta^{-1}\alpha) 1^-=1^-.
\end{align*}
In addition, we can confirm
\begin{align*}
  (\delta^{-1}\alpha) \tilde{1} = \tilde{1},\,\, (\delta^{-1}\alpha) 1_-
  =1_-.
\end{align*}
    Indeed, it follows that
    \begin{align*}
    \exp\Big(\ad(\Big(\dfrac{v}{2}\Big)_-)\Big)\exp(\ad(L_-))\tilde{1}&=\exp\Big(\ad(\Big(\dfrac{v}{2}\Big)_-)\Big)(0,0,L,1,0,0)=(0,0,L,1,0,v),
    \\
    \exp\Big(\ad(\Big(\dfrac{v}{2}\Big)_-)\Big)\exp(\ad(L_-))1_-&=\exp\Big(\ad(\Big(\dfrac{v}{2}\Big)_-)\Big)1_-=1_-,
    \end{align*}
    so that we have $ (\delta^{-1}\alpha)\tilde{1}=\delta^{-1}(0,0,P,1,0,t)=\tilde{1},(\delta^{-1}\alpha)1_-=\delta^{-1}1_-=1_-$.

\noindent Hence, $ \delta^{-1}\alpha \in
    ((E_{8,\sH})^C)_{\tilde{1},1^-,1_-} \cong (E_{7,\sH})^C $ follows from Proposition \ref{proposition 8.3},
    so that we obtain
\begin{align*}
((E_{8,\sH})^C)_{1_-} = \exp(\ad(((\mathfrak{P}_{\sH})^C)_- \oplus
    C_-))(E_{7,\sH})^C.
\end{align*}

    Furthermore, for $\beta \in (E_{7,\sH})^C$, it is easy to verify that
\begin{align*}
\beta(\exp(\ad(L_-)))\beta^{-1} = \exp(\ad(\beta L_-)),\quad
    \beta((\exp(\ad(v_-)))\beta^{-1} = \exp(\ad(v_-)).
\end{align*}
    Indeed, for $(\varPhi, P, Q, r, s, t) \in (\mathfrak{e}_{8,\sH})^C$,
    by doing simple computation, we have the following
    \begin{align*}
    \beta \ad (L_-) \beta^{-1}(\varPhi, P, Q, r, s, t)&= \beta [L_-,
    \beta^{-1}(\varPhi, P, Q, r, s, t)]
    \\[1mm]
    &= [\beta L_-,\beta \beta^{-1}(\varPhi, P, Q, r, s, t)]\,(\beta \in
    (E_{7,\sH})^C \subset (E_{8,\sH})^C )
    \\[1mm]
    &= [\beta L_-,(\varPhi, P, Q, r, s, t)]
    \\[1mm]
    &= \ad (\beta L_-) (\varPhi, P, Q, r, s, t),
    \end{align*}
    that is, $\beta \ad (L_-) \beta^{-1}= \ad (\beta L_-)$. It is clear that the latter formula holds.

    \noindent Hence we have
    \begin{align*}
    \beta(\exp(\ad(L_-)))\beta^{-1}&=\beta \,
    \Bigl(\displaystyle{\sum_{n=0}^{\infty}\dfrac{1}{n!}\ad (L_-)^n}
    \Bigr)\,\beta^{-1}
    \\[2mm]
    &=\displaystyle{\sum_{n=0}^{\infty}\dfrac{1}{n!}(\beta\ad
    (L_-)\beta^{-1})^n}\,\,\,(\, \beta \ad (L_-) \beta^{-1}= \ad (\beta
    L_-))
    \\[2mm]
    &=\displaystyle{\sum_{n=0}^{\infty}\dfrac{1}{n!}(\ad (\beta L_-))^n}
    \\[2mm]
    &= \exp(\ad(\beta L_-)),
    \end{align*}
that is, $ \beta(\exp(\ad(L_-)))\beta^{-1}=\exp(\ad(\beta L_-)) $.
    By the argument similar to above, we also see that
    $\beta((\exp(\ad(v_-))) \beta^{-1} = \exp(\ad(v_-)$.
Thus, for $ \alpha \in ((E_{8,\sH})^C)_{1_-} $, it follows from $ \alpha=\exp(N_-+u_-)\beta \in \exp(\ad(((\mathfrak{P}_{\sH})^C)_- \oplus
    C_-))\oplus (E_{7,\sH})^C=((E_{8,\sH})^C)_{1_-} $ that
\begin{align*}
&\quad \alpha\exp(\ad(L_-+v_-))\alpha^{-1}
\\
&=(\exp(N_-+u_-)\beta)\exp(\ad(L_-+v_-))(\exp(N_-+u_-)\beta)^{-1}
\\
&=(\exp(\ad(N_-))\exp(\ad(u_-))\beta)(\exp(\ad(L_-))\exp(\ad(v_-)))(\exp(\ad(N_-))\exp(\ad(u_-))\beta)^{-1}
\\
&=(\exp(\ad(N_-))\exp(\ad(u_-))\beta)(\exp(\ad(L_-))\exp(\ad(v_-)))(\beta^{-1}\exp(\ad(u_-))^{-1}\exp(\ad(N_-))^{-1})
\\
&=(\exp(\ad(N_-))\exp(\ad(u_-)))(\exp(\ad(\beta L_-))\exp(\ad(v_-)))(\exp(\ad(u_-))^{-1}\exp(\ad(N_-))^{-1})
\\
&=(\exp(\ad(N_-))\exp(\ad(\beta L_-))\exp(\ad(N_-))^{-1})(\exp(\ad(u_-))\exp(\ad(v_-))\exp(\ad(u_-))^{-1})
\\
&\in \exp(\ad(((\mathfrak{P}_{\sH})^C)_-)\exp(\ad(C_-))=\exp(\ad(((\mathfrak{P}_{\sH})^C)_- \oplus C_-)).
\end{align*}

\noindent This shows that $\exp(\ad(((\mathfrak{P}_{\sH})^C)_- \oplus
C_-)) = \exp(\ad(((\mathfrak{P}_{\sH})^C)_-)\exp(\ad(C_-))$ is a normal \vspace{0.5mm}subgroup of the group$((E_{8,\sH})^C)_{1_-}$.
In addition, we have a split exact sequence
\begin{align*}
1 \to \exp(\ad((\mathfrak{P}_{\sH})^C)_- \oplus C_-))
 \overset{j}{\longrightarrow}((E_{8,\sH})^C)_{1_-}
\overset{\overset{p}{\scalebox{1.0}{$\longrightarrow$}}}{\underset{s}{\longleftarrow}} (E_{7,\sH})^C \to 1.
\end{align*}
Indeed, first we define a mapping $ j $ by $ j(\delta)=\delta $. Then it is
clear that $ j $ is a injective homomorphism, and subsequently, define a
mapping $ p $ by $ p(\alpha):=p(\delta\beta)=\beta $. Then, since  it
follows that
\begin{align*}
p(\alpha_1\alpha_2)&=p((\delta_1\beta_1)(\delta_2\beta_2))
\\
&=p(\delta_1(\beta_1\delta_2{\beta_1}^{-1})\beta_1\beta_2)
\;\;(\beta_1\delta_2{\beta_1}^{-1}=:{\delta_2}' \in \exp(\ad((\mathfrak{P}_{\sH})^C)_- \oplus C_-))
\\
&=p((\delta_1{\delta_2}')(\beta_1\beta_2))
\\
&=\beta_1\beta_2
\\
&=p(\alpha_1)p(\alpha_2),
\end{align*}
$ p $ is a homomorphism. Subsequently, let $ \beta \in (E_{7,\sH})^C $. We
choose $ \delta \in \exp(\ad((\mathfrak{P}_{\sH})^C)_- \oplus C_-))$, then there exists $ \alpha \in ((E_{8,\sH})^C)_{1_-} $ such that $
\alpha=\delta\beta $. This implies that $ p $ is surjective. Finally, we
define a mapping $ s $ by $ s(\beta)=\beta$. Then it is clear that $ s $
is also a injective homomorphism and we have
\begin{align*}
    ps(\beta)=p(\beta)=\beta,
\end{align*}
that is, $ ps=1 $. With above, the short sequence is a split exact
sequence.

Hence the group  $((E_{8,\sH})^C)_{1_-}$ is a semi-direct product of
$\exp(\ad(((\mathfrak{P}_{\sH})^C)_- \oplus C_-))$ and $(E_{7,\sH})^C$:
\begin{align*}
((E_{8,\sH})^C)_{1_-} =
    \exp(\ad(((\mathfrak{P}_{\sH})^C)_- \oplus C_-))\rtimes (E_{7,\sH})^C.
\end{align*}

Finally, the connectedness of $((E_{8,\sH})^C)_{1_-}$ follows from
the connectedness of  $\exp(\ad(((\mathfrak{P}_{\sH})^C)_- \allowbreak \oplus C_-))$
and $(E_{7,\sH})^C $(Theorem \ref{theorem 7.8}).
\end{proof}

Before proving the connectedness of the group $ (E_{8,\sH})^C $, we will make some preparations.

For $R \in (\mathfrak{e}_{8,\sH})^C$, we define a $C$-linear mapping $R
\times R : (\mathfrak{e}_{8,\sH})^C \to (\mathfrak{e}_{8,\sH})^C $ by
\begin{align*}
(R \times R)R_1 = [R, [R, R_1]\,] + \dfrac{1}{18}B_{8,\sH}(R, R_1)R,\,\,
R_1 \in (\mathfrak{e}_{8,\sH})^C,
\end{align*}
where $B_{8,\sH}$ is the Killing form of $ (\mathfrak{e}_{8,\sH})^C $ (Theorem \ref{theorem 8.1}).
Using this mapping, we define a space $\mathfrak{W}_{\sH}$ by
\begin{align*}
(\mathfrak{W}_{\sH})^C = \{R \in (\mathfrak{e}_{8,\sH})^C \, | \, R \times R =
0, R \not= 0\}.
\end{align*}


\begin{lemma}\label{lemma 8.6}
    For $R = (\varPhi, P, Q, r, s, t) \in (\mathfrak{e}_{8,\sH})^C$,
    $R \not=0$ belongs to $(\mathfrak{e}_{8,\sH})^C$ if and only if $R$
    satisfies the following conditions
    \vspace{3mm}

    {\rm (1)} $2s\varPhi - P \times P = 0$ \quad {\rm (2)} $2t\varPhi + Q
    \times Q = 0$ \quad {\rm (3)} $2r\varPhi + P \times Q = 0$
    \vspace{1mm}

    {\rm (4)} $\varPhi P -3rP - 3sQ = 0 $ \quad
    {\rm (5)} $\varPhi Q + 3rQ - 3tP = 0 $ \quad {\rm (6)} $\{P, Q\} -
    16(st + r^2) = 0$
    \vspace{1mm}

    {\rm (7)} $2(\varPhi P \times Q_1 + 2P \times \varPhi Q_1 - rP \times
    Q_1 - sQ \times Q_1) - \{P, Q_1\}\varPhi = 0$
    \vspace{1mm}

    {\rm (8)} $2(\varPhi Q \times P_1 + 2Q \times \varPhi P_1 + rQ \times
    P_1 - tP \times P_1) \!- \{Q, P_1\}\varPhi = 0$
    \vspace{1mm}

    {\rm (9)} $8((P \times Q_1)Q - stQ_1 - r^2Q_1 - \varPhi^2Q_1 +
    2r\varPhi Q_1) + 5\{P, Q_1\}Q
    -2\{Q, Q_1\}P = 0$
    \vspace{1mm}

    \hspace*{-1.7mm}{\rm (10)} $8((Q \times P_1)P + stP_1 + r^2P_1 +
    \varPhi^2P_1 + 2r\varPhi P_1) \,+ \,5\{Q, P_1\}P
    -2\{P, P_1\}Q= 0$
    \vspace{1mm}

    \hspace*{-1.7mm}{\rm (11)} $10(\ad\,\varPhi)^2\varPhi_1 + Q \times
    \varPhi_1P - P \times \varPhi_1Q) +   B_{7,\sH}(\varPhi,
    \varPhi_1)\varPhi = 0$
    \vspace{1mm}

    \hspace*{-1.7mm}{\rm (12)} $10(\varPhi_1\varPhi P -2\varPhi\varPhi_1P
    - r\varPhi_1P - s\varPhi_1Q) + B_{7,\sH}(\varPhi, \varPhi_1)P = 0$
    \vspace{1mm}

    \hspace*{-1.7mm}{\rm (13)} $10(\varPhi_1\varPhi Q -2\varPhi\varPhi_1Q
    + r\varPhi_1Q - t\varPhi_1P) + B_{7,\sH}(\varPhi, \varPhi_1)Q = 0,$

    \vspace{1mm}
    \noindent for all $\varPhi_1 \in(\mathfrak{e}_{7,\sH})^C, P_1, Q_1 \in
    (\mathfrak{P}_{\sH})^C$, where $B_{7,\sH}$ is the Killing form of the
    Lie algebra $(\mathfrak{e}_{7,\sH})^C$.
\end{lemma}

\begin{proof}
   Note that $ (\varPhi,P \times Q)_7=\{\varPhi P,Q\} $ holds for $ \varPhi \in (\mathfrak{e}_{7,\sH})^C, P,Q \in (\mathfrak{P}_{\sH})^C $\!.  For $R = (\varPhi, P, Q, r, s, t) \allowbreak \in (\mathfrak{e}_{8,\sH})^C$, by
    doing simple computation of $(R\, \times\, R)R_1=0$ for all $
    R_1=(\varPhi_1, P_1, Q_1, \allowbreak r_1, s_1, t_1) \in
    (\mathfrak{e}_{8,\sH})^C$, we have the required relational formulas
    above.
\end{proof}

\begin{proposition}\label{proposition 8.7}
    The group $((E_{8,\sH})^C)_0$ acts on $(\mathfrak{W}_{\sH})^C$
    transitively.
\end{proposition}
\begin{proof}
    Since $\alpha \in (E_{8,\sH})^C$ leaves the Killing form $B_{8,\sH}$
    invariant $: B_{8,\sH}(\alpha R,$ $\alpha R') = B_8(R, R'), \allowbreak R, R' \in
    (\mathfrak{e}_{8,\sH})^C$, the group $(E_{8,\sH})^C$ acts on
    $(\mathfrak{W}_{\sH})^C$.
    Indeed, let $ R \in (\mathfrak{W}_{\sH})^C $. Then, for all $  R_1 \in
    (\mathfrak{e}_{8,\sH})^C $, it follows that
    \begin{align*}
    (\alpha R \times \alpha R)R_1 &= [\alpha R, [\alpha R, \alpha R_1]\,]
    +(1/18) B_{8,\sR}(\alpha R, R_1)\alpha R
    \\[0mm]
    &= \alpha[\,[R, [R, \alpha^{-1}R_1]\,] +
    (1/18)B_{8,\sR}(R, \alpha^{-1}R_1)\alpha R
    \\[0mm]
    &= \alpha((R \times R)\alpha^{-1}R_1
    \\[0mm]
    &= 0.
    \end{align*}
    \vspace{-5mm}

    \noindent Hence the group $(E_{8,\sH})^C$ acts on
    $(\mathfrak{W}_{\sH})^C$, that is, $ \alpha R \in (E_{8,\sH})^C $. We will
    show that this action is transitive. First, for all $  R_1 \in
    (\mathfrak{e}_{8,\sH})^C $, it follows from
    \begin{align*}
    (1_- \times 1_-)R_1 &= [1_-,[1_-, (\varPhi_1, P_1, Q_1, r_1, s_1,
    t_1)] \, ] +
    (1/18)B_8(1_-, R_1)1_-
    \\[0mm]
    &= [1_-, (0, 0, P_1, -s_1, 0, 2r_1)] + 2s_11_-
    \\[0mm]
    &= (0, 0, 0, 0, -2s_1) + 2s_11_-
    \\[0mm]
    &= 0
    \end{align*}
    that $1_- \in (\mathfrak{W}_{\sH})^C$. Then, any element $R \in
    (\mathfrak{W}_{\sH})^C$ can be transformed to $1_- \in (\mathfrak{W}_{\sH})^C$
    by some $\alpha \in ((E_{8,\sH})^C)_0$. We will prove this below.
    \vspace{1mm}

    Case (i) where $R = (\varPhi, P, Q, r, s,t), t \not= 0$.

    From Lemma \ref{lemma 8.6} (2),(5) and (6), we see
    $$
    \varPhi = -\dfrac{1}{2t}Q \times Q, \; P = \dfrac{r}{t}Q -
    \dfrac{1}{6t^2}(Q \times Q)Q, \; s = -\dfrac{r^2}{t} +
    \dfrac{1}{96t^3}\{Q, (Q \times Q)Q\}.
    $$
    Now, let $\varTheta:= \ad(0, P_1, 0, r_1, s_1, 0) \in
    \ad((\mathfrak{e}_{8,\sH})^C), r_1\not=0 $, then we compute $\varTheta^n1_-$:
    \begin{align*}
    &\quad \varTheta^n1_-
    \\
    &= \begin{pmatrix} ((-2)^{n-1} + (-1)^n){r_1}^{n-2}P_1 \times P_1
    \vspace{0.5mm}\\
    \Big((-2)^{n-1} - \dfrac{1 + (-1)^{n-1}}{2}\Big){r_1}^{n-2}s_1P_1 +
    \Big(\dfrac{1 - (-2)^n}{6} + \dfrac{(-1)^n}{2}\Big){r_1}^{n-3}(P_1
    \times P_1)P_1
    \vspace{0.5mm}\\
    ((-2)^n + (-1)^{n-1}){r_1}^{n-1}P_1
    \vspace{0.5mm}\\
    (-2)^{n-1}{r_1}^{n-1}s_1
    \vspace{0.5mm}\\
    -((-2)^{n-2} + 2^{n-2}){r_1}^{n-2}{s_1}^2 + \dfrac{2^{n-2} +
    (-2)^{n-2} - (-1)^{n-1}-1}{24}{r_1}^{n-4}\{P_1,(P_1 \times P_1)P_1\}
    \vspace{0.5mm}\\
    (-2)^n{r_1}^n \end{pmatrix}.
    \end{align*}
    Hence, by doing straightforward computation, we have
    \begin{align*}
    &\quad (\exp\varTheta)1_- = \Big(\dsum_{n=
    0}^\infty\dfrac{1}{n!}\varTheta^n \Big)1_-
    \vspace{0.5mm}\\
    &= \begin{pmatrix}   -\dfrac{1}{2{r_1}^2}(e^{-2r_1} -2e^{-r_1} +
    1)P_1 \times P_1
    \vspace{0.5mm}\\
    \dfrac{s_1}{2{r_1}^2}(-e^{-2r_1} - e^{r_1} + e^{-r_1} + 1)P_1 +
    \dfrac{1}{6{r_1}^3}(-e^{-2r_1} + e^{r_1} + 3e^{-r_1} - 3)(P_1 \times
    P_1)P_1
    \vspace{0.5mm}\\
    \dfrac{1}{r_1}(e^{-2r_1} - e^{-r_1})P_1
    \vspace{0.5mm}\\
    \dfrac{s_1}{2r_1}(1 - e^{-2r_1})
    \vspace{0.5mm}\\
    -\dfrac{{s_1}^2}{4{r_1}^2}(e^{-2r_1} + e^{2r_1} -2) +
    \dfrac{1}{96{r_1}^4}(e^{2r_1} + e^{-2r_1} - 4e^{r_1} - 4e^{-r_1} +
    6)\{P_1, (P_1 \times P_1)P_1\}
    \vspace{0.5mm}\\
    e^{-2r_1} \end{pmatrix}.
    \end{align*}
In the case where $ t\not=1 $. For a given $ R=(\varPhi,P,Q,r,s,t),t\not=0 $,
we can choose $ P_1 \in (\mathfrak{P}_{\sH})^C,r_1,s_1 \in C $ such that
\begin{align*}
\dfrac{1}{r_1}(e^{-2r_1} - e^{-r_1})P_1=Q, \;\; \dfrac{s_1}{2r_1}(1 -
e^{-2r_1})=r, \;\; e^{-2r_1}=t.
\end{align*}
 Indeed,
we choose some one value of $ \sqrt{t} $ satisfying $ (\sqrt{t})^2=t $ and
some one value of $ \log t $ for $ t \in C $, respectively. Then because
of $ t-t^2\not=0, t-1\not=0 $, we can get
\begin{align*}
P_1=\dfrac{(\sqrt{t}+t)\log t}{2(t-t^2)}Q,\;\; r_1=-\dfrac{\log
t}{2},\;\;  s_1=\dfrac{\log t}{t-1}r.
\end{align*}
    Hence, by using these $ P_1,r_1,s_1 $, we obtain
    \begin{align*}
    (\exp\varTheta)1_- = \begin{pmatrix} -\dfrac{1}{2t}Q \times Q
    \vspace{1mm}\\
    \dfrac{r}{t}Q - \dfrac{1}{6t^2}(Q \times Q)Q
    \vspace{0mm}\\
    Q
    \vspace{0mm}\\
    r
    \vspace{0mm}\\
    -\dfrac{r^2}{t} + \dfrac{1}{96 t^3}\{Q, (Q \times Q)Q\}
    \vspace{0mm}\\
    t
    \end{pmatrix}
    =R.
    \end{align*}
    In the case where $ t=1 $. Then, a given $
    R=(\varPhi,P,Q,r,s,t),t\not=0 $ is of the form
    \begin{align*}
    \varPhi=-\dfrac{1}{2}Q \times Q,\;\;P=rQ-\dfrac{1}{6}(Q\times
    Q)Q,\;\;s=-r^2+\dfrac{1}{96}\{Q,(Q \times Q)Q\},
    \end{align*}
    so let $\varTheta:= \ad(0, P_1, 0, 0, s_1, 0) \in
    \ad((\mathfrak{e}_{8,\sH})^C) $, we have
    \begin{align*}
    (\exp\varTheta)1_-
    = \Big(\dsum_{n= 0}^\infty\dfrac{1}{n!}\varTheta^n \Big)1_-
    = \begin{pmatrix}
    -\dfrac{1}{2}P_1 \times P_1
    \vspace{1mm}\\
    -s_1P_1+ \dfrac{1}{6}(P_1 \times P_1)P_1
    \vspace{0mm}\\
    -P_1
    \vspace{0mm}\\
    s_1
    \vspace{0mm}\\
    -{s_1}^2+\dfrac{1}{96}\{P_1, (P_1 \times P_1)P_1\}
    \vspace{0mm}\\
    1
    \end{pmatrix}.
    \end{align*}
    Here, as in the case $ t\not=1 $, we choose $ P_1 \in
    (\mathfrak{J}_{\sH})^C, s_1 \in C $ such that $ -P_1=Q, s_1=r $.
    Hence, by using these $ P_1,s_1 $, we obtain
    \begin{align*}
    (\exp\varTheta)1_- = \begin{pmatrix} -\dfrac{1}{2}Q \times Q
    \vspace{1mm}\\
    rQ-\dfrac{1}{6}(Q \times Q)Q
    \vspace{0mm}\\
    Q
    \vspace{0mm}\\
    r
    \vspace{0mm}\\
    -r^2 + \dfrac{1}{96}\{Q, (Q \times Q)Q\}
    \vspace{0mm}\\
    1
    \end{pmatrix}
    =R.
    \end{align*}
    Thus $R$ is transformed to $1_-$ by $(\exp \varTheta)^{-1} \in
    ((E_{8,\sH})^C)_0$.
\vspace{2mm}

   Case (ii) where  $R = (\varPhi, P, Q, r, s, 0), s \not= 0$.

    Let $\varTheta:=\ad(0, 0, 0, 0, {\pi}/{2}, -{\pi}/{2})) \in
    \ad((\mathfrak{e}_{8,\sH})^C)$. Then we have
    $$
    (\exp\varTheta) R =(\varPhi, Q,-P, -r, 0, -s), \;\;\; -s \not= 0.
    $$
    Hence this case is reduced to Case (i).
 \vspace{2mm}

    Case (iii) where  $R = (\varPhi, P, Q, r, 0, 0), r \not= 0$.

    Again, from Lemma \ref{lemma 8.6} (2),(5) and (6), we have
    $$
    Q \times Q = 0, \;\; \varPhi Q = -3rQ, \;\; \{P, Q\} = 16r^2.
    $$
    Then, let $\varTheta:=\ad(0, Q, 0, 0, 0, 0) \in \ad((\mathfrak{e}_{8,\sH})^C)$, we have
    $$
    (\exp\varTheta)R = (\varPhi, P + 2rQ, Q, r, -4r^2, 0),
    \;\; -4r^2 \not= 0.
    $$
    Hence this case is reduced to Case (ii).
 \vspace{2mm}

   Case (iv) where $R = (\varPhi, P, Q, 0, 0, 0), Q \not= 0$.

    We can choose $P_1 \in (\mathfrak{P}_{\sH})^C$ such that $\{P_1, Q\}
    \not= 0$. Indeed, $ \{P_1,Q\}=0 $ for all $ P_1 \in (\mathfrak{P}_{\sH})^C $ implies $ Q=0 $, so that
    there exists $ P_1 \in (\mathfrak{P}_{\sH})^C $ such that $\{P_1, Q\}
    \not= 0$.

    Now, let $\varTheta:= \ad(0, P_1, 0, 0, 0, 0) \in
    \ad((\mathfrak{e}_{8,\sH})^C)$, we have
    \begin{align*}
     (\exp \varTheta)R
    = \left( \begin{array}{c}
    \varPhi+P_1 \times Q
    \vspace{1mm}\\
    \,P-\varPhi P_1+\dfrac{1}{2}(P_1 \times Q)P_1+\dfrac{1}{16}\{P_1,Q\}P_1
    \vspace{1mm}\\
    Q
    \vspace{1mm}\\
    -\dfrac{1}{8}\{P_1, Q\}
    \vspace{1mm}\\
    \dfrac{1}{4}\{P_1, P\}+\dfrac{1}{8}\{P_1, -\varPhi
    P_1\}+\dfrac{1}{24}\{P_1,( P_1 \times Q)P_1\}
    \vspace{1mm}\\
    0
     \end{array} \right),\;\;  -\dfrac{1}{8}\{P_1, Q\} \not =0.
    \end{align*}
    Hence this case is reduced to Case (iii).
\vspace{2mm}

   Case (v) where $R = (\varPhi, P, 0, 0, 0, 0), P \not= 0$.

    As in Case (iv), we choose $Q_1 \in (\mathfrak{P}_{\sH})^C$ such that
    $\{P, Q_1\} \not= 0$. Then, let $\varTheta:= \ad(0, 0, Q_1, 0, 0,\allowbreak  0)
    \in \ad((\mathfrak{e}_{8,\sH})^C)$, we have
    \begin{align*}
    (\exp \varTheta)R
    =\left( \begin{array}{c}
    \varPhi-P \times Q_1
    \vspace{1mm}\\
     P
    \vspace{1mm}\\
     -\varPhi Q_1+\dfrac{1}{2}(P \times Q_1)Q_1+\dfrac{1}{16}\{P,Q_1\}Q_1
    \vspace{1mm}\\
    \dfrac{1}{8}\{P, Q_1\}
    \vspace{1mm}\\
    0
     \vspace{1mm}\\
     -\dfrac{1}{8}\{Q_1, -\varPhi Q_1\} -\dfrac{1}{24}\{Q_1, -(P \times
     Q_1)Q_1  \}
    \end{array}\right),\;\; \dfrac{1}{8}\{P, Q_1\} \not=0.
    \end{align*}
    Hence this case is also reduced to Case (iii).
\vspace{2mm}

    Case (vi) where $R = (\varPhi, 0, 0, 0, 0, 0), \varPhi \not= 0.$
    We can choose $P_1 \in (\mathfrak{P}_{\sH})^C$ such that $\varPhi P_1 \not
    = 0$. Indeed, $\varPhi P_1=0 $ for all $ P_1 \in (\mathfrak{P}_{\sH})^C $ implies $ \varPhi=0 $, so that there exists $ P_1 \in (\mathfrak{P}_{\sH})^C $ such that $\varPhi P_1 \not= 0$.
    Then, let $ \varTheta:=\ad(0,P_1,0,0,0,0) \in \ad((\mathfrak{e}_{8,\sH})^C)$, we have
    $$
    (\exp\ad(0, P_1, 0, 0, 0, 0))R = \Big(\varPhi, -\varPhi P_1, 0, 0,
    \dfrac{1}{8}\{\varPhi P_1, P_1\}, 0 \Big),\;\; -\varPhi P_1\not=0.
    $$
    Hence this case is reduced to Case (v).

    With above, the proof of this proposition is completed.
\end{proof}

Now, we will prove the theorem as the immediate aim .

\begin{theorem}\label{theorem 8.8}
    The homogeneous space $(E_{8,\sH})^C/((E_{8,\sH})^C)_{1_-}$ is
    homeomorphic to the space $(\mathfrak{W}_{\sH})^C${\rm : }
    $(E_{8,\sH})^C/((E_{8,\sH})^C)_{1_-} \simeq (\mathfrak{W}_{\sH})^C$.

    In particular, the group $(E_{8,\sH})^C$ is connected.
\end{theorem}
\begin{proof}
    Since the group $(E_{8,\sH})^C$ acts on the space $(\mathfrak{W}_{\sH})^C$ transitively (Proposition \ref{proposition 8.7}), the former half of this theorem is proved.
The latter half was shown as follows. The connectedness of the group $(E_{8,\sH})^C$ follows from the connectedness of the group $((E_{8,\sH})^C)_{1_-}$ and the space $(\mathfrak{W}_{\sH})^C=((E_{8,\sH})^C)_0 1_-$ (Propositions \ref{proposition 8.5}, \ref{proposition 8.7}).
\end{proof}

In order to give a polar decomposition of the group $ (E_{8,\sH})^C $, we prove the following lemma.

\begin{lemma}\label{lemma 8.9}
 The group $  (E_{8,\sH})^C $ is an algebraic subgroup of the general linear group $ GL(133,C) \allowbreak =\Iso_{C}((\mathfrak{c}_{8,\sH})^C) $ and satisfies the condition that $ \alpha \in (E_{8,\sH})^C $ implies $ \alpha^*=(\tau\lambda_\omega)\alpha^{-1}(\lambda_\omega\tau) \in (E_{8,\sH})^C $, where $ \alpha^* $ in the transpose of $ \alpha $ with respect to the Hermite inner product $ \langle R_1,R_2 \rangle ${\rm :} $ \langle\alpha^* R_1, R_2 \rangle=\langle R_1,\alpha R_2 \rangle $.
\end{lemma}
\begin{proof}
First, it is easy to verify that $ \alpha^*=(\tau\lambda_\omega)\alpha^{-1}(\lambda_\omega\tau) $. Indeed, it follows from
\begin{align*}
\langle\alpha^* R_1, R_2 \rangle&=\langle R_1,\alpha R_2 \rangle=-\dfrac{1}{15}B_{8,\sH}((\tau\lambda_\omega)R_1,\alpha R_2)
\\
&=-\dfrac{1}{15}B_{8,\sH}(\alpha(\alpha^{-1}(\tau\lambda_\omega)R_1),\alpha R_2)
\\
&=-\dfrac{1}{15}B_{8,\sH}(\alpha^{-1}(\tau\lambda_\omega)R_1, R_2)
\\
&=-\dfrac{1}{15}B_{8,\sH}((\tau\lambda_\omega)((\lambda_\omega\tau)\alpha^{-1}(\tau\lambda_\omega)R_1), R_2)
\\
&=\langle (\lambda_\omega\tau)\alpha^{-1}(\tau\lambda_\omega)R_1, R_2\rangle
\end{align*}
that $ \alpha^*=(\lambda_\omega\tau)\alpha^{-1}(\tau\lambda_\omega) $, that is, $ \alpha^*=(\tau\lambda_\omega)\alpha^{-1}(\lambda_\omega\tau) $.

Subsequently, we will show $ \alpha^* \in (E_{8,\sH})^C $. It is clear that $ \alpha^* \in \Iso_{C}((\mathfrak{e}_{8,\sH})^C) $ and note that $ [\tau R_1,\tau R_2]=\tau[R_1,R_2],R_i \in (\mathfrak{e}_{8,\sH})^C $ and $ \lambda_\omega \in (E_{8,\sH})^C $, it follows that
\begin{align*}
[\alpha^* R_1, \alpha^* R_2]&=[(\tau\lambda_\omega)\alpha^{-1}(\lambda_\omega\tau)R_1, (\tau\lambda_\omega)\alpha^{-1}(\lambda_\omega\tau)R_2]
\\
&=\tau[\lambda_\omega\alpha^{-1}(\lambda_\omega\tau)R_1, \lambda_\omega\alpha^{-1}(\lambda_\omega\tau)R_2]
\\
&=(\tau\lambda_\omega)[\alpha^{-1}(\lambda_\omega\tau)R_1, \alpha^{-1}(\lambda_\omega\tau)R_2]
\\
&=(\tau\lambda_\omega)\alpha^{-1}[(\lambda_\omega\tau)R_1,(\lambda_\omega\tau)R_2]
\\
&=(\tau\lambda_\omega)\alpha^{-1}(\lambda_\omega\tau)[R_1,R_2]
\\
&=\alpha^*[R_1,R_2].
\end{align*}
Hence we have $ \alpha^* \in (E_{8,\sH})^C $.

Finally, the group $ (E_{8,\sH})^C $ is defined by the algebraic relation $ \alpha[R_1,R_2]=[\alpha R_1. \alpha R_2] $, so that $ (E_{8,\sH})^C $
is algebraic.
\end{proof}

Let $ O((\mathfrak{e}_{8,\sH})^C) $ be the orthogonal subgroup $ GL(133,C)=\Iso_{C}((\mathfrak{e}_{8,\sH})^C) $:
\begin{align*}
O(133,C)=O((\mathfrak{e}_{8,\sH})^C)=\left\lbrace \alpha \in \Iso_{C}((\mathfrak{e}_{8,\sH})^C) \relmiddle{|}\langle \alpha R_1,\alpha R_2 \rangle=\langle R_1,R_2\rangle \right\rbrace.
\end{align*}
Then we have the following
\begin{align*}
((E_{8,\sH})^C  \cap O((\mathfrak{e}_{8,\sH})^C)&=\left\lbrace \alpha \in \Iso_{C}((\mathfrak{e}_{8,\sH})^C) \relmiddle{|}
\begin{array}{l}
\alpha[R_1,R_2]=[\alpha R_1. \alpha R_2],
\\
\langle \alpha R_1,\alpha R_2 \rangle=\langle R_1,R_2\rangle
\end{array}
\right\rbrace
\\
&=E_{8,\sH}.
\end{align*}

Using Chevally's lemma(\cite[Lemma 2]{che}), from Lemma \ref{lemma 8.9} we have a homeomorphism
\begin{align*}
(E_{8,\sH})^C & \simeq ((E_{8,\sH})^C  \cap O((\mathfrak{e}_{8,\sH})^C)) \times \R^d
\\
&=E_{8,\sH} \times \R^d,
\end{align*}
where the dimension $ d $ of the Euclidian part is computed from Lemma \ref{lemma 8.2} as follows:
\begin{align*}
d=\dim((E_{8,\sH})^C) - \dim(E_{8,\sH})
=133 \times 2-133
=133.
\end{align*}

With above, we have the following theorem.

\begin{theorem}\label{theorem 8.10}
The group $ (E_{8,\sH})^C $ is homeomorphic to the topological product of the group $ E_{8,\sH} $ and a $ 133 $-dimensional Euclidian space $ \R^{133} ${\rm :}
\begin{align*}
(E_{8,\sH})^C \simeq E_{8,\sH} \times \R^{133}.
\end{align*}

In particular, the group $ E_{8,\sH} $ is connected.
\end{theorem}
\begin{proof}
The former half has been already proved above. As for the latter half, the connectedness of the group $ E_{8,\sH} $ follows from the connectedness of the group $ (E_{8,\sH})^C $ (Theorem \ref{theorem 8.8})
\end{proof}

Next, we move on to study a group $ ({E_8}^C)^{\varepsilon_1,\varepsilon_2} $ defined below.

We define a subgroup $ ({E_8}^C)^{\varepsilon_1,\varepsilon_2} $ of $ E_8 $ by
\begin{align*}
({E_8}^C)^{\varepsilon_1,\varepsilon_2}:=\left\lbrace \alpha \in E_8 \relmiddle{|} \varepsilon_1\alpha=\alpha\varepsilon_1,\varepsilon_2\alpha=\alpha\varepsilon_2 \right\rbrace.
\end{align*}

The immediate aim is to prove the connectedness of the group $ ({E_8}^C)^{\varepsilon_1,\varepsilon_2} $ by the same procedure as the proof of the connectedness of the group $ (E_{8,\sH})^C $.

First, we consider a subgroup $
(({E_8}^C)^{\varepsilon_1,\varepsilon_2})_{\tilde{1},1^-,1_-} $ of $ ({E_8}^C)^{\varepsilon_1,\varepsilon_2} $:
\begin{align*}
(({E_8}^C)^{\varepsilon_1,\varepsilon_2})_{\tilde{1},1^-,1_-}=\left\lbrace \alpha \in ({E_8}^C)^{\varepsilon_1,\varepsilon_2}
\relmiddle{|} \alpha \tilde{1}=\tilde{1},\alpha 1^-=1^-, \alpha
1_-=1_-\right\rbrace.
\end{align*}

Then we have the following proposition.

\begin{proposition}\label{proposition 8.11}
    The group $ (({E_8}^C)^{\varepsilon_1,\varepsilon_2})_{\tilde{1},1^-,1_-} $ is isomorphic to the
    group $ ({E_7}^C)^{\varepsilon_1,\varepsilon_2} ${\rm :} \\ $ (({E_8}^C)^{\varepsilon_1,\varepsilon_2})_{\tilde{1},1^-,1_-}
    \cong ({E_7}^C)^{\varepsilon_1,\varepsilon_2} $.
\end{proposition}
\begin{proof}
Since $ \varepsilon_i,i=1,2 $ satisfies $ \varepsilon_i\tilde{1}=\tilde{1}, \varepsilon_i1^-=1^-, \varepsilon_i1_-=1_-  $, $ \varepsilon_i $ induces an automorphism of the group $ ({E_8}^C)_{\tilde{1},1^-,1_-} $.
Hence, note that $ ({E_8}^C)_{\tilde{1},1^-,1_-} \cong {E_7}^C $ (\cite[Proposition 5.7.1]{iy0}), we have
\begin{align*}
 (({E_8}^C)^{\varepsilon_1,\varepsilon_2})_{\tilde{1},1^-,1_-}=(({E_8}^C)_{\tilde{1},1^-,1_-})^{\varepsilon_1,\varepsilon_2} \cong ({E_7}^C)^{\varepsilon_1,\varepsilon_2}.
\end{align*}
\end{proof}

Subsequently, we consider a subgroup $ (({E_8}^C)^{\varepsilon_1,\varepsilon_2})_{1_-}$ of $ ({E_8}^C)^{\varepsilon_1,\varepsilon_2} $:
\begin{align*}
(({E_8}^C)^{\varepsilon_1,\varepsilon_2})_{1_-}=\left\lbrace  \alpha \in ({E_8}^C)^{\varepsilon_1,\varepsilon_2} \relmiddle{|}\alpha 1_-=1_- \right\rbrace.
\end{align*}

Then we prove the following lemma.

\begin{lemma}\label{lemma 8.12}
    {\rm (1)} The Lie algebra $ ({\mathfrak{e}_8}^C)^{\varepsilon_1,\varepsilon_2} $ of the group $({E_8}^C)^{\varepsilon_1,\varepsilon_2} $ is given by
    \begin{align*}
    ({\mathfrak{e}_8}^C)^{\varepsilon_1,\varepsilon_2}=
    \left\lbrace R=(\varPhi,P,Q,r,s,t) \relmiddle{|}
    \begin{array}{l}
    \varPhi \in ({\mathfrak{e}_7}^C)^{\varepsilon_1,\varepsilon_2}, P,Q \in (\mathfrak{P}_{\sH})^C, r,s,t \in C
    \end{array}
    \right\rbrace.
    \end{align*}

    In particular, we have $ \dim_C(({\mathfrak{e}_8}^C)^{\varepsilon_1,\varepsilon_2})=66+32\times 2+3=133 $.
    \vspace{2mm}

    {\rm (2)}
    The Lie algebra $ (({\mathfrak{e}_8}^C)^{\varepsilon_1,\varepsilon_2})_{1_-} $ of the group $ (({E_8}^C)^{\varepsilon_1,\varepsilon_2})_{1_-} $ is given by
    \begin{align*}
    (({\mathfrak{e}_8}^C)^{\varepsilon_1,\varepsilon_2})_{1_-}
    &=\left\lbrace R \in ({\mathfrak{e}_8}^C)^{\varepsilon_1,\varepsilon_2} \relmiddle{|} [R, 1_-]=0
    \right\rbrace
    \\
    &=\left\lbrace (\varPhi, 0, Q, 0, 0, t) \in({\mathfrak{e}_8}^C)^{\varepsilon_1,\varepsilon_2}
    \relmiddle{|}
    \begin{array}{l}
    \varPhi \in ({\mathfrak{e}_7}^C)^{\varepsilon_1,\varepsilon_2},
    Q \in (\mathfrak{P}_{\sH})^C,
    t \in C
    \end{array} \right\rbrace .
    \end{align*}

    In particular, we have $ \dim_C((({\mathfrak{e}_8}^C)^{\varepsilon_1,\varepsilon_2})_{1_-}) = 66+32+1 = 99 $.
\end{lemma}
\begin{proof}

     By the straightforward computation, we can easily obtain the required results.
\end{proof}

We prove the following proposition needed in the proof of connectedness.

\begin{proposition}\label{proposition 8.13}
    The group $(({E_8}^C)^{\varepsilon_1,\varepsilon_2})_{1_-}$ is a semi-direct product of groups \\
    $\exp(\ad(((\mathfrak{P}_{\sH})^C)_- \allowbreak \oplus C_- ))$ and $
    ({E_7}^C)^{\varepsilon_1,\varepsilon_2} ${\rm:}
    \begin{align*}
    (({E_8}^C)^{\varepsilon_1,\varepsilon_2})_{1_-}=\exp(\ad(((\mathfrak{P}_{\sH})^C)_- \oplus C_-
    )) \rtimes ({E_7}^C)^{\varepsilon_1,\varepsilon_2}.
    \end{align*}

    In particular, the group $(({E_8}^C)^{\varepsilon_1,\varepsilon_2})_{1_-}$ is connected.
\end{proposition}
\begin{proof}
Using Proposition \ref{proposition 8.11} and Lemma \ref{lemma 8.12}, as in the proof of Proposition \ref{proposition 8.5}, this proposition is proved by replacing $ (E_{8,\sH})^C, (E_{7,\sH})^C, (\mathfrak{e}_{8,\sH})^C $ of the proof of Proposition \ref{proposition 8.5} with $ ({E_8}^C)^{\varepsilon_1,\varepsilon_2}, \allowbreak ({E_7}^C)^{\varepsilon_1,\varepsilon_2}, ({\mathfrak{e}_8}^C)^{\varepsilon_1,\varepsilon_2} $, respectively. Note that the connectedness of the group $ ({E_7}^C)^{\varepsilon_1,\varepsilon_2} $ follows from $ ({E_7}^C)^{\varepsilon_1,\varepsilon_2} \cong Spin(12,C) $ (\cite[Proposition 1.1.7]{miya1}).

\end{proof}

Before proving the connectedness of the group $ ({E_8}^C)^{\varepsilon_1,\varepsilon_2} $, we will make some preparations.

For $R \in {\mathfrak{e}_8}^C $, we define a $C$-linear mapping $R
\times R : {\mathfrak{e}_8}^C \to {\mathfrak{e}_8}^C $ by
\begin{align*}
(R \times R)R_1 = [R, [R, R_1]\,] + \dfrac{1}{30}B_8(R, R_1)R,\,\,
R_1 \in {\mathfrak{e}_8}^C,
\end{align*}
where $B_8$ is the Killing form of $ {\mathfrak{e}_8}^C $ (\cite[Theorem 5.3.2]{iy0}).
Using this mapping, we define a space $ (\mathfrak{W}^C)_{\varepsilon_1,\varepsilon_2}$ by
\begin{align*}
 (\mathfrak{W}^C)_{\varepsilon_1,\varepsilon_2}= \left\lbrace R \in {\mathfrak{e}_8}^C \relmiddle{|} R \times R =0, R \not= 0, \varepsilon_1 R=R, \varepsilon_2 R=R \right\rbrace,
\end{align*}
where the action of $ \varepsilon_i,i=1,2 $ to $ {\mathfrak{e}_8}^C $ is given by
\begin{align*}
\varepsilon_i(\varPhi,P,Q,r,s,t)=((\ad\varepsilon_i)\varPhi,\varepsilon_iP,\varepsilon_iQ,r,s,t),\;\;(\varPhi,P,Q,r,s,t) \in {\mathfrak{e}_8}^C,
\end{align*}
and note that  $ P,Q $-part of $ R=(\varPhi,P,Q,r,s,t) \in (\mathfrak{W}^C)_{\varepsilon_1,\varepsilon_2} $ belong to $ (\mathfrak{P}_{\sH})^C$: $ P,Q \in (\mathfrak{P}_{\sH})^C $.


\begin{lemma}\label{lemma 8.14}
    For $R = (\varPhi, P, Q, r, s, t) \in {\mathfrak{e}_8}^C $,
    $R \not=0$ belongs to $(\mathfrak{W}^C)_{\varepsilon_1,\varepsilon_2} $ if and only if $R$
    satisfies the following conditions
    \vspace{1mm}

    {\rm (1)} $2s\varPhi - P \times P = 0$ \quad {\rm (2)} $2t\varPhi + Q
    \times Q = 0$ \quad {\rm (3)} $2r\varPhi + P \times Q = 0$
    \vspace{1mm}

    {\rm (4)} $\varPhi P - 3rP - 3sQ = 0 $
    {\rm (5)} $\varPhi Q + 3rQ - 3tP = 0 $ \quad {\rm (6)} $\{P, Q\} -
    16(st + r^2) = 0$
    \vspace{1mm}

    {\rm (7)} $2(\varPhi P \times Q_1 + 2P \times \varPhi Q_1 - rP \times
    Q_1 - sQ \times Q_1) - \{P, Q_1\}\varPhi = 0$
    \vspace{1mm}

    {\rm (8)} $2(\varPhi Q \times P_1 + 2Q \times \varPhi P_1 + rQ \times
    P_1 - tP \times P_1) \!- \{Q, P_1\}\varPhi = 0$
    \vspace{1mm}

    {\rm (9)} $8((P \times Q_1)Q - stQ_1 - r^2Q_1 - \varPhi^2Q_1 +
    2r\varPhi Q_1) + 5\{P, Q_1\}Q
    -2\{Q, Q_1\}P = 0$
    \vspace{1mm}

    \hspace*{-1.7mm}{\rm (10)} $8((Q \times P_1)P + stP_1 + r^2P_1 +
    \varPhi^2P_1 + 2r\varPhi P_1) \,+ \,5\{Q, P_1\}P
    -2\{P, P_1\}Q= 0$
    \vspace{1mm}

    \hspace*{-1.7mm}{\rm (11)} $18(\ad\,\varPhi)^2\varPhi_1 + Q \times
    \varPhi_1P - P \times \varPhi_1Q) +   B_7(\varPhi,
    \varPhi_1)\varPhi = 0$
    \vspace{1mm}

    \hspace*{-1.7mm}{\rm (12)} $18(\varPhi_1\varPhi P -2\varPhi\varPhi_1P
    - r\varPhi_1P - s\varPhi_1Q) + B_7(\varPhi, \varPhi_1)P = 0$
    \vspace{1mm}

    \hspace*{-1.7mm}{\rm (13)} $18(\varPhi_1\varPhi Q -2\varPhi\varPhi_1Q
    + r\varPhi_1Q - t\varPhi_1P) + B_7(\varPhi, \varPhi_1)Q = 0,$

    \vspace{1mm}
    \noindent for all $\varPhi_1 \in ({\mathfrak{e}_7}^C)^{\varepsilon_1,\varepsilon_2}  P_1, Q_1 \in
    (\mathfrak{P}_{\sH})^C$, where $B_7$ is the Killing form of the
    Lie algebra ${\mathfrak{e}_7}^C $.
\end{lemma}
\begin{proof}
   Note that $ (\varPhi,P \times Q)_7=\{\varPhi P,Q\} $ holds for $ \varPhi \in ({\mathfrak{e}_7}^C)^{\varepsilon_1,\varepsilon_2}, P,Q \in (\mathfrak{P}_{\sH})^C $\!.  For $R = (\varPhi, P, \allowbreak Q, r, s, t) \allowbreak \in (\mathfrak{W}^C)_{\varepsilon_1,\varepsilon_2}$, by
    doing simple computation of $(R\, \times\, R)R_1=0$ for all $
    R_1=(\varPhi_1, P_1, Q_1, r_1, s_1, \allowbreak t_1) \in
    (\mathfrak{W}^C)_{\varepsilon_1,\varepsilon_2}$, we have the required relational formulas
    above.
\end{proof}

\begin{proposition}\label{proposition 8.15}
    The group $(({E_8}^C)^{\varepsilon_1,\varepsilon_2})_0$ acts on $(\mathfrak{W}^C)_{\varepsilon_1,\varepsilon_2}$ transitively.
\end{proposition}
\begin{proof}
Using Lemmas \ref{lemma 8.12} (1), \ref{lemma 8.14}, as in the proof of Proposition \ref{proposition 8.7}, this proposition is proved by replacing $ (E_{8,\sH})^C, (\mathfrak{e}_{8,\sH})^C $ of the proof of Proposition \ref{proposition 8.7} with $ ({E_8}^C)^{\varepsilon_1,\varepsilon_2}, ({\mathfrak{e}_8}^C)^{\varepsilon_1,\varepsilon_2} $, respectively.

\end{proof}

Now, we will prove the theorem as the immediate aim .

\begin{theorem}\label{theorem 8.16}
    The homogeneous space $({E_8}^C)^{\varepsilon_1,\varepsilon_2}/(({E_8}^C)^{\varepsilon_1,\varepsilon_2})_{1_-}$ is
    homeomorphic to the space $(\mathfrak{W}^C)_{\varepsilon_1,\varepsilon_2}${\rm : }
    $({E_8}^C)^{\varepsilon_1,\varepsilon_2}/(({E_8}^C)^{\varepsilon_1,\varepsilon_2})_{1_-} \simeq (\mathfrak{W}^C)_{\varepsilon_1,\varepsilon_2}$.

    In particular, the group $({E_8}^C)^{\varepsilon_1,\varepsilon_2}$ is connected.
\end{theorem}
\begin{proof}
    Since the group $({E_8}^C)^{\varepsilon_1,\varepsilon_2}$ acts on the space $(\mathfrak{W}^C)_{\varepsilon_1,\varepsilon_2}$
    transitively (Proposition \ref{proposition 8.15}), the former half of
    this theorem is proved.

The latter half was shown as follows. The connectedness of the group $({E_8}^C)^{\varepsilon_1,\varepsilon_2}$ follows from the connectedness of the group $(({E_8}^C)^{\varepsilon_1,\varepsilon_2})_{1_-}$ and the space $(\mathfrak{W}^C)_{\varepsilon_1,\varepsilon_2}=(({E_8}^C)^{\varepsilon_1,\varepsilon_2})_0 1_-$ (Propositions \ref{proposition 8.13}, \ref{proposition 8.15}).
\end{proof}

In order to determine the type of the group $ ({E_8}^C)^{\varepsilon_1,\varepsilon_2} $ as Lie algebras, we move the determination of the root system and the Dynkin diagram of the Lie algebra $ ({\mathfrak{e}_8}^C)^{\varepsilon_1,\varepsilon_2} $.

Here, we define a Lie subalgebra $ \mathfrak{h}_8 $ of $ ({\mathfrak{e}_8}^C)^{\varepsilon_1,\varepsilon_2} $ by
\begin{align*}
\mathfrak{h}_8:=\left\lbrace R_8=(\varPhi,0,0,w,0,0) \relmiddle{|}
\begin{array}{l}
\varPhi:=\varPhi(\phi,0,0,\mu) \in \mathfrak{h}_7, \\
\quad \phi=\delta +(\mu_1E_1+\mu_2E_2+\mu_3E_3)^\sim
\\
\qquad \delta:=(L_1,L_2,L_3),
\\
\qquad\;\; L_1=\lambda_0(iG_{01})+\lambda_1(iG_{23})+\lambda_2(i(G_{45}+G_{67})),
\\
\qquad\;\;  L_2=\pi\kappa L_1, L_3=\kappa\pi L_1,\lambda_i \in C,
\\
\qquad \mu_k \in C, \mu_1+\mu_2+\mu_3=0,
\mu \in C,
\\
w \in C
\end{array}
\right\rbrace.
\end{align*}

Then $ \mathfrak{h}_8 $ is a Cartan subalgebra of $ ({\mathfrak{e}_8}
^C)^{\varepsilon_1,\varepsilon_2} $. Indeed, it is clear that $
({\mathfrak{e}_8}^C)^{\varepsilon_1,\varepsilon_2} $ is abelian. Next, by
doing straightforward computation, we can confirm that $ [R,
R_8]\in \mathfrak{h}_8, R \in ({\mathfrak{e}_8}^C)^{\varepsilon_1,\varepsilon_2} $ implies $ R \in \mathfrak{h}_8 $
for any $ R_8 \in \mathfrak{h}_8 $.

\begin{theorem}\label{theorem 8.17}
    The rank of the Lie algebra $ ({\mathfrak{e}_8}^C)^{\varepsilon_1,
    \varepsilon_2} $ is seven. The roots $ \varDelta $ of $ ({\mathfrak{e}_8}^C)^{\varepsilon_1,\varepsilon_2} $ relative to $ \mathfrak{h}_8 $ are given by
    \begin{align*}
        \varDelta=\left\lbrace
        \begin{array}{l}
            \pm(\lambda_0-\lambda_1), \;
            \pm(\lambda_0+\lambda_1),\; \pm 2\lambda_2, \;\pm(1/2)(-2\lambda_0+\mu_2-\mu_3),
            \vspace{0.5mm}\\
            \pm(1/2)(-2\lambda_0-\mu_2+\mu_3),\;
            \pm(1/2)(-2\lambda_1+\mu_2-\mu_3),\;
            \pm(1/2)(-2\lambda_1-\mu_2+\mu_3),
            \vspace{0.5mm}\\
            \pm(1/2)(\lambda_0-\lambda_1-2\lambda_2-\mu_1+\mu_3),\;
            \pm(1/2)(\lambda_0-\lambda_1-2\lambda_2+\mu_1-\mu_3),
            \vspace{0.5mm}\\
            \pm(1/2)(\lambda_0-\lambda_1+2\lambda_2-\mu_1+\mu_3),\;
            \pm(1/2)(\lambda_0-\lambda_1+2\lambda_2+\mu_1-\mu_3),
            \vspace{0.5mm}\\
            \pm(1/2)(\lambda_0+\lambda_1+2\lambda_2+\mu_1-\mu_2),\;
            \pm(1/2)(\lambda_0+\lambda_1+2\lambda_2-\mu_1+\mu_2),
            \vspace{0.5mm}\\
            \pm(1/2)(-\lambda_0-\lambda_1+2\lambda_2+\mu_1-\mu_2),\;
            \pm(1/2)(-\lambda_0-\lambda_1+2\lambda_2-\mu_1+\mu_2),
\vspace{0.5mm}\\
\pm(\mu_1+(2/3)\mu),\;\pm(\mu_2+(2/3)\mu),\;\pm(\mu_3+(2/3)\mu),\;\pm(-\lambda_0-(1/2)\mu_1+(2/3)\mu),
\vspace{0.5mm}\\
\pm(\lambda_0-(1/2)\mu_1+(2/3)\mu),\;
\pm(-\lambda_1-(1/2)\mu_1+(2/3)\mu),\;
\pm(\lambda_1-(1/2)\mu_1+(2/3)\mu),
\vspace{0.5mm}\\
\pm(-(1/2)(-\lambda_0+\lambda_1+2\lambda_2)-(1/2)\mu_2+(2/3)\mu),
\vspace{0.5mm}\\
\pm((1/2)(-\lambda_0+\lambda_1+2\lambda_2)-(1/2)\mu_2+(2/3)\mu),
\vspace{0.5mm}\\
\pm(-(1/2)(-\lambda_0+\lambda_1-2\lambda_2)-(1/2)\mu_2+(2/3)\mu),
\vspace{0.5mm}\\
\pm((1/2)(-\lambda_0+\lambda_1-2\lambda_2)-(1/2)\mu_2+(2/3)\mu),
\vspace{0.5mm}\\
\pm(-(1/2)(-\lambda_0-\lambda_1-2\lambda_2)-(1/2)\mu_3+(2/3)\mu),
\vspace{0.5mm}\\
\pm((1/2)(-\lambda_0-\lambda_1-2\lambda_2)-(1/2)\mu_3+(2/3)\mu),
\vspace{0.5mm}\\
\pm(-(1/2)(\lambda_0+\lambda_1-2\lambda_2)-(1/2)\mu_3+(2/3)\mu),
\vspace{0.5mm}\\
\pm((1/2)(\lambda_0+\lambda_1-2\lambda_2)-(1/2)\mu_3+(2/3)\mu),
\vspace{0.5mm}\\
\pm(\mu_1-(1/3)\mu+w),\;\pm(\mu_2-(1/3)\mu+w),\;\pm(\mu_3-(1/3)\mu+w),
\vspace{0.5mm}\\
\pm(-\lambda_0-(1/2)\mu_1-(1/3)\mu+w ),\;\pm(\lambda_0-(1/2)\mu_1-(1/3)\mu+w ),
\vspace{0.5mm}\\
\pm(-\lambda_1-(1/2)\mu_1-(1/3)\mu+w ),\;\pm(\lambda_1-(1/2)\mu_1-(1/3)\mu+w ),
\vspace{0.5mm}\\
\pm((-1/2)(-\lambda_0+\lambda_1+2\lambda_2)-(1/2)\mu_2-(1/3)\mu+w),
\vspace{0.5mm}\\
\pm((1/2)(-\lambda_0+\lambda_1+2\lambda_2)-(1/2)\mu_2-(1/3)\mu+w),
\vspace{0.5mm}\\
\pm((-1/2)(-\lambda_0+\lambda_1-2\lambda_2)-(1/2)\mu_2-(1/3)\mu+w),
\vspace{0.5mm}\\
\pm((1/2)(-\lambda_0+\lambda_1-2\lambda_2)-(1/2)\mu_2-(1/3)\mu+w),
\vspace{0.5mm}\\
\pm((-1/2)(-\lambda_0-\lambda_1-2\lambda_2)-(1/2)\mu_3-(1/3)\nu+w),
\vspace{0.5mm}\\
\pm((1/2)(-\lambda_0-\lambda_1-2\lambda_2)-(1/2)\mu_3-(1/3)\nu+w),
\vspace{0.5mm}\\
 \end{array}
        \right\rbrace.
\end{align*}

\begin{align*}
\varDelta=\left\lbrace
\begin{array}{l}
\pm((-1/2)(\lambda_0+\lambda_1-2\lambda_2)-(1/2)\mu_3-(1/3)\nu+w),
\vspace{0.5mm}\\
\pm((1/2)(\lambda_0+\lambda_1-2\lambda_2)-(1/2)\mu_3-(1/3)\nu+w),
\vspace{0.5mm}\\
\pm(-\mu_1+(1/3)\mu+w),\;\pm(-\mu_2+(1/3)\mu+w),\;\pm(-\mu_3+(1/3)\mu+w),
\vspace{0.5mm}\\
\pm(-\lambda_0+(1/2)\mu_1+(1/3)\mu+w ),\;\pm(\lambda_0+(1/2)\mu_1+(1/3)\mu+w ),
\vspace{0.5mm}\\
\pm(-\lambda_1+(1/2)\mu_1+(1/3)\mu+w ),\;\pm(\lambda_1+(1/2)\mu_1+(1/3)\mu+w ),
\vspace{0.5mm}\\
\pm((-1/2)(-\lambda_0+\lambda_1+2\lambda_2)+(1/2)\mu_2+(1/3)\mu+w),
\vspace{0.5mm}\\
\pm((1/2)(-\lambda_0+\lambda_1+2\lambda_2)+(1/2)\mu_2+(1/3)\mu+w),
\vspace{0.5mm}\\
\pm((-1/2)(-\lambda_0+\lambda_1-2\lambda_2)+(1/2)\mu_2+(1/3)\mu+w),
\vspace{0.5mm}\\
\pm((1/2)(-\lambda_0+\lambda_1-2\lambda_2)+(1/2)\mu_2+(1/3)\mu+w),
\vspace{0.5mm}\\
\pm((-1/2)(-\lambda_0-\lambda_1-2\lambda_2)+(1/2)\mu_3+(1/3)\nu+w),
\vspace{0.5mm}\\
\pm((1/2)(-\lambda_0-\lambda_1-2\lambda_2)+(1/2)\mu_3+(1/3)\nu+w),
\vspace{0.5mm}\\
\pm((-1/2)(\lambda_0+\lambda_1-2\lambda_2)+(1/2)\mu_3+(1/3)\nu+w),
\vspace{0.5mm}\\
\pm((1/2)(\lambda_0+\lambda_1-2\lambda_2)+(1/2)\mu_3+(1/3)\nu+w),
\vspace{0.5mm}\\
\pm(\mu+w),\;\pm(-\mu+w),\;\pm(2w)
\end{array}
\right\rbrace
\end{align*}
\end{theorem}
\begin{proof}
The roots of $ ({\mathfrak{e}_7}^C)^{\varepsilon_1,\varepsilon_2} $ are also the roots of $ ({\mathfrak{e}_8}^C)^{\varepsilon_1,\varepsilon_2} $. Indeed, let the root $ \alpha $ of $ ({\mathfrak{e}_7}^C)^{\varepsilon_1,\varepsilon_2} $ and its associated root vector $ \varPhi_s \in ({\mathfrak{e}_7}^C)^{\varepsilon_1,\varepsilon_2} \subset ({\mathfrak{e}_8}^C)^{\varepsilon_1,\varepsilon_2} $. Then, for $ R_8 \in \mathfrak{h}_8 $, we have
    \begin{align*}
    [R_8, \varPhi_s]&=[(\varPhi,0,0,w,0,0) (\varPhi_s,0,0,0,0,0)]
    \\
    &=([\varPhi,\varPhi_s],0,0,0,0,0)
    \\
    &=(\alpha(\varPhi)\varPhi_s,0,0,0,0,0)
    \\
    &=\alpha(\varPhi)(\varPhi_s,0,0,0,0,0)
    \\
    &=\alpha(R_8)\varPhi_s.
    \end{align*}

    We will determine the remainders of roots. We will show a few examples.

    \noindent First, let $ R_8=(\varPhi,0,0,w,0,0) \in \mathfrak{h}_8 $ and $ {R^-}_{\dot{E}_1}:=(0,\dot{E}_1,0,0,0,0) \in ({\mathfrak{e}_8}^C)^{\varepsilon_1,\varepsilon_2} $, where $ \dot{E}_1:=(E_1,0,0,0) \in (\mathfrak{P}_{\sH})^C $. Then it follows that
    \begin{align*}
    [R_8,{R^-}_{\dot{E}_1}]&=[(\varPhi,0,0,w,0,0),(0,\dot{E}_1,0,0,0,0)]
    \\
    &=(0,\varPhi\dot{E}_1+r\dot{E}_1,0,0,0,0)
    \\
    &=(0,\varPhi(\phi,0,0,\mu)(E_1,0,0,0)+w(E_1,0,0,0),0,0,0,0)
    \\
    &=(0,(\phi-(1/3)\mu+w)E_1,0,0,0),0,0,0,0),\,\,(\phi=\tilde{T}_0)
    \\
    &=(0,((\mu_1-(1/3)\mu+w)E_1,0,0,0),0,0,0,0)
    \\
    &=(\mu_1-(1/3)\mu+w)(0,\dot{E}_1,0,0,0,0)
    \\
    &=(\mu_1-(1/3)\mu+w){R^-}_{\dot{E}_1},
    \end{align*}
    that is, $ [R_8,{R^-}_{\dot{E}_1}]=(\mu_1-(1/3)\mu+w){R^-}_{\dot{E}_1} $. Hence we see that $ \mu_1-(1/3)\mu+w $ is a root and $ (0,\dot{E}_1,0,0,0,0)$ is an associated root vector.
    Next, let $ {R^-}_{\dot{F}_1(1+ie_1)}:=(0,\dot{F}_1(1+ie_1),0,0,0,0) \in ({\mathfrak{e}_8}^C)^{\varepsilon_1,\varepsilon_2} $, where $ \dot{F}_1(1+ie_1):=(F_1(1+ie_1),0,0,0) \in \mathfrak{P}^C $. Then it follows that
\begin{align*}
[R_8,{R^-}_{\dot{F}_1(1+ie_1)}]&=[(\varPhi,0,0,w,0,0),(0,\dot{F}_1(1+ie_1),0,0,0,0)]
\\
&=(0,\varPhi\dot{F}_1(1+ie_1)+r\dot{F}_1(1+ie_1),0,0,0,0)
\\
&=(0,\varPhi(\phi,0,0,\mu)(F_1(1+ie_1),0,0,0)+r(F_1(1+ie_1),0,0,0),0,0,0,0)
\\
&=(0,(\phi-(1/3)\mu+w)F_1(1+ie_1),0,0,0),0,0,0,0),\,\,(\phi=(L_1,L_2,L_3)+\tilde{T}_0)
\\
&=(0,((-\lambda_0+(1/2)(\mu_2+\mu_3)-(1/3)\mu+w)F_1(1+ie_1),0,0,0),0,0,0,0)
\\
&=(0,(-\lambda_0-(1/2)\mu_1-(1/3)\mu+w)\dot{F}_1(1+ie_1),0,0,0,0)
\\
&=(-\lambda_0-(1/2)\mu_1-(1/3)\mu+w){R^-}_{\dot{F}_1(1+ie_1)},
\end{align*}
that is, $ [R_8,{R^-}_{\dot{F}_1(1+ie_1)}]=(-\lambda_0-(1/2)\mu_1-(1/3)\mu+w){R^-}_{\dot{F}_1(1+ie_1)} $. Hence we see that $ -\lambda_0-(1/2)\mu_1-(1/3)\mu+w $ is a root and $ (0,\dot{F}_1(1+ie_1),0,0,0,0) $ is its associated root vector. All of roots and these associated root vectors except ones of the Lie algebra $ ({\mathfrak{e}_7}^C)^{\varepsilon_1,\varepsilon_2} $ above are obtained as follows:
\begin{longtable}[c]{cl}
	\hspace{2mm}
	$ \text{roots}  $
	& \hspace{-5mm}
	$ \text{associated root vectors} $
	\cr
   $ \mu_1-(1/3)\mu+w $ 
	&
	$ (0,\dot{E}_1,0,0,0,0) $
	\cr
	$ -(\mu_1-(1/3)\mu+w) $
	&
	$ (0,0,\text{\d{$ E $}}_1,0,0,0) $
	\vspace{1.5mm}\cr
	$ \mu_2-(1/3)\mu+w $ 
	&
	$ (0,\dot{E}_2,0,0,0,0) $
	\cr
	$ -(\mu_2-(1/3)\mu+w) $
	&
	$ (0,0,\text{\d{$ E $}}_2,0,0,0) $
	\vspace{1.5mm}\cr
	$ \mu_3-(1/3)\mu+w $ 
	&
	$ (0,\dot{E}_3,0,0,0,0) $
	\cr
	$ -(\mu_3-(1/3)\mu+w) $
	&
	$ (0,0,\text{\d{$ E $}}_3,0,0,0) $
	\vspace{1.5mm}\cr
	$ -\lambda_0-(1/2)\mu_1-(1/3)\mu+w $ 
	&
	$ (0,\dot{F_1}(1+ie_1),0,0,0,0) $
	\cr
   $ -(-\lambda_0-(1/2)\mu_1-(1/3)\mu+w) $
	&
	$ (0,0,\text{\d{$ F_1 $}}(1-ie_1),0,0,0) $
	\vspace{1.5mm}\cr
	$ \lambda_0-(1/2)\mu_1-(1/3)\mu+w $ 
	&
	$ (0,\dot{F_1}(1-ie_1),0,0,0,0) $
	\cr
	$ -(\lambda_0-(1/2)\mu_1-(1/3)\mu+w) $
	&
	$ (0,0,\text{\d{$ F_1 $}}(1+ie_1),0,0,0) $
	\vspace{1.5mm}\cr
   $ -\lambda_1-(1/2)\mu_1-(1/3)\mu+w $ 
	&
	$ (0,\dot{F_1}(e_2+ie_3),0,0,0,0) $
	\cr
   $ -(-\lambda_0-(1/2)\mu_1-(1/3)\mu+w) $
	&
	$ (0,0,\text{\d{$ F_1 $}}(e_2-ie_3),0,0,0) $
	\vspace{1.5mm}\cr
	$ \lambda_1-(1/2)\mu_1-(1/3)\mu+w $ 
	&
	$ (0,\dot{F_1}(e_2-ie_3),0,0,0,0) $
	\cr
	$ -(\lambda_1-(1/2)\mu_1-(1/3)\mu+w) $
	&
	$ (0,0,\text{\d{$ F_1 $}}(e_2+ie_3),0,0,0) $
	\vspace{1.5mm}\cr
	$ (-1/2)(-\lambda_0+\lambda_1+2\lambda_2)-(1/2)\mu_2-(1/3)\mu+w $ 
	&
	$ (0,\dot{F_2}(1+ie_1),0,0,0,0) $
	\cr
	$ -((-1/2)(-\lambda_0+\lambda_1+2\lambda_2)-(1/2)\mu_2-(1/3)\mu+w) $
	&
	$ (0,0,\text{\d{$ F_2 $}}(1-ie_1),0,0,0) $
	\vspace{1.5mm}\cr
	$ (1/2)(-\lambda_0+\lambda_1+2\lambda_2)-(1/2)\mu_2-(1/3)\mu+w $ 
	&
	$ (0,\dot{F_2}(1-ie_1),0,0,0,0) $
	\cr
	$ -((-1/2)(-\lambda_0+\lambda_1+2\lambda_2)-(1/2)\mu_2-(1/3)\mu+w) $
	&
	$ (0,0,\text{\d{$ F_2 $}}(1+ie_1),0,0,0) $
	\vspace{1.5mm}\cr
   $ (-1/2)(-\lambda_0+\lambda_1-2\lambda_2)-(1/2)\mu_2-(1/3)\mu+w $ 
	&
	$ (0,\dot{F_2}(e_2+ie_3),0,0,0,0) $
	\cr
	$ -((-1/2)(-\lambda_0+\lambda_1-2\lambda_2)-(1/2)\mu_2-(1/3)\mu+w) $
	&
	$ (0,0,\text{\d{$ F_2 $}}(e_2-ie_3),0,0,0) $
	\vspace{1.5mm}\cr
   $ (1/2)(-\lambda_0+\lambda_1-2\lambda_2)-(1/2)\mu_2-(1/3)\mu+w $ 
	&
	$ (0,\dot{F_2}(e_2-ie_3),0,0,0,0) $
	\cr
	$ -((-1/2)(-\lambda_0+\lambda_1-2\lambda_2)-(1/2)\mu_2-(1/3)\mu+w) $
	&
	$ (0,0,\text{\d{$ F_2 $}}(e_2+ie_3),0,0,0) $
	\vspace{1.5mm}\cr
	$(-1/2)(-\lambda_0-\lambda_1-2\lambda_2)-(1/2)\mu_3-(1/3)\nu+w $ 
	&
	$ (0,\dot{F_3}(1+ie_1),0,0,0,0) $
	\cr
	$ -((-1/2)(-\lambda_0-\lambda_1-2\lambda_2)-(1/2)\mu_3-(1/3)\nu+w) $
	&
	$ (0,0,\text{\d{$ F_3 $}}(1-ie_1),0,0,0) $
	\vspace{1.5mm}\cr
	$ (1/2)(-\lambda_0-\lambda_1-2\lambda_2)-(1/2)\mu_3-(1/3)\nu+w $ 
	&
	$ (0,\dot{F_3}(1-ie_1),0,0,0,0) $
	\cr
	$ -((1/2)(-\lambda_0-\lambda_1-2\lambda_2)-(1/2)\mu_3-(1/3)\nu+w) $
	&
	$ (0,0,\text{\d{$ F_3 $}}(1+ie_1),0,0,0) $
	\vspace{1.5mm}\cr
   $(-1/2)(\lambda_0+\lambda_1-2\lambda_2)-(1/2)\mu_3-(1/3)\nu+w $ 
	&
	$ (0,\dot{F_3}(e_2+ie_3),0,0,0,0) $
	\cr
	$ -((-1/2)(-\lambda_0-\lambda_1-2\lambda_2)-(1/2)\mu_3-(1/3)\nu+w) $
	&
	$ (0,0,\text{\d{$ F_3 $}}(e_2-ie_3),0,0,0) $
	\vspace{1.5mm}\cr
	$ (1/2)(\lambda_0+\lambda_1-2\lambda_2)-(1/2)\mu_3-(1/3)\nu+w $ 
	&
	$ (0,\dot{F_3}(e_2-ie_3),0,0,0,0) $
	\cr
	$ -((1/2)(\lambda_0+\lambda_1-2\lambda_2)-(1/2)\mu_3-(1/3)\nu+w) $
	&
	$ (0,0,\text{\d{$ F_3 $}}(e_2+ie_3),0,0,0) $
	\vspace{1.5mm}\cr
   $ -\mu_1+(1/3)\mu+w $   
	&
	$ (0,\text{\d{$ E $}}_1,0,0,0,0) $
	\cr
	$ -(-\mu_1+(1/3)\mu+w) $
	&
	$ (0,0,\dot{E}_1,0,0,0) $
	\vspace{1.5mm}\cr
	$ -\mu_2+(1/3)\mu+w $ 
	&
	$ (0,\text{\d{$ E $}}_2,0,0,0,0) $
	\cr
	$ -(-\mu_2+(1/3)\mu+w) $
	&
	$ (0,0,\dot{E}_2,0,0,0) $
	\vspace{1.5mm}\cr
	$ -\mu_3+(1/3)\mu+w $ 
	&
	$ (0,\text{\d{$ E $}}_3,0,0,0,0) $
	\cr
	$ -(-\mu_3+(1/3)\mu+w) $
	&
	$ (0,0,\dot{E}_3,0,0,0) $
	\vspace{1.5mm}\cr
	$ -\lambda_0+(1/2)\mu_1+(1/3)\mu+w $ 
	&
	$ (0,\text{\d{$ F_1 $}}(1+ie_1),0,0,0,0) $
	\cr
	$ -(-\lambda_0+(1/2)\mu_1+(1/3)\mu+w) $
	&
	$ (0,0,\dot{F_1}(1-ie_1),0,0,0) $
	\vspace{1.5mm}\cr
	$ \lambda_0+(1/2)\mu_1+(1/3)\mu+w $ 
	&
	$ (0,\text{\d{$ F_1 $}}(1-ie_1),0,0,0,0) $
	\cr
	$ -(\lambda_0+(1/2)\mu_1+(1/3)\mu+w) $
	&
	$ (0,0,\dot{F_1}(1+ie_1),0,0,0) $
	\vspace{1.5mm}\cr
   $ -\lambda_1+(1/2)\mu_1+(1/3)\mu+w $ 
	&
	$ (0,\text{\d{$ F_1 $}}(e_2+ie_3),0,0,0,0) $
	\cr
	$ -(-\lambda_0+(1/2)\mu_1+(1/3)\mu+w) $
	&
	$ (0,0,\dot{F_1}(e_2-ie_3),0,0,0) $
	\vspace{1.5mm}\cr
	$ \lambda_1+(1/2)\mu_1+(1/3)\mu+w $ 
	&
	$ (0,\text{\d{$ F_1 $}}(e_2-ie_3),0,0,0,0) $
	\cr
	$ -(\lambda_1+(1/2)\mu_1+(1/3)\mu+w) $
	&
	$ (0,0,\dot{F_1}(e_2+ie_3),0,0,0) $
	\vspace{1.5mm}\cr
	$ (-1/2)(-\lambda_0+\lambda_1+2\lambda_2)+(1/2)\mu_2+(1/3)\mu+w $ 
	&
	$ (0,\text{\d{$ F_2 $}}(1+ie_1),0,0,0,0) $
	\cr
	$ -((-1/2)(-\lambda_0+\lambda_1+2\lambda_2)+(1/2)\mu_2+(1/3)\mu+w) $
	&
	$ (0,0,\dot{F_2}(1-ie_1),0,0,0) $
	\vspace{1.5mm}\cr
	$ (1/2)(-\lambda_0+\lambda_1+2\lambda_2)+(1/2)\mu_2+(1/3)\mu+w $ 
	&
	$ (0,\text{\d{$ F_2 $}}(1-ie_1),0,0,0,0) $
	\cr
	$ -((1/2)(-\lambda_0+\lambda_1+2\lambda_2)+(1/2)\mu_2+(1/3)\mu+w) $
	&
	$ (0,0,\dot{F_2}(1+ie_1),0,0,0) $
	\vspace{1.5mm}\cr
   $ (-1/2)(-\lambda_0+\lambda_1-2\lambda_2)+(1/2)\mu_2+(1/3)\mu+w $ 
	&
	$ (0,\text{\d{$ F_2 $}}(e_2+ie_3),0,0,0,0) $
	\cr
	$ -((-1/2)(-\lambda_0+\lambda_1-2\lambda_2)+(1/2)\mu_2+(1/3)\mu+w) $
	&
	$ (0,0,\dot{F_2}(e_2-ie_3),0,0,0) $
	\vspace{1.5mm}\cr
	$ (1/2)(-\lambda_0+\lambda_1-2\lambda_2)+(1/2)\mu_2+(1/3)\mu+w $ 
	&
	$ (0,\text{\d{$ F_2 $}}(e_2-ie_3),0,0,0,0) $
	\cr
	$ -((1/2)(-\lambda_0+\lambda_1-2\lambda_2)+(1/2)\mu_2+(1/3)\mu+w) $
	&
	$ (0,0,\dot{F_2}(e_2+ie_3),0,0,0) $
	\vspace{1.5mm}\cr
	$(-1/2)(-\lambda_0-\lambda_1-2\lambda_2)+(1/2)\mu_3+(1/3)\mu+w $ 
	&
	$ (0,\text{\d{$ F_3 $}}(1+ie_1),0,0,0,0) $
	\cr
	$ -((-1/2)(-\lambda_0-\lambda_1-2\lambda_2)+(1/2)\mu_3+(1/3)\mu+w)  $
	&
	$ (0,0,\dot{F_3}(1-ie_1),0,0,0) $
	\vspace{1.5mm}\cr
	$ (1/2)(-\lambda_0-\lambda_1-2\lambda_2)+(1/2)\mu_3+(1/3)\mu+w $ 
	&
	$ (0,\text{\d{$ F_3 $}}(1-ie_1),0,0,0,0) $
	\cr
	$ -((1/2)(-\lambda_0-\lambda_1-2\lambda_2)+(1/2)\mu_3+(1/3)\mu+w) $
	&
	$ (0,0,\dot{F_3}(1+ie_1),0,0,0) $
	\vspace{1.5mm}\cr
   $(-1/2)(\lambda_0+\lambda_1-2\lambda_2)+(1/2)\mu_3+(1/3)\mu+w $ 
	&
	$ (0,\text{\d{$ F_3 $}}(e_2+ie_3),0,0,0,0) $
	\cr
	$ -((-1/2)(\lambda_0+\lambda_1-2\lambda_2)+(1/2)\mu_3+(1/3)\mu+w)  $
	&
	$ (0,0,\dot{F_3}(e_2-ie_3),0,0,0) $
	\vspace{1.5mm}\cr
	$ (1/2)(\lambda_0+\lambda_1-2\lambda_2)+(1/2)\mu_3+(1/3)\mu+w $ 
	&
	$ (0,\text{\d{$ F_3 $}}(e_2-ie_3),0,0,0,0) $
	\cr
	$ -((1/2)(\lambda_0+\lambda_1-2\lambda_2)+(1/2)\mu_3+(1/3)\mu+w) $
	&
	$ (0,0,\dot{F_3}(e_2+ie_3),0,0,0) $
	\vspace{1.5mm}\cr
	$ \mu+w $  
	&
	$ (0,\dot{1},0,0,0,0) $
	\cr
	$ -(\mu+w) $
	&
	$ (0,0,\underset{\dot{}}{1},0,0,0) $
	\cr
	$ -\mu+w $
	&
	$ (0,\text{\d{$ 1 $}},0,0,0,0) $
	\cr
	$ -(-\mu+w) $
	&
	$ (0,0,\dot{1},0,0,0) $
	\vspace{1.5mm}\cr
	$ 2w $ 
	&
	$ (0,0,0,0,1,0)  $
	\cr
	$ -2w $
	&
	$ (0,0,0,0,0,1)  $,
	\cr
\end{longtable}
where $ \dot{F_k}(x):=(F_k(x),0,0,0),\underset{\dot{}}{F_k}(x):=(0,F_k(x),0,0),\dot{1}:=(0,0,1,0), \underset{\dot{}}{1}:=(0,0,0,1) $ in $ \mathfrak{P}^C $.

Thus, since $ ({\mathfrak{e}_8}^C)^{\varepsilon_1,\varepsilon_2} $ is spanned by $ \mathfrak{h}_8 $ and associated root vectors above, the roots obtained above are all.
\end{proof}

Subsequently, we prove the following theorem.

\begin{theorem}\label{theorem 8.18}
In the root system $ \varDelta $ of Theorem {\rm \ref{theorem 8.17}}
 \begin{align*}
      \varPi=\left\lbrace \alpha_1,\alpha_2, \alpha_3, \alpha_4,\alpha_5,
      \alpha_6, \alpha_7 \right\rbrace
  \end{align*}
is a fundamental root system of $  ({\mathfrak{e}_8}^C)^{\varepsilon_1,
\varepsilon_2} $, where
$
\alpha_1=\mu+w,
\alpha_2=-2w,
\alpha_3=\mu_1-(1/3)\mu+w,
\alpha_4=-(1/2)(\lambda_0+\lambda_1+2\lambda_2+\mu_1-\mu_2),
\alpha_5=-(1/2)(-2\lambda_0+\mu_2-\mu_3),
\alpha_6=-(\lambda_0-\lambda_1),
\alpha_7=2\lambda_2
 $.
The Dynkin diagram of $ ({\mathfrak{e}_8}^C)^{\varepsilon_1,\varepsilon_2}
$ is given by
\vspace{-2mm}
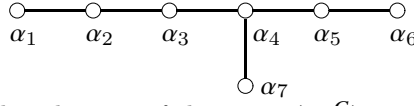
\begin{center}
\setlength{\unitlength}{1mm}
  \scalebox{1.0}
	{\setlength{\unitlength}{1mm}
\begin{picture}(100,20)
\put(20,10){\circle{2}} \put(19,6){$\alpha_1$}
\put(21,10){\line(1,0){8}}
\put(30,10){\circle{2}} \put(29,6){$\alpha_2$}
\put(31,10){\line(1,0){8}}
\put(40,10){\circle{2}} \put(39,6){$\alpha_3$}
\put(41,10){\line(1,0){8}}
\put(50,10){\circle{2}} \put(51,6){$\alpha_4$}
\put(51,10){\line(1,0){8}}
\put(50,8.8){\line(0,-1){7.8}}
\put(50,0){\circle{2}} \put(52,-1){$\alpha_7$}
\put(51,10){\line(1,0){8}}
\put(60,10){\circle{2}} \put(59,6){$\alpha_5$}
\put(61,10){\line(1,0){8}}
\put(70,10){\circle{2}} \put(69,6){$\alpha_6$}
\end{picture}}

\end{center}

In particular, we have that the type of the group $ ({E_8}^C)^{\varepsilon_1,\varepsilon_2} $  as Lie algebras is $ E_7 $.
\end{theorem}
\begin{proof}
The all positive roots are expressed by $ \alpha_1,\alpha_2,\alpha_3,
\alpha_4,\alpha_5,\alpha_6 $ as follows:
   \begin{align*}
-(\lambda_0-\lambda_1)&=\alpha_6,   
\\
-(\lambda_0+\lambda_1)&=\alpha_1+2\alpha_2+3\alpha_3+4\alpha_4+2\alpha_5+\alpha_6+2\alpha_7,
\\
2\lambda_2&=\alpha_7,
\\
-(1/2)(-2\lambda_0+\mu_2-\mu_3)&=\alpha_5,
\\
(1/2)(-2\lambda_0-\mu_2+\mu_3)&=\alpha_1+2\alpha_2+3\alpha_3+4\alpha_4+3\alpha_5+2\alpha_6+2\alpha_7,
\\
-(1/2)(-2\lambda\_1+\mu_2-\mu_3)&=\alpha_5+\alpha_6,
\\
(1/2)(-2\lambda_1-\mu_2+\mu_3)&=\alpha_1+2\alpha_2+3\alpha_3+4\alpha_4+3\alpha_5+\alpha_6+2\alpha_7,
\\
(1/2)(\lambda_0-\lambda_1-2\lambda_2-\mu_1+\mu_3)&=\alpha_4+\alpha_5,
\\
-(1/2)(\lambda_0-\lambda_1-2\lambda_2+\mu_1-\mu_3)&=\alpha_4+\alpha_5+\alpha_6+\alpha_7,
\\
(1/2)(\lambda_0-\lambda_1+2\lambda_2-\mu_1+\mu_3)&=\alpha_4+\alpha_5+\alpha_7,
\\
-(1/2)(\lambda_0-\lambda_1+2\lambda_2+\mu_1-\mu_3)&=\alpha_4+\alpha_5+\alpha_6,
\\
-(1/2)(\lambda_0+\lambda_1+2\lambda_2+\mu_1-\mu_2)&=\alpha_4,
\\
-(1/2)(\lambda_0+\lambda_1+2\lambda_2-\mu_1+\mu_2)&=\alpha_1+2\alpha_2+3\alpha_3+3\alpha_4+2\alpha_5+\alpha_6+\alpha_7,
\\
(1/2)(-\lambda_0-\lambda_1+2\lambda_2+\mu_1-\mu_2)&=\alpha_1+2\alpha_2+3\alpha_3+3\alpha_4+2\alpha_5+\alpha_6+2\alpha_7,
\\
(1/2)(-\lambda_0-\lambda_1+2\lambda_2-\mu_1+\mu_2)&=\alpha_4+\alpha_7 
\\
\mu_1+(2/3)\mu&=\alpha_1+\alpha_2+\alpha_3,
\\
-(\mu_2+(2/3)\mu)&=\alpha_2+2\alpha_3+2\alpha_4+2\alpha_5+\alpha_6+\alpha_7,
\\
\mu_3+(2/3)\mu&=\alpha_1+\alpha_2+\alpha_3+2\alpha_4+2\alpha_5+
\alpha_6+\alpha_7,
\\
-\lambda_0-(1/2)\mu_1+(2/3)\mu&=\alpha_1+\alpha_2+\alpha_3+2\alpha_4+\alpha_5+\alpha_6+\alpha_7,
\\
-(\lambda_0-(1/2)\mu_1+(2/3)\mu)&=\alpha_2+2\alpha_3+2\alpha_4+\alpha_5+
\alpha_6+\alpha_7,
\\
-\lambda_1-(1/2)\mu_1+(2/3)\mu&=\alpha_1+\alpha_2+\alpha_3+2\alpha_4+\alpha_5+\alpha_7,
\\
-(\lambda_1-(1/2)\mu_1+(2/3)\mu)&=\alpha_2+2\alpha_3+\alpha_4+\alpha_5+\alpha_7,
\\
-(1/2)(-\lambda_0+\lambda_1+2\lambda_2)-(1/2)\mu_2+(2/3)\mu&=\alpha_1+
\alpha_2+\alpha_3+\alpha_4+\alpha_5,
\\
(1/2)(-\lambda_0+\lambda_1+2\lambda_2)-(1/2)\mu_2+(2/3)\mu&=\alpha_1+\alpha_2+\alpha_3+\alpha_4+\alpha_5+\alpha_6+\alpha_7,
\\
-(1/2)(-\lambda_0+\lambda_1-2\lambda_2)-(1/2)\mu_2+(2/3)\mu
&=\alpha_1+\alpha_2+\alpha_3+\alpha_4+\alpha_5+\alpha_7,
\\
(1/2)(-\lambda_0+\lambda_1-2\lambda_2)-(1/2)\mu_2+(2/3)\mu
&=\alpha_1+\alpha_2+\alpha_3+\alpha_4+\alpha_5+\alpha_6,
\\
-((-1/2)(-\lambda_0-\lambda_1-2\lambda_2)-(1/2)\mu_3+(2/3)\mu)
&=\alpha_2+2\alpha_3+\alpha_4+2\alpha_5+\alpha_6+\alpha_7,
\\
(1/2)(-\lambda_0-\lambda_1-2\lambda_2)-(1/2)\mu_3+(2/3)\mu
&=\alpha_1+\alpha_2+\alpha_3+\alpha_4,
\\
-(1/2)(\lambda_0+\lambda_1-2\lambda_2)-(1/2)\mu_3+(2/3)\mu
&=\alpha_1+\alpha_2+\alpha_3+\alpha_4+\alpha_7,
\\
-((1/2)(\lambda_0+\lambda_1-2\lambda_2)-(1/2)\mu_3+(2/3)\mu)
&=\alpha_2+2\alpha_3+\alpha_4+2\alpha_5+\alpha_6+2\alpha_7,
\\
\mu_1-(1/3)\mu+w
&=\alpha_3,
\\
-(\mu_2-(1/3)\mu+w)
&=\alpha_1+2\alpha_2+2\alpha_3+2\alpha_4+2\alpha_5+\alpha_6+\alpha_7,
\\
\mu_3-(1/3)\mu+w
&=\alpha_3+2\alpha_4+2\alpha_5+\alpha_6+\alpha_7,
\\
-\lambda_0-(1/2)\mu_1-(1/3)\mu+w
&=\alpha_3+2\alpha_4+\alpha_5+\alpha_6+\alpha_7,
\\
-(\lambda_0-(1/2)\mu_1-(1/3)\mu+w)
&=\alpha_1+2\alpha_2+2\alpha_3+2\alpha_4+\alpha_5+\alpha_6+\alpha_7,
\\
-\lambda_1-(1/2)\mu_1-(1/3)\mu+w
&=\alpha_3+2\alpha_4+\alpha_5+\alpha_7,
\\
-(\lambda_1-(1/2)\mu_1-(1/3)\mu+w)
&=\alpha_1+2\alpha_2+2\alpha_3+2\alpha_4+\alpha_5+\alpha_7,
\\
-(1/2)(-\lambda_0+\lambda_1+2\lambda_2)-(1/2)\mu_2-(1/3)\mu+w
&=\alpha_3+\alpha_4+\alpha_5,
\\
(1/2)(-\lambda_0+\lambda_1+2\lambda_2)-(1/2)\mu_2-(1/3)\mu+w
&=\alpha_3+\alpha_4+\alpha_5+\alpha_6+\alpha_7,
\\
-(1/2)(-\lambda_0+\lambda_1-2\lambda_2)-(1/2)\mu_2-(1/3)\mu+w
&=\alpha_3+\alpha_4+\alpha_5+\alpha_7,
\\
(1/2)(-\lambda_0+\lambda_1+2\lambda_2)-(1/2)\mu_2-(1/3)\mu+w
&=\alpha_3+\alpha_4+\alpha_5+\alpha_6+\alpha_7,
\\
-((-1/2)(-\lambda_0-\lambda_1-2\lambda_2)-(1/2)\mu_3-(1/3)\mu+w)
&=\alpha_1+2\alpha_2+2\alpha_3+3\alpha_4+2\alpha_5+\alpha_6+2\alpha_7,
\\
(1/2)(-\lambda_0-\lambda_1-2\lambda_2)-(1/2)\mu_3-(1/3)\mu+w
&=\alpha_3+\alpha_4,
\\
(-1/2)(\lambda_0+\lambda_1-2\lambda_2)-(1/2)\mu_3-(1/3)\mu+w
&=\alpha_3+\alpha_4+\alpha_7,
\\
-((1/2)(\lambda_0+\lambda_1-2\lambda_2)-(1/2)\mu_3-(1/3)\mu+w)
&=\alpha_1+2\alpha_2+2\alpha_3+3\alpha_4+2\alpha_5+\alpha_6+2\alpha_7,
\\
-(-\mu_1+(1/3)\mu+w)
&=\alpha_2+\alpha_3,
\\
-\mu_2+(1/3)\mu+w
&=\alpha_1+\alpha_2+2\alpha_3+2\alpha_4+2\alpha_5+\alpha_6+\alpha_7,
\\
-(-\mu_3+(1/3)\mu+w)
&=\alpha_2+\alpha_3+2\alpha_4+2\alpha_5+\alpha_6+\alpha_7,
\\
-\lambda_0+(1/2)\mu_1+(1/3)\mu+w
&=\alpha_1+\alpha_2+2\alpha_3+2\alpha_4+\alpha_5+\alpha_6+\alpha_7,
\\
-(\lambda_0+(1/2)\mu_1+(1/3)\mu+w)
&=\alpha_2+\alpha_3+2\alpha_4+\alpha_5+\alpha_6+\alpha_7,
\\
-\lambda_1+(1/2)\mu_1+(1/3)\mu+w
&=\alpha_1+\alpha_2+2\alpha_3+2\alpha_4+\alpha_5+\alpha_7,
\\
-(\lambda_1+(1/2)\mu_1+(1/3)\mu+w)
&=\alpha_2+\alpha_3+2\alpha_4+\alpha_5+\alpha_7,
\\
-((-1/2)(-\lambda_0+\lambda_1+2\lambda_2)+(1/2)\mu_2+(1/3)\mu+w)
&=\alpha_2+\alpha_3+\alpha_4+\alpha_5+\alpha_6+\alpha_7,
\\
-((1/2)(-\lambda_0+\lambda_1+2\lambda_2)+(1/2)\mu_2+(1/3)\mu+w)
&=\alpha_2+\alpha_3+\alpha_4+\alpha_5,
\\
-((-1/2)(-\lambda_0+\lambda_1-2\lambda_2)+(1/2)\mu_2+(1/3)\mu+w)
&=\alpha_2+\alpha_3+\alpha_4+\alpha_5+\alpha_6,
\\
-((1/2)(-\lambda_0+\lambda_1-2\lambda_2)+(1/2)\mu_2+(1/3)\mu+w)
&=\alpha_2+\alpha_3+\alpha_4+\alpha_5+\alpha_7,
\\
-((-1/2)(-\lambda_0-\lambda_1-2\lambda_2)+(1/2)\mu_3+(1/3)\mu+w)
&=\alpha_2+\alpha_3+\alpha_4,
\\
(1/2)(-\lambda_0-\lambda_1-2\lambda_2)+(1/2)\mu_3+(1/3)\mu+w
&=\alpha_1+\alpha_2+2\alpha_3+3\alpha_4+2\alpha_5+\alpha_6+\alpha_7,
\\
(-1/2)(\lambda_0+\lambda_1-2\lambda_2)+(1/2)\mu_3+(1/3)\mu+w
&=\alpha_1+\alpha_2+2\alpha_3+3\alpha_4+2\alpha_5+\alpha_6+2\alpha_7,
\\
-((1/2)(\lambda_0+\lambda_1-2\lambda_2)+(1/2)\mu_3+(1/3)\mu+w)
&=\alpha_2+\alpha_3+\alpha_4+\alpha_7,
\\
\mu+w
&=\alpha_1,
\\
-(-\mu+w)
&=\alpha_1+\alpha_2,
\\
-2w
&=\alpha_2.
\end{align*}
Hence $ \varPi $ is a fundamental root system of $  ({\mathfrak{e}_8}
^C)^{\varepsilon_1,\varepsilon_2} $.

Then, for $ R, R' \in {\mathfrak{e}_8}^C$, the Killing form $ B_8 $ of $ {\mathfrak{e}_8}^C $ is given by
\begin{align*}
B_8(R,R')&=B_8((\varPhi,P,Q,r,s,t),(\varPhi',P',Q',r',s',t'))
\\
&=\dfrac{5}{3} B_7(\varPhi,\varPhi')+15\{Q,P'\}-15\{P,Q'\}+120rr'+60ts'+60st'
\end{align*}
(\cite[Theorem 5.3.2]{iy0}), so that for $
R_8:=(\varPhi,0,0,w,0,0),
{R_8}':=(\varPhi',0,0,w',0,0) \in \mathfrak{h}_8$, we have
\begin{align*}
 B_8(R_7,{R_8}')&=\dfrac{5}{3}B_7(\varPhi,\varPhi')+120ww'
\\
&=\dfrac{5}{3}(36(\lambda_0{\lambda_0}'+\lambda_1{\lambda_1}'
  +2\lambda_2{\lambda_2}')+18(\mu_1{\mu_1}'+\mu_2{\mu_2}'+\mu_3{\mu_3}')
  +24\mu\mu')+120ww'
\\
&=60(\lambda_0{\lambda_0}'+\lambda_1{\lambda_1}'
  +2\lambda_2{\lambda_2}')+30(\mu_1{\mu_1}'+\mu_2{\mu_2}'+\mu_3{\mu_3}')
  +40\mu\mu'+120ww',
\end{align*}
where $ \varPhi:=\varPhi(\lambda_0(iG_{01})+
 \lambda_1(iG_{23})+\lambda_2(i(G_{45}+G_{67})+(\mu_1E_1+\mu_2E_2+
 \mu_3E_3)^\sim,0,0,\mu ,\varPhi':=\varPhi({\lambda_0}'(iG_{01})+{\lambda_1}'(iG_{23})+
{\lambda_2}'(i(G_{45}+G_{67})+({\mu_1}'E_1+{\mu_2}'E_2+{\mu_3}'E_3)^\sim ,0,0,\mu')$.

Now, the canonical elements $ R_{\alpha_1},R_{\alpha_2}, R_{\alpha_3}, R_{\alpha_4},R_{\alpha_5},R_{\alpha_6},R_{\alpha_7} $ corresponding to
$ \alpha_1, \alpha_2,\alpha_3,\alpha_4, \allowbreak \alpha_5, \alpha_6, \alpha_7$ are determined as follows:
\begin{align*}
R_{\alpha_1}&=(\varPhi(0,0,0,\dfrac{1}{40}),0,0,\dfrac{1}{120},0,0,),
\\
R_{\alpha_2}&=(0,0,0,-\dfrac{1}{60},0,0),
\\
R_{\alpha_3}&=(\varPhi((\dfrac{1}{45}E_1-\dfrac{1}{90}E_2-\dfrac{1}{90}E_3)^\sim,0,0,-\dfrac{1}{120}),0,0,\dfrac{1}{120},0,0),
\\
R_{\alpha_4}&=(\varPhi(-\dfrac{1}{120}(iG_{01})-\!\dfrac{1}{120}(iG_{23})-
\dfrac{1}{120}(i(G_{45}+G_{67}))+(-\dfrac{1}{60}E_1+\dfrac{1}{60}E_2)^\sim,0,0,0),0,0,0,0,0),
\\
R_{\alpha_5}&=(\varPhi(\dfrac{1}{60}(iG_{01})+(-\dfrac{1}{60}E_2+\dfrac{1}{60}E_3)^\sim,0,0,0),0,0,0,0,0),
\\
R_{\alpha_6}&=(\varPhi(-\dfrac{1}{60}(iG_{01})+\dfrac{1}{60}(iG_{23}),0,0,0),0,0,0,0,0),
\\
R_{\alpha_7}&=(\varPhi(-\dfrac{1}{60}(i(G_{45}+G_{67})),0,0,0),0,0,0,0,0).
\end{align*}

Hence we have the following
\begin{align*}
(\alpha_1,\alpha_1)&=B_8(R_{\alpha_1},R_{\alpha_1})=40\cdot\left( -\dfrac{1}{40}\right)^2 +120\cdot\left(\dfrac{1}{120} \right)^2=\dfrac{1}{30},
\\
(\alpha_1,\alpha_2)&=B_8(R_{\alpha_1},R_{\alpha_2})=120\cdot\left( \dfrac{1}{120}\right)\left(-\dfrac{1}{60} \right)=-\dfrac{1}{60},
\\
(\alpha_1,\alpha_3)&=(\alpha_1,\alpha_4)=(\alpha_1,\alpha_5)=(\alpha_1,\alpha_6)=(\alpha_1,\alpha_7)=0,
\\
(\alpha_2,\alpha_2)&=B_8(R_{\alpha_2},R_{\alpha_2})=120\cdot\left( -\dfrac{1}{60}\right)=\dfrac{1}{30},
\\
(\alpha_2,\alpha_3)&=B_8(R_{\alpha_2},R_{\alpha_3})=120\cdot\left( -\dfrac{1}{60}\right)\left( \dfrac{1}{120}\right)=-\dfrac{1}{60},
\\
(\alpha_2,\alpha_4)&=(\alpha_2,\alpha_5)=(\alpha_2,\alpha_6)=(\alpha_2,\alpha_7)=0,
\\
(\alpha_3,\alpha_3)
&=B_8(R_{\alpha_3},R_{\alpha_3})=30\cdot\left(\left(\dfrac{1}{45} \right)^2+\left(-\dfrac{1}{90} \right)^2+\left(-\dfrac{1}{90} \right)^2   \right)+40\cdot\left( -\dfrac{1}{120}\right)^2+120\cdot\left( \dfrac{1}{120}\right)^2=\dfrac{1}{30},
\\
(\alpha_3,\alpha_4)&=B_8(R_{\alpha_3},R_{\alpha_4})=30\cdot\left( \left( \dfrac{1}{45}\right)\left(-\dfrac{1}{60} \right)+\left(-\dfrac{1}{90} \right)\left( \dfrac{1}{60}\right)    \right)=-\dfrac{1}{60},
\\
(\alpha_3,\alpha_5)&=(\alpha_3,\alpha_6)=(\alpha_3,\alpha_7)=0,
\\
(\alpha_4,\alpha_4)
&=B_8(R_{\alpha_4},R_{\alpha_4})=\!60\cdot \left(\left(-\dfrac{1}{120} \right)^2+\left(-\dfrac{1}{120} \right)^2+2\left(-\dfrac{1}{120} \right)^2   \right)+30\cdot\left( \left( -\dfrac{1}{60}\right)^2+\cdot\left( \dfrac{1}{60}\right)^2\right) \!=\!\dfrac{1}{30},
\\
(\alpha_4,\alpha_5)&=B_8(R_{\alpha_4},R_{\alpha_5})=60\cdot \left(-\dfrac{1}{120} \right)\left(\dfrac{1}{60} \right)+30\cdot \left(\dfrac{1}{60} \right) \left(-\dfrac{1}{60} \right)=-\dfrac{1}{60},
\\
(\alpha_4,\alpha_6)&=0,
\\
(\alpha_4,\alpha_7)&=60\cdot2\cdot \left(-\dfrac{1}{120} \right)\left( \dfrac{1}{60}\right)=-\dfrac{1}{60},
\\
(\alpha_5,\alpha_5)&=B_8(R_{\alpha_5},R_{\alpha_5})=60\cdot \left(\dfrac{1}{60} \right)^2+30\cdot \left(\left(-\dfrac{1}{60} \right)^2+\left(\dfrac{1}{60} \right)^2  \right)=\dfrac{1}{30},
\\
(\alpha_5,\alpha_6)&=B_8(R_{\alpha_5},R_{\alpha_6})=60\cdot\left(\dfrac{1}{60} \right)\left(-\dfrac{1}{60} \right) =-\dfrac{1}{60},
\\
(\alpha_5,\alpha_7)&=0,
\\
(\alpha_6,\alpha_6)&=B_8(R_{\alpha_6},R_{\alpha_6})=60\cdot \left(\left(-\dfrac{1}{60} \right)^2+\left(\dfrac{1}{60} \right)^2  \right)=\dfrac{1}{30},
\\
(\alpha_6,\alpha_7)&=0,
\\
(\alpha_7,\alpha_7)&=B_8(R_{\alpha_7},R_{\alpha_7})=60\cdot 2\cdot \left(\dfrac{1}{60} \right)^2=\dfrac{1}{30}.
\end{align*}

Thus, using the inner product above, we have
\begin{align*}
    \cos\theta_{12}&=\dfrac{(\alpha_1,\alpha_2)}{\sqrt{(\alpha_1,\alpha_1)(\alpha_2,\alpha_2)}}=-\dfrac{1}{2},\quad
    \cos\theta_{13}=\dfrac{(\alpha_1,\alpha_3)}{\sqrt{(\alpha_1,\alpha_1)
    (\alpha_3,\alpha_3)}}=0,
    \\
    \cos\theta_{14}&=\dfrac{(\alpha_1,\alpha_4)}{\sqrt{(\alpha_1,\alpha_1)(\alpha_4,\alpha_4)}}=0,\quad
    \cos\theta_{15}=\dfrac{(\alpha_1,\alpha_5)}{\sqrt{(\alpha_1,\alpha_1)
    (\alpha_5,\alpha_5)}}=0,
    \\
    \cos\theta_{16}&=\dfrac{(\alpha_1,\alpha_6)}{\sqrt{(\alpha_1,\alpha_1)(\alpha_6,\alpha_6)}}=0, \quad
    \cos\theta_{17}=\dfrac{(\alpha_1,\alpha_7)}{\sqrt{(\alpha_1,\alpha_1)
    (\alpha_7,\alpha_7)}}=0,
\\
    \cos\theta_{23}&=\dfrac{(\alpha_2,\alpha_3)}{\sqrt{(\alpha_2,\alpha_2)(\alpha_3,\alpha_3)}}=-\dfrac{1}{2}, \quad
    \cos\theta_{24}=\dfrac{(\alpha_2,\alpha_4)}{\sqrt{(\alpha_2,\alpha_2)
    (\alpha_4,\alpha_4)}}=0,
\\
    \cos\theta_{25}&=\dfrac{(\alpha_2,\alpha_5)}{\sqrt{(\alpha_2,\alpha_2)(\alpha_5,\alpha_5)}}=0, \quad
    \cos\theta_{26}=\dfrac{(\alpha_2,\alpha_6)}{\sqrt{(\alpha_2,\alpha_2)(\alpha_6,\alpha_6)}}=0,
\\
    \cos\theta_{27}&=\dfrac{(\alpha_2,\alpha_7)}{\sqrt{(\alpha_2,\alpha_2)(\alpha_7,\alpha_7)}}=0, \quad
    \cos\theta_{34}=\dfrac{(\alpha_3,\alpha_4)}{\sqrt{(\alpha_3,\alpha_3)
    (\alpha_4,\alpha_4)}}=-\dfrac{1}{2},
\\
    \cos\theta_{35}&=\dfrac{(\alpha_3,\alpha_5)}{\sqrt{(\alpha_3,\alpha_3)(\alpha_5,\alpha_5)}}=0,\quad
    \cos\theta_{36}=\dfrac{(\alpha_3,\alpha_6)}{\sqrt{(\alpha_3,\alpha_3)
    (\alpha_6,\alpha_6)}}=0,
\\
   \cos\theta_{37}&=\dfrac{(\alpha_3,\alpha_7)}{\sqrt{(\alpha_3,\alpha_3)
    (\alpha_7,\alpha_7)}}=0,\quad
   \cos\theta_{45}=\dfrac{(\alpha_4,\alpha_5)}{\sqrt{(\alpha_4,\alpha_4)(\alpha_5,\alpha_5)}}=-\dfrac{1}{2},
\\
    \cos\theta_{46}&=\dfrac{(\alpha_4,\alpha_6)}{\sqrt{(\alpha_4,\alpha_4)(\alpha_6,\alpha_6)}}=0, \quad
    \cos\theta_{47}=\dfrac{(\alpha_4,\alpha_7)}{\sqrt{(\alpha_4,\alpha_4)(\alpha_7,\alpha_7)}}=-\dfrac{1}{2},
\\
    \cos\theta_{56}&=\dfrac{(\alpha_5,\alpha_6)}{\sqrt{(\alpha_5,\alpha_5)(\alpha_6,\alpha_6)}}=-\dfrac{1}{2},\quad
     \cos\theta_{57}=\dfrac{(\alpha_5,\alpha_7)}{\sqrt{(\alpha_5,\alpha_5)(\alpha_7,\alpha_7)}}=0,
\\
     \cos\theta_{67}&=\dfrac{(\alpha_6,\alpha_7)}{\sqrt{(\alpha_6,\alpha_6)(\alpha_7,\alpha_7)}}=0.
\end{align*}
so that we can draw the required Dynkin diagram.

Therefore we have that the type of the group $ ({E_8}^C)^{\varepsilon_1,\varepsilon_2} $ as Lie algebras is $ E_7 $.
\end{proof}

In order to determine the structure of the group $ ({E_8}^C)^{\varepsilon_1,\varepsilon_2} $, we determine the center \\ $ z(({E_8}^C)^{\varepsilon_1,\varepsilon_2}) $ of $ ({E_8}^C)^{\varepsilon_1,\varepsilon_2} $.

\begin{proposition}\label{proposition 8.19}
The center $ z(({E_8}^C)^{\varepsilon_1,\varepsilon_2}) $ of the group $ ({E_8}^C)^{\varepsilon_1,\varepsilon_2} $ is isomorphic to a cyclic group of order two {\rm :}
\begin{align*}
z(({E_8}^C)^{\varepsilon_1,\varepsilon_2})=\left\lbrace1,\gamma \right\rbrace  \cong \Z_2.
\end{align*}
\end{proposition}
\begin{proof}
Let $ \alpha \in z(({E_8}^C)^{\varepsilon_1, \varepsilon_2}) $.
Then, since $ 1_- \in  (\mathfrak{W}^C)_{ \varepsilon_1, \varepsilon_2}$, we can set $ \alpha 1_-:=(\varPhi,P,Q,r,s,t) \in (\mathfrak{W}^C)_{ \varepsilon_1, \varepsilon_2} $.
For $ \beta \in ({E_7}^C)^{\varepsilon_1, \varepsilon_2} \subset ({E_8}^C)^{\varepsilon_1, \varepsilon_2}$, we see
\begin{align*}
\beta(\alpha 1_-)=\alpha(\beta 1_-)=\alpha 1_-(=(\varPhi,P,Q,r,s,t)),
\end{align*}
on the other hand, we see
\begin{align*}
\beta(\alpha 1_-)=\beta(\varPhi,P,Q,r,s,t)=(\beta\varPhi\beta^{-1},\beta P,\beta Q,r,s,t).
\end{align*}
Hence we have
\begin{align*}
    \left\lbrace
    \begin{array}{l}
    \beta \varPhi\beta^{-1}=\varPhi,
    \\
    \beta P=P,  \quad\qquad {\text{for all}}\,\, \beta \in ({E_7}^C)^{ \varepsilon_1, \varepsilon_2}
    \\
    \beta Q=Q.
    \end{array}\right.
\end{align*}
Here, let the element $ -1 \in ({E_7}^C)^{\varepsilon_1, \varepsilon_2} $ as $ \beta $. Then we have $ P=Q=0 $. Moreover, let the element $ \omega 1 \in ({E_6}^C)^{\varepsilon_1, \varepsilon_2} \subset ({E_7}^C)^{\varepsilon_1, \varepsilon_2} $ as $ \beta $, where $ \omega \in C,\omega^3=1, \omega \not=1 $. Then, for $ \varPhi:=\varPhi(\phi,A,B,\nu) $, we have $ A=B=0 $. Hence $ \alpha 1_- $ is of the form $ (\varPhi(\phi,0,0,\nu),0,0,r,s,t) $: $ \alpha 1_- =(\varPhi(\phi,0,0,\nu),0,0,r,s,t)  $, moreover we see $ t\not=0 $ for $ \alpha 1_-=(\varPhi(\phi,0,0,\nu),0,0,r,s,t) $.

\noindent Indeed, assume $ t=0 $, then we have $ r=0 $ from Lemma \ref{lemma 8.14} (6).
Hence $ \alpha 1_- $ is of the form $ ((\varPhi(\phi,0,0,\nu),0,0,0,s,0)$. In addition, $\varPhi(\phi,0,0,\nu) $ satisfies the condition
$ \varPhi Q'=0 $ for all $ Q' \in (\mathfrak{P}_{\sH})^C $ from Lemma \ref{lemma 8.14} (5), so set $ Q':=(0,0,0,1) $, then we have $ \nu=0 $. Moreover set $ Q':=(E,0,0,0) $ and $ Q':=(F_k(x), 0,0,0),x \in \H^C,k=1,2 $, then we have $ T=0  $ and $ D_i=0,i=1,2,3, a_i=0 $, where $ \phi:=(D_1,D_2,D_3)+\tilde{A}_1(a_1)+ \tilde{A}_2(a_2)+\tilde{A}_3(a_3)+\tilde{T} \in ({\mathfrak{e}_6}^C)^{\varepsilon_1,\varepsilon_2} $, that is, $ \phi=0 $.
Thus $ \alpha 1_- $ is of the form $ s1^-$: $ \alpha 1_-=s1^-, s\not=0 $.
Let the mapping $ \varphi: SU(2) \to E_8 \subset {E_8}^C $ which is defined in \cite[Theorem 5.7.4]{iy0} as $ \varphi_3 $. Then, for $ A_1:=
    \begin{pmatrix}
    \,0\, & \,i\, \\
    i & 0
    \end{pmatrix} \in SU(2)=\{A \in M(2,C)\,|\, (\tau\,{}^t\!A)A=E, \det A=1\} $, we choose the element $ \varphi(A_4)=\exp((\pi/2)\ad(0,0,0,0,i,i)) \in ({E_8}^C)^{\varepsilon_1,\varepsilon_2} $ (Lemma \ref{lemma 8.12} (1)).
    It follows that
    \begin{align*}
    \varphi(A_1)(\alpha 1_-)=\alpha(\varphi(A_1) 1_-)=\alpha 1^-,
    \end{align*}
    on the other hand, it follows that
    \begin{align*}
    \varphi(A_1)(\alpha 1_-)=\varphi(A_1)(s1^-)=s1_-.
    \end{align*}
    Thus we have $ \alpha 1^-=s1_- $.

    \noindent In addition, from $ B_8(\alpha 1_-, \alpha 1^-)=B_8(1_-,1^-) $, we have $ s^2=1 $. Moreover, it follows that
    \begin{align*}
    [\alpha 1_-, \alpha 1^-]=[s1^-,s1_-]=s^2[1^-,1_-]=s^2\tilde{1},
    \end{align*}
    on the other hand, it follows that
    \begin{align*}
    [\alpha 1_-, \alpha 1^-]=\alpha [1_-,1^-]=\alpha(-\tilde{1})=-\alpha \tilde{1}.
    \end{align*}
    Thus we have $ \alpha \tilde{1}=-\tilde{1} $ from $ s^2=1 $. Consequently, we obtain the following
    \begin{align*}
    \alpha \tilde{1}=-\tilde{1}, \quad \alpha 1^-=s1_-, \quad  \alpha 1_-=s1^-
    \end{align*}
    with $ s^2=1 $, that is,
    \begin{align*}
    {\rm (i)}
    \left\lbrace
    \begin{array}{l}
    \alpha \tilde{1}=-\tilde{1}
    \\
    \alpha 1^-=1_-
    \\
    \alpha 1_-=1^-,\,\,
    \end{array}
    \right.
    {\rm (ii)}
    \left\lbrace
    \begin{array}{l}
    \alpha \tilde{1}=-\tilde{1}
    \\
    \alpha 1^-=-1_-
    \\
    \alpha 1_-=1^-,\,\,
    \end{array}
    \right.
    {\rm (iii)}
    \left\lbrace
    \begin{array}{l}
    \alpha \tilde{1}=-\tilde{1}
    \\
    \alpha 1^-=1_-
    \\
    \alpha 1_-=-1^-,\,\,
    \end{array}
    \right.
    {\rm (iv)}
    \left\lbrace
    \begin{array}{l}
    \alpha \tilde{1}=-\tilde{1}
    \\
    \alpha 1^-=-1_-
    \\
    \alpha 1_-=-1^-.
    \end{array}
    \right.
    \end{align*}

Case (i). Again let the mapping $ \varphi $ be used above. \vspace{1mm} Then, for $ A_2:=\begin{pmatrix}
    \;\dfrac{i}{\sqrt{2}}\; & \;\dfrac{i}{\sqrt{2}}\;
    \vspace{1mm}\\
    \dfrac{i}{\sqrt{2}} & -\dfrac{i}{\sqrt{2}}
    \end{pmatrix} \in SU(2)$, \vspace{1mm} we choose the element $ \varphi(A_2)=\exp((\pi/2)\ad(0,0,0,-i/\sqrt{2},i/\sqrt{2},i/\sqrt{2} )) \in ({E_8}^C)^{\varepsilon_1,\varepsilon_2} $ (Lemma \ref{lemma 8.12}(1)). It follows from $ \varphi(A_1)\alpha=\alpha\varphi(A_1) $ that
    \begin{align*}
    \varphi(A_2)(\alpha \tilde{1})=\alpha\varphi(A_1)\tilde{1}=\alpha(0,0,0,0,1,1)=1^-+1_-,
    \end{align*}
    on the other hand, it follows that
    \begin{align*}
    \varphi(A_2)(\alpha \tilde{1})=\varphi(A_1)(-\tilde{1})=(0,0,0,0,-1,-1)=-(1^-+1_-).
    \end{align*}
    Hence this result is contradiction, so that this case is impossible.

    Case (ii). It follows that
    \begin{align*}
    [\alpha 1^-, \alpha 1_-]=[-1_-, 1^-]=-[1_-, 1^-]=-(2\tilde{1})=2\tilde{1},
    \end{align*}
    on the other hand, it follows that
    \begin{align*}
    [\alpha 1^-, \alpha 1_-]=\alpha [1_-, 1^-]=\alpha[1_-, 1^-]=\alpha(2\tilde{1})=2\alpha\tilde{1}
    \end{align*}
    Hence this results are contrary to $ \alpha\tilde{1}=-\tilde{1} $, so that this case is also impossible.

    Case (iii). The same result as in Case (ii) is obtained, so that this case is also impossible.

    Case (iv). As in Case (i), let the mapping $ \varphi $. Then, for $ A_3:=\begin{pmatrix}
    \;\dfrac{1}{\sqrt{2}}\; & \;\dfrac{i}{\sqrt{2}}\;
    \vspace{1mm} \\
    \dfrac{i}{\sqrt{2}} & \dfrac{1}{\sqrt{2}}
    \end{pmatrix} \in SU(2)$, \vspace{1mm} we choose the element $ \varphi(A_3)=\exp((\pi/4)\ad(0,0,0,0,i,i )) \in ({E_8}^C)^{\varepsilon_1,\varepsilon_2} $ (Lemma \ref{lemma 8.12} (1)). It follows that
    \begin{align*}
    \varphi(A_3)(\alpha \tilde{1})=\alpha\varphi(A_3)\tilde{1}=\alpha(0,0,0,0,-i,i)=(-i)1^-+i1_-,
    \end{align*}
    on the other hand, it follows from $ \varphi(A_3)\alpha=\alpha\varphi(A_3) $ that
    \begin{align*}
    \varphi(A_3)(\alpha \tilde{1})=\varphi(A_3)(-\tilde{1})=-(0,0,0,0,-i,i)=i1^-+(-i)1_-.
    \end{align*}
    Hence this result is contradiction, so that this case is impossible.

    With above, the assumption as $ t=0 $ is unreasonable. Thus we obtain
    \begin{align*}
    \alpha 1_-=(\varPhi, 0,0,r,s,t),\,\, t\not=0.
    \end{align*}
    Here, from Lemma \ref{lemma 8.14} (2), we have
    \begin{align*}
    \varPhi=-\dfrac{1}{2t}Q \times Q,
    \end{align*}
so that we have $ \varPhi=0 $  from $ Q=0 $. Therefore $ \alpha 1_- $ is of the form
\begin{align*}
\alpha 1_-=(0, 0,0,r,s,t),\,\, t\not=0.
\end{align*}

Now, let the mapping $ \varphi $. Then, for $ A_4:=
    \begin{pmatrix}
    \,0\, & \,-1\, \\
    1 & 0
    \end{pmatrix} \in SU(2) $, we choose the elements $ \varphi(A_4)=\exp((\pi/2)\ad(0,0,0,0,1,-1)) \in ({E_8}^C)^{\varepsilon_1,\varepsilon_2 } $ (Lemma \ref{lemma 8.12}(1)). Using $ \varphi(A_1) $ defined above together, it follows that
    \begin{align*}
    \varphi(A_1)\varphi(A_4)(\alpha 1_-)=\alpha(\varphi(A_1)\varphi(A_4)1_-)=\alpha(-1_-)=-\alpha 1_-=(0,0,0,-r,-s,-t),
    \end{align*}
    on the other hand, it follows that
    \begin{align*}
    \varphi(A_1)\varphi(A_4)(\alpha 1_-)
    &= \varphi(A_1)\varphi(A_4)(0,0,0,r,s,t)=\varphi(A_1)(0,0,0,-r,-t,-s)
    \\
    &=(0,0,0,r-s,-t).
    \end{align*}
    Hence we have $ r=0 $, so $ \alpha 1_-=(0,0,0,0,s,t), t\not=0 $. Here, since $ \alpha 1_- \in \mathfrak{W}_{\varepsilon_1,\varepsilon_2} $ again, we see $ st=0 $ from Lemma \ref{lemma 8.14} (6). Thus we have $ s=0 $ because of $ t\not=0 $.
     Namely, we have
    \begin{align*}
     \alpha 1_-=t1_- ,t\not=0.
    \end{align*}

   Again, choose the element $ \varphi(A_1) $. Then it follows that
    \begin{align*}
    \varphi(A_1)(\alpha 1_-)=\alpha \varphi(A_1)1_-=\alpha 1^-,
    \end{align*}
    on the other hand, it follows that
    \begin{align*}
    \varphi(A_1)(\alpha 1_-)=\varphi(A_1)(t1_-)=t1^-,
    \end{align*}
    that is, $ \alpha 1^-=t1^-,t\not=0 $.

    \noindent Here, using the Killing form $ B_8 $, it follows that
    \begin{align*}
    B_8(\alpha 1^-, \alpha 1_-)=B_8(1^-,1_-)=60,
    \end{align*}
    on the other hand, it follows that
    \begin{align*}
    B_8(\alpha 1^-, \alpha 1_-)=B_8(t1^-,t1_-)=60t^2.
    \end{align*}
    Hence we have $ t^2=1 $,

    \noindent so that it follows from $ \alpha 1_-=t1_- , \alpha 1^-=t1^-,t\not=0 $ that
     \begin{align*}
     [\alpha 1^-, \alpha 1_-]&=\alpha[1^-,1_-]=\alpha\tilde{1},
     \\
     [\alpha 1^-, \alpha 1_-]&=[t1^-,t1_-]=t^2[1^-,1_-]=t^2\tilde{1}=\tilde{1},
     \end{align*}
    that is, $ \alpha \tilde{1}=\tilde{1}$.

    Thus, from $ t^2=1 $, we have the following
    \begin{align*}
    \left\lbrace
    \begin{array}{l}
    \alpha \tilde{1}=\tilde{1}
    \\
    \alpha 1^-=1^-
    \\
    \alpha 1_-=1_-
    \end{array}
    \right. \qquad {\text{or}} \qquad
    \left\lbrace
    \begin{array}{l}
    \alpha \tilde{1}=\tilde{1}
    \\
    \alpha 1^-=-1^-
    \\
    \alpha 1_-=-1_-.
    \end{array}
    \right.
    \end{align*}

In the latter case, we choose the element $ \varphi(A_2) \in ({E_8}^C)^{\varepsilon_1,\varepsilon_2} $ again, then it follows that
    \begin{align*}
    \varphi(A_2)(\alpha 1^-)=\alpha \varphi(A_2) 1^-=\alpha (1/2)(-\tilde{1}-1^-+1_-)=(1/2)(-\tilde{1}+1^--1_-),
    \end{align*}
    on the other hand, it follows that
    \begin{align*}
     \varphi(A_2)(\alpha 1^-)=\varphi(A_2)(-1^-)=-(1/2)(-\tilde{1}-1^-+1_-)=(1/2)(\tilde{1}+1^--1_-).
    \end{align*}
   Hence this result is contradiction, so that the latter case is impossible.

   In the former case, note that $ (({E_8}^C)^{\varepsilon_1,\varepsilon_2})_{\tilde{1},1^-,1_-}=(({E_8}^C)_{\tilde{1},1^-,1_-})^{\varepsilon_1,\varepsilon_2} \cong ({E_7}^C)^{\varepsilon_1,\varepsilon_2}$, we see $ \alpha \in ({E_7}^C)^{\varepsilon_1,\varepsilon_2} $.

   Thus we obtain $ \alpha \in z(({E_7}^C)^{\varepsilon_1,\varepsilon_2})=\{1, \gamma ,-1,-\gamma \}=\{1, \gamma\} \times \{1,-\gamma \} \cong \Z_2 \times \Z_2$ (\cite[Proposition 1.1.6]{miya1}). However, since it is easy to verify that $ -1, -\gamma \not\in {E_8}^C $, we have $  z(({E_8}^C)^{\varepsilon_1,\varepsilon_2}) \subset \{1, \gamma \} $ and vice versa.

   Therefore we have the required result
   \begin{align*}
   z(({E_8}^C)^{\varepsilon_1,\varepsilon_2})=\{1,\gamma \} \cong \Z_2.
   \end{align*}
\end{proof}

As for the group $ ({E_8}^C)^{\varepsilon_1,\varepsilon_2} $, the following is a summary of what we have obtained up to this point in this section.

\begin{itemize}
\setlength{\leftskip}{3mm}
\item The type of the group $ ({E_8}^C)^{\varepsilon_1,\varepsilon_2} $ as Lie algebras is $ E_7 $.
\item The group $ ({E_8}^C)^{\varepsilon_1,\varepsilon_2} $ is connected.
\item The center $ z(({E_8}^C)^{\varepsilon_1,\varepsilon_2}) $ is isomorphic to a cyclic group of order two: $ z(({E_8}^C)^{\varepsilon_1,\varepsilon_2}) \cong \Z_2 $.
\end{itemize}

Now, we determine the structure of the group $ ({E_8}^C)^{\varepsilon_1,\varepsilon_2} $.

\begin{theorem}\label{theorem 8.20}
The group $ ({E_8}^C)^{\varepsilon_1,\varepsilon_2} $ is isomorphic to the group $ {E_7}^C ${\rm :} $ ({E_8}^C)^{\varepsilon_1,\varepsilon_2} \cong {E_7}^C $.
\end{theorem}
\begin{proof}
Since the group $ ({E_8}^C)^{\varepsilon_1,\varepsilon_2} $ is connected (Theorem \ref{theorem 8.16}) and the type of the group $ ({E_8}^C)^{\varepsilon_1,\varepsilon_2} $ as Lie algebras is $ E_7 $ (Theorem \ref{theorem 8.18}), note that $ z({E_7}^C)=\{1,-1\} \cong \Z_2 $, we have
\begin{align*}
({E_8}^C)^{\varepsilon_1,\varepsilon_2} \cong {E_7}^C \qquad \text{ or } \qquad ({E_8}^C)^{\varepsilon_1,\varepsilon_2} \cong {E_7}^C /\Z_2.
\end{align*}
Then, since the center $ z(({E_8}^C)^{\varepsilon_1,\varepsilon_2}) $ is isomorphic to a cyclic group of order two (Proposition \ref{proposition 8.19}), we have the required isomorphism
\begin{align*}
({E_8}^C)^{\varepsilon_1,\varepsilon_2} \cong {E_7}^C.
\end{align*}
\end{proof}

We consider a real form $ (E_8)^{\varepsilon_1,\varepsilon_2} $ of the group $ ({E_8}^C)^{\varepsilon_1,\varepsilon_2} $:
\begin{align*}
(E_8)^{\varepsilon_1,\varepsilon_2}=\left\lbrace \alpha \in ({E_8}^C)^{\varepsilon_1,\varepsilon_2} \relmiddle{|} \langle \alpha R_1,\alpha
 R_2 \rangle=\langle R_1,R_2 \rangle \right\rbrace,
\end{align*}
where the Hermite inner product $ \langle R_1, R_2 \rangle $ is defined by $ (-1/15)B_8(\tau\lambda_\omega R_1, R_2) $: $  \langle R_1, R_2 \rangle =(-1/15)\allowbreak B_8(\tau\lambda_\omega R_1, R_2)$. Then, since $ \alpha \in (E_8)^{\varepsilon_1,\varepsilon_2} $ commutes with $ \varepsilon_i, i=1,2 $, $ \alpha $ induces a $ C $-linear isomorphism of $ ({\mathfrak{e}_8}^C)^{\varepsilon_1,\varepsilon_2} $. Hence
the group $ (E_8)^{\varepsilon_1,\varepsilon_2} $ is a compact Lie group as the closed subgroup of the unitary group $ U(133)=\left\lbrace \alpha \in \Iso_C(({\mathfrak{e}_8}^C)^{\varepsilon_1,\varepsilon_2})\relmiddle{|} \langle \alpha R_1, \alpha R_2 \rangle=\langle R_1, R_2 \rangle \right\rbrace  $. \vspace{1mm}

Then we have the following lemma.

\begin{lemma}\label{lemma 8.21}
The Lie algebra $ (\mathfrak{e}_8)^{\varepsilon_1,\varepsilon_2} $ of the group $  (E_8)^{\varepsilon_1,\varepsilon_2} $ is given by
\begin{align*}
    (\mathfrak{e}_8)^{\varepsilon_1,\varepsilon_2}&=
    \left\lbrace R=(\varPhi,P,-\tau\lambda P,r,s,-\tau s) \relmiddle{|}
    \begin{array}{l}
    \varPhi \in (\mathfrak{e}_7)^{\varepsilon_1,\varepsilon_2}, P \in (\mathfrak{P}_{\sH})^C, r \in i\R, s\in C
    \end{array}
    \right\rbrace,
    \end{align*}
where $ (\mathfrak{e}_7)^{\varepsilon_1,\varepsilon_2} $ is defined by
\begin{align*}
(\mathfrak{e}_7)^{\varepsilon_1,\varepsilon}
=\left\lbrace \varPhi(\phi,A,-\tau A ,\nu) \relmiddle{|}
\begin{array}{l}
\phi=(D_1,D_2,D_3) \vspace{1mm}\\
            \quad\;\; +\tilde{A}_1(a_1)+\tilde{A}_2(a_2)+\tilde{A}_3(a_3)\\
            \quad\;\; +i(\tau_1E_1+\tau_2E_2+\tau_3E_3+F_1(t_1) \\
            \quad\;\; +F_2(t_2)+F_3(t_3))^\sim,
            \\
            \qquad D_1=d_{01}G_{01}+d_{02}G_{02}+d_{03}G_{03} \\
            \qquad +d_{12}G_{12}+d_{13}G_{13}+d_{23}G_{23} \\
            \qquad +d_{45}(G_{45}+G_{67})+d_{46}(G_{46}-G_{57}) \\
            \qquad +d_{47}(G_{47}+G_{56}),d_{ij} \in \R, \\
            \qquad D_2=\pi\kappa D_1, D_3=\kappa\pi D_1,\\
            \qquad a_k, t_k \in \H ,\\
            \qquad \tau_k \in \R, \tau_1+\tau_2+\tau_3=0,\\
         \quad A \in (\mathfrak{J}_{\sH})^C,
         \nu \in i\R \\
\end{array}
 \right\rbrace.
\end{align*}

    In particular, we have $ \dim((\mathfrak{e}_8)^{\varepsilon_1,\varepsilon_2})=66+64+3=133 $.
\end{lemma}
\begin{proof}
Note that $ (\mathfrak{e}_8)^{\varepsilon_1,\varepsilon_2}=(({\mathfrak{e}_8}^C)^{\varepsilon_1,\varepsilon_2})^{\tau\lambda_\omega} $, using the result of Lemma \ref{lemma 8.12}(1), we can obtain the required result by straightforward computations.
\end{proof}

We determine the structure of the group $ (E_8)^{\varepsilon_1,\varepsilon_2} $.

\begin{theorem}\label{theorem 8.22}
The group $ (E_8)^{\varepsilon_1,\varepsilon_2} $ is isomorphic to the group $ E_7 ${\rm :} $ (E_8)^{\varepsilon_1,\varepsilon_2} \cong E_7 $.
\end{theorem}
\begin{proof}
First, it is easy to verify that $ z(({E_8}^C)^{\varepsilon_1,\varepsilon_2}) \! \subset\! z((E_8)^{\varepsilon_1,\varepsilon_2}) $. Indeed,
let $ \delta \! \in \! z(({E_8}^C)^{\varepsilon_1,\varepsilon_2}) $. Then,
since $ (E_8)^{\varepsilon_1,\varepsilon_2} \subset ({E_8}^C)^{\varepsilon_1,\varepsilon_2} $, $ \delta $ commutes with any elements $ \beta \in (E_8)^{\varepsilon_1,\varepsilon_2} $. Hence we have $ \delta \in z((E_8)^{\varepsilon_1,\varepsilon_2}) $, that is, $ z(({E_8}^C)^{\varepsilon_1,\varepsilon_2}) \subset z((E_8)^{\varepsilon_1,\varepsilon_2}) $. Here, from $ ({E_8}^C)^{\varepsilon_1,\varepsilon_2} \cong {E_7}^C $ (Theorem \ref{theorem 8.20}), the group $ ({E_8}^C)^{\varepsilon_1,\varepsilon_2} $ is simply connected. Moreover, it follows from $ \varepsilon_i(\tau\lambda_\omega)=(\tau\lambda_\omega)\varepsilon_i,i=1,2 $ that
\begin{align*}
(E_8)^{\varepsilon_1,\varepsilon_2}=(({E_8}^C)^{\tau\lambda_\omega})^{\varepsilon_1,\varepsilon_2}=(({E_8}^C)^{\varepsilon_1,\varepsilon_2})^{\tau\lambda_\omega}.
\end{align*}
Hence the group $ (E_8)^{\varepsilon_1,\varepsilon_2} $ is connected. In addition, since the group $ (E_8)^{\varepsilon_1,\varepsilon_2} $ is a real form of the group $ ({E_8}^C)^{\varepsilon_1,\varepsilon_2} $, the type of $ (E_8)^{\varepsilon_1,\varepsilon_2} $ as Lie algebras is $ E_7 $. Thus, together with the fact that the group $ (E_8)^{\varepsilon_1,\varepsilon_2} $ is compact, we have the following
\begin{align*}
(E_8)^{\varepsilon_1,\varepsilon_2} \cong E_7 \quad \text{or} \quad (E_8)^{\varepsilon_1,\varepsilon_2} \cong E_7/\Z_2.
\end{align*}
In the latter case, since the center $ z(E_7/\Z_2) $ of the group $ E_7/\Z_2 $ is trivial, the center $ z((E_8)^{\varepsilon_1,\varepsilon_2}) $ of the group $ (E_8)^{\varepsilon_1,\varepsilon_2} $ is also trivial. However, it follows from $ z(({E_8}^C)^{\varepsilon_1,\varepsilon_2}) \subset z((E_8)^{\varepsilon_1,\varepsilon_2}) $ as mentioned above and  $ z(({E_8}^C)^{\varepsilon_1,\varepsilon_2})=\{1,\gamma \} $ (Proposition \ref{proposition 8.19}), we see $ \gamma \in z((E_8)^{\varepsilon_1,\varepsilon_2}) $. This is contradiction, so that this case is impossible.

Therefore we have the required isomorphism
\begin{align*}
(E_8)^{\varepsilon_1,\varepsilon_2} \cong E_7.
\end{align*}
\end{proof}

We consider a subgroup $ ((F_4)^{\varepsilon_1,\varepsilon_2})_{E_1,E_2,E_3} $ of the group $ (F_4)^{\varepsilon_1,\varepsilon_2} $:
\begin{align*}
((F_4)^{\varepsilon_1,\varepsilon_2})_{E_1,E_2,E_3}=\left\lbrace \alpha \in (F_4)^{\varepsilon_1,\varepsilon_2} \relmiddle{|} \alpha E_1=E_1, \alpha E_2=E_2, \alpha E_3=E_3 \right\rbrace .
\end{align*}

Then, in order to prove  the main theorem, we need the following proposition.

\begin{proposition}\label{proposition 8.23}
The group $ ((F_4)^{\varepsilon_1,\varepsilon_2})_{E_1,E_2,E_3} $ is isomorphic to the group $ Sp(1)\times Sp(1)\times Sp(1) ${\rm :} $ ((F_4)^{\varepsilon_1,\varepsilon_2})_{E_1,E_2,E_3} \cong Sp(1)\times Sp(1)\times Sp(1) $.
\end{proposition}
\begin{proof}
We define a mapping $ \varphi_{4,\varepsilon_i}: Sp(1)\times Sp(1)\times Sp(1) \to ((F_4)^{\varepsilon_1,\varepsilon_2})_{E_1,E_2,E_3} $ by
\begin{align*}
\varphi_{4,\varepsilon_i}(p_1,p_2,p_3)(M+\a)=A_3M{A_3}^*+\a {A_3}^*,\;\;M+\a \in \mathfrak{J}_{\sH}\oplus \H^3=\mathfrak{J},
\end{align*}
where $ A_3:=\diag(p_1,p_2,p_3) \in Sp(3) $. Note that $ \varphi_{4,\varepsilon_i} $ is the restriction of the mapping $ \varphi_4:Sp(1)\times Sp(3) \to (F_4)^\gamma $ (\cite[Theorem 2.11.2]{iy0}), that is, $ \varphi_{4,\varepsilon_i}(p_1,p_2,p_3)=\varphi_4(1,A_3) $.

First, we will prove that $ \varphi_{4,\varepsilon_i} $ is well-defined and a homomorphism. However, since $ \varphi_{4,\varepsilon_i} $ is the restriction of the mapping $ \varphi_4 $, it is clear.

Next, we will prove that $ \varphi_{4,\varepsilon_i} $ is surjective. Let $ \alpha \in ((F_4)^{\varepsilon_1,\varepsilon_2})_{E_1,E_2,E_3} $. Then, since $ {\varepsilon_i}^2=\gamma,i=1,2 $, we have $ \alpha \in ((F_4)^\gamma)_{E_1,E_2,E_3} \subset (F_4)^\gamma $. Hence there exist $ p \in Sp(1) $ and $ A \in Sp(3) $ such that $ \alpha=\allowbreak \varphi_4(p,A) $ (\cite[Theorem 2.11.2]{iy0}). First, it follows from $ \alpha E_i=E_i,i=1,2,3 $ that
\begin{align*}
E_i=\alpha E_i=\varphi_4(p,A)E_i=AE_iA^*,
\end{align*}
that is, $ E_i=AE_iA^*,i=1,2,3 $. Set $ A:=
\begin{pmatrix}
\; a_{11}\; & \; a_{12}\; & \; a_{13} \; \\
a_{21} & a_{22} & a_{23} \\
a_{31} & a_{32} & a_{33}
\end{pmatrix} \in Sp(3)$, then we have the following
\begin{itemize}
\setlength{\leftskip}{30mm}
\item [] $ a_{11}\ov{a_{11}}=1, a_{21}\ov{a_{21}}=a_{31}\ov{a_{31}}=0 $ ($ E_1=AE_1A^* $),
\item [] $ a_{22}\ov{a_{22}}=1, a_{12}\ov{a_{12}}=a_{32}\ov{a_{32}}=0 $ ($ E_2=AE_2A^* $),
\item [] $ a_{33}\ov{a_{33}}=1, a_{13}\ov{a_{13}}=a_{23}\ov{a_{23}}=0 $ ($ E_3=AE_3A^* $).
\end{itemize}
Hence we see $ a_{11}, a_{22},a_{33} \in Sp(1) $ and $ a_{ij}=0,i\not=j $, so that $ A $ is of the form
\begin{align*}
\diag(p_1,p_2,p_3)(=A_3),\;\,p_i \in Sp(1),i=1,2,3.
\end{align*}
Moreover, it follows from $ \varepsilon_i\alpha=\alpha\varepsilon_i, i=1,2 $ that
\begin{align*}
\varepsilon_i\alpha{\varepsilon_i}^{-1}(M+\a)
&=\varepsilon_i\varphi_4(p,A_3)(M+(-e_i)\a)=\varepsilon_i(A_3M{A_3}^*+p(-e_i\a){A_3}^*)
\\
&=A_3M{A_3}^*+e_i(p(-e_i\a){A_3}^*)=A_3M{A_3}^*+(-e_ipe_i)\a {A_3}^*
\\
&=\varphi_4(-e_ipe_i,A_3)(M+\a),\,\,M+\a \in \mathfrak{J}_{\sH}\oplus \H^3=\mathfrak{J},
\end{align*}
that is, $ \varphi_4(-e_ipe_i,A_3)=\varphi_4(p,A_3) $.

Hence we have the following
\begin{align*}
\left\lbrace
\begin{array}{l}
-e_ipe_i=p \\
A_3=A_3
\end{array}\right. \qquad \text{ or } \qquad
\left\lbrace
\begin{array}{l}
-e_ipe_i=-p \\
A_3=-A_3.
\end{array}\right.
\end{align*}
The latter case is impossible because of $ A_3\not=\O $ ($ \O $ is zero matrix). In the former case, from the first condition, we have $ p \in \R $, and together with $ p \in Sp(1) $, we see
\begin{align*}
p=1 \quad \text{ or }\quad p=-1.
\end{align*}
Thus we have
\begin{align*}
\alpha=\varphi_4(1,A_3) \quad \text{ or } \quad \alpha=\varphi_4(-1,A_3).
\end{align*}
However, as for $ \alpha=\varphi_4(-1,A_3) $, it follows from $ \varphi_4(-1,-E)=1 $ that
\begin{align*}
\alpha=\varphi_4(-1,A_3)\cdot 1=\varphi_4(-1,A_3)\varphi_4(-1,-E)=\varphi_4(1,-A_3),\;\;-A_3 \in Sp(3).
\end{align*}
With above, there exist $ A_3=\diag(p_1,p_2,p_3),p_i \in Sp(1),i=1,2,3 $ such that
\begin{align*}
\alpha=\varphi_4(1,A_3)=\varphi_{4,\varepsilon_i}(p_1,p_2,p_3).
\end{align*}
The proof of surjective is completed.

Finally, we will determine $ \Ker\,\varphi_{4,\varepsilon_i} $. Since $ \varphi_{4,\varepsilon_i} $ is the restriction of the mapping $ \varphi_4 $, it is clear $ \Ker\,\varphi_{4,\varepsilon_i}=\{E\} $.

Therefore we have the required isomorphism
\begin{align*}
((F_4)^{\varepsilon_1,\varepsilon_2})_{E_1,E_2,E_3} \cong Sp(1)\times Sp(1)\times Sp(1).
\end{align*}
\end{proof}

Now, we prove the main theorem in this section.

\begin{theorem}\label{theorem 8.24}
The group $ E_{8,\sH} $ is isomorphic to the group $ E_7/\Z_2, \Z_2=\{1,\gamma \}${\rm :} $ E_{8,\sH} \cong E_7/\Z_2 $.
\end{theorem}
\begin{proof}
Let $ E_7 $ as the group $ (E_8)^{\varepsilon_1,\varepsilon_2} $ (Theorem \ref{theorem 8.22}). then we define a mapping $ g_{{}_8}:(E_8)^{\varepsilon_1,\varepsilon_2} \allowbreak \to E_{8,\sH}$ by
\begin{align*}
g_{{}_8}(\alpha)=\alpha\big|_{(\mathfrak{e}_{8,\sH})^C}.
\end{align*}

First, we will prove that $ g_{{}_8} $ is well-defined. Let $ \alpha \in (E_8)^{\varepsilon_1,\varepsilon_2} $, then we have $ \alpha R \in (\mathfrak{e}_{8,\sH})^C $ for $ R \in (\mathfrak{e}_{8,\sH})^C $. Indeed, any elements $ R \in (\mathfrak{e}_{8,\sH})^C $ satisfies $ \gamma R=R $. Hence it follows from $ (E_8)^{\varepsilon_1,\varepsilon_2} \subset (E_8)^\gamma $ that
\begin{align*}
\alpha R=\alpha (\gamma R)=\gamma (\alpha R),\;\; R \in (\mathfrak{e}_{8,\sH})^C,
\end{align*}
that is, $ \alpha R \in (\mathfrak{e}_{8,\sH})^C $. Thus $ \alpha $ induces a $ C $-linear isomorphism of $ (\mathfrak{e}_{8,\sH})^C $, so that we have $ g_{{}_8}(\alpha) \in  E_{8,\sH} $, that is, $ g_{{}_8} $ is well-defined. It is clear that $ g_{{}_8} $ is a homomorphism.

Next we will determine $ \Ker\,g_{{}_8}$. Let $ \alpha \in \Ker\,g_{{}_8} $. For $ 1_-=(0,0,0,0,0,1) \in (\mathfrak{e}_{8,\sH})^C $, we have
\begin{align*}
\alpha 1_-=1_-.
\end{align*}
Hence we see $ \alpha \in ((E_8)^{\varepsilon_1,\varepsilon_2})_{1_-}=((E_8)_{1_-})^{\varepsilon_1,\varepsilon_2} \cong (E_7)^{\varepsilon_1,\varepsilon_2} $ (Subsection 2.5). Moreover, for $ {R^-}_{\dot{1}}:=(0,\dot{1},0,0,\allowbreak 0,0) \in (\mathfrak{e}_{8,\sH})^C, \dot{1}=(0,0,1,0) \in \mathfrak{P}^C $, we have
\begin{align*}
\alpha \dot{1}=\dot{1}.
\end{align*}
Hence we see $ \alpha \in ((E_7)^{\varepsilon_1,\varepsilon_2})_{\dot{1}} \subset ((E_7)_{\dot{1}})^{\varepsilon_1,\varepsilon_2} \cong (E_6)^{\varepsilon_1,\varepsilon_2} $ (Subsection 2.4). In addition, for $ {R^-}_{\dot{E}}:=(0,\dot{E},0,0,\allowbreak 0,0) \in (\mathfrak{e}_{8,\sH})^C, \dot{E}=(E,0,0,0) \in \mathfrak{P}^C $, we have
\begin{align*}
\alpha E= E.
\end{align*}
Hence we see $ \alpha \in ((E_6)^{\varepsilon_1,\varepsilon_2})_{E} \subset ((E_6)_{E})^{\varepsilon_1,\varepsilon_2} \cong (F_4)^{\varepsilon_1,\varepsilon_2} $ (Subsection 2.3). Moreover, for $ {R^-}_{\dot{E_i}}:=(0,\dot{E_i},0,0,\allowbreak 0,0) \in (\mathfrak{e}_{8,\sH})^C, \dot{E_i}=(E_i,0,0,0) \in \mathfrak{P}^C, i=1,2,3 $, we have
\begin{align*}
\alpha E_i=E_i,\;\;i=1,2,3.
\end{align*}
Hence we see $ \alpha \in ((F_4)^{\varepsilon_1,\varepsilon_2})_{E_1,E_2,E_3} $, so that there exists $ (p_1,p_2,p_3) \in Sp(1)\times Sp(1)\times Sp(1) $ such that $ \alpha=\varphi_{4,\varepsilon_i}(p_1,p_2,p_3) $ (Proposition \ref{proposition 8.23}).

\noindent Subsequently, for $ {R^-}_{\dot{F}_k(1)}=(0,\dot{F}_k(1),0,0,0,0) \in (\mathfrak{e}_{8,\sH})^C,\dot{F}_k(1)=(F_k(1),0,0,0),k=1,2,3 $, it follows from $ \alpha {R^-}_{\dot{F}_k(1)}={R^-}_{\dot{F}_k(1)} $ that
\begin{align*}
{R^-}_{\dot{F}_k(1)}&=\alpha {R^-}_{\dot{F}_k(1)}=\varphi_{4,\varepsilon_i}(p_1,p_2,p_3){R^-}_{\dot{F}_k(1)}=(0,\varphi_{4,\varepsilon_i}(p_1,p_2,p_3)\dot{F}_k(1),0,0,0,0)
\\
&=(0,\diag(p_1,p_2,p_3)\dot{F}_k(1)\diag(p_1,p_2,p_3)^*,0,0,0,0).
\end{align*}
Hence we have $ p_2\ov{p_3}=p_3\ov{p_1}=p_1\ov{p_2}=1 $, and together with $ p_i \in Sp(1) $, we see $ p_1=p_2=p_3:=p $. Thus $ \alpha $ is of the form
\begin{align*}
\alpha=\varphi_{4,\varepsilon_i}(p,p,p), \;\; p \in Sp(1).
\end{align*}
Moreover, for $ {R^-}_{\dot{F}_1(e_i)}=(0,\dot{F}_1(e_i),0,0,0,0) \in (\mathfrak{e}_{8,\sH})^C,\dot{F}_1(e_i)=(F_1(e_i),0,0,0),i=1,2 $, it follows from $ \alpha {R^-}_{\dot{F}_1(e_i)}={R^-}_{\dot{F}_1(e_i)} $ that
\begin{align*}
{R^-}_{\dot{F}_1(e_i)}&=\alpha {R^-}_{\dot{F}_1(e_i)}=\varphi_{4,\varepsilon_i}(p,p,p){R^-}_{\dot{F}_1(e_i)}=(0,\varphi_{4,\varepsilon_i}(p,p,p)\dot{F}_1(e_i),0,0,0,0)
\\
&=(0,\diag(p,p,p)\dot{F}_1(e_i)\diag(p,p,p)^*,0,0,0,0).
\end{align*}
Hence we have $ pe_i\ov{p}=1,i=1,2 $, and together with $ p \in Sp(1) $, we see $ p^2=1, p \in \R $, that is, $ p=1 $ or $ p=-1 $, so that we have $ \alpha=1 $ when $ p=1 $ or $ \alpha=\gamma $ when $ p=-1 $. With above, we see $ \Ker\, g_{{}_8} \subset \{1,\gamma \} $
and vice versa. Thus we obtain
\begin{align*}
\Ker\, g_{{}_8} = \{1,\gamma \} \cong \Z_2.
\end{align*}

Finally, we will prove that $ g_{{}_8} $ is surjective. Since the group $ E_{8,\sH} $ is connected (Theorem \ref{theorem 8.10}) and $ \Ker\, g_{{}_8}$ is discrete, and together with $ \dim((\mathfrak{e}_8)^{\varepsilon_1,\varepsilon_2}) =133=\dim(\mathfrak{e}_{8,\sH})$ (Lemmas \ref{lemma 8.2}, \ref{lemma 8.21}), $ g_{{}_8} $ is surjective.

Therefore, from Theorem \ref{theorem 8.22}, we have the required isomorphism
\begin{align*}
E_{8,\sH} \cong E_7/\Z_2.
\end{align*}
\end{proof}



\begin{thebibliography}{99}
\bibitem{che} \textsc{C. Chevalley},
\newblock Theory of Lie groups I, Princeton Univ. Press, 1946.







\bibitem{miya1}\textsc{T. Miyashita and I. Yokota},
\newblock $ 3 $-graded decompositions of exceptional Lie algebra $
\mathfrak{g} $ and group realizations of $ \mathfrak{g}_{ev}, \mathfrak{g}_0 $ and $ \mathfrak{g}_{ed} $ Part II, $ G=E_7 $, Case 1, J. Math. Kyoto Univ., 46-2(2006), 383-413.




\bibitem{iy10} \textsc{I. Yokota},
\newblock Non-symmetry Freudenthal's magic square, J. Fac. Sci. Shinshu Univ. 20(1985), 13-13.


\bibitem{iy7} \textsc{I. Yokota},
\newblock Realizations of involutive automorphisms $\sigma$ and $G^\sigma$ of exceptional linear Lie groups $G$, Part I, $G = G_2, F_4$, and $E_6$, Tsukuba J. Math. 14(1990), 185-223.

\bibitem{iy2} \textsc{I. Yokota},
\newblock Realizations of involutive automorphisms $\sigma$ and $G^\sigma$ of exceptional linear Lie groups $G$, Part II, $G = E_7$, Tsukuba J. Math. 14(1990), 378-404.

\bibitem{iy1}\textsc{I. Yokota},
\newblock Realizations of involutive automorphisms $\sigma$ and $G^\sigma$ of exceptional linear Lie groups $G$, Part III, $G = E_8$, Tsukuba J. Math. 15(1991), 301-314.

\bibitem{iy4}\textsc{I. Yokota},
\newblock $ 2 $-graded decompositions of exceptional Lie algebra $ \mathfrak{g} $ and group realizations of $ \mathfrak{g}_{ev}, \mathfrak{g}_0 $ Part I, $ G=G_2,F_4,E_6 $, Japanese J. Math., 24-2(1998), 275-296.

\bibitem{iy3}\textsc{I. Yokota},
\newblock $ 3 $-graded decompositions of exceptional Lie algebra $ \mathfrak{g} $ and group realizations of $ \mathfrak{g}_{ev}, \mathfrak{g}_0 $ and $ \mathfrak{g}_{ed} $ Part I, $ G=G_2,F_4,E_6 $, J. Math. Kyoto univ., 41-3(2001), 449-474.

\bibitem{iy0} \textsc{I. Yokota},
\newblock Exceptional Lie groups, arXiv:math/0902.0431vl[mathDG](2009).


\bibitem{iy11} \textsc{I. Yokota},
\newblock Exceptional linear simple Lie groups (in Japaneae), Gendai-sugakusya, Kyoto, 1992.
\end{thebibliography}
\end{document}